\documentclass[a4paper, 11pt]{article}

\pdfoutput=1
\usepackage{fullpage}

\RequirePackage[OT1]{fontenc}
\RequirePackage{amsthm,amsmath}
\RequirePackage[numbers]{natbib}
\usepackage{chapterbib}
\RequirePackage[colorlinks,citecolor=blue,urlcolor=blue]{hyperref}

\usepackage{amsfonts}
\usepackage{algorithmic}
\usepackage{algorithm}
\usepackage{amssymb, bbm}
\newboolean{ARXIVVERSION}
\setboolean{ARXIVVERSION}{true}
\usepackage[normalem]{ulem}

\usepackage{graphicx}
\usepackage{url}
\usepackage{tikz}
\usetikzlibrary{arrows}
\usepackage{paralist}
\newdimen\arrowsize
\pgfarrowsdeclare{arcsq}{arcsq}
{
  \arrowsize=0.2pt
  \advance\arrowsize by .5\pgflinewidth
  \pgfarrowsleftextend{-4\arrowsize-.5\pgflinewidth}
  \pgfarrowsrightextend{.5\pgflinewidth}
}
{
  \arrowsize=1.5pt
  \advance\arrowsize by .5\pgflinewidth
  \pgfsetdash{}{0pt} %
  \pgfsetroundjoin   %
  \pgfsetroundcap    %
  \pgfpathmoveto{\pgfpoint{0\arrowsize}{0\arrowsize}}
  \pgfpatharc{-90}{-140}{4\arrowsize}
  \pgfusepathqstroke
  \pgfpathmoveto{\pgfpointorigin}
  \pgfpatharc{90}{140}{4\arrowsize}
  \pgfusepathqstroke
}

\newcommand{\independent}{\mbox{${}\perp\mkern-11mu\perp{}$}}
\newcommand{\notindependent}{\mbox{${}\not\!\perp\mkern-11mu\perp{}$}}

\DeclareMathOperator*{\argmin}{argmin}

\newcommand{\B}[1]{\mathbf{#1}}
\newcommand{\checknew}{\mathring}
\newcommand{\mb}{\mathbf}
\newcommand{\mbb}{\boldsymbol}
\newcommand{\prob}{{\mathbb P}}
\newcommand{\pr}{{\mathbb P}}
\newcommand{\R}{{\mathbb R}}

\newcommand{\E}{{\mathbb E}}

\newcommand{\var}{{\mathbf{var}}}
\newcommand{\cov}{{\mathbf{cov}}}  
\newcommand{\tr}{{\mathbf{tr}}}

\newcommand{\given}{\mid}
\newcommand{\abs}[1]{\left|#1\right|}

\newcommand{\indist}{\stackrel{d}{\to}}

\newcommand{\inprob}{\stackrel{P}{\to}}

\newcommand{\floor}[1]{\left\lfloor #1 \right\rfloor}

\newcommand{\ind}{\mathbbm{1}}
\newcommand{\tv}[1]{\left\| #1 \right\|_{\text{TV}}}
\newtheorem{theorem}{Theorem}
\newtheorem{example}[theorem]{Example}
\newtheorem{lemma}[theorem]{Lemma}
\newtheorem{corollary}[theorem]{Corollary}

\newtheorem{proposition}[theorem]{Proposition}
\newtheorem{remark}[theorem]{Remark}

\tikzset{every picture/.style={line width=0.6pt}}
\tikzset{every picture/.style={outer sep=.4mm}}
\tikzstyle{graphnode} = 
   [circle,draw=black,minimum size=25pt,text centered,inner sep=0.2mm] 
\tikzstyle{observed}   =[graphnode,fill=white,text=black]
\tikzstyle{unobserved}   =[graphnode,fill=white,text=black,style=dashed]
\tikzstyle{graphnodesmall} = 
   [circle,draw=black,minimum size=20pt,text centered,inner sep=0.2mm] 
\tikzstyle{observedsmall}   =[graphnodesmall,fill=white,text=black]
\tikzstyle{unobservedsmall}   =[graphnodesmall,fill=white,text=black,style=dashed]
\tikzstyle{graphnodestiny} = 
   [circle,draw=black,minimum size=15pt,text centered,inner sep=0.2mm] 
\tikzstyle{observedtiny}   =[graphnodestiny,fill=white,text=black]
\tikzstyle{unobservedtiny}   =[graphnodestiny,fill=white,text=black,style=dashed]

\title{The Hardness of Conditional Independence Testing and the Generalised Covariance Measure}%

\author{Rajen D. Shah\thanks{RDS was supported by The Alan Turing Institute under the EPSRC grant EP/N510129/1, an EPSRC Programme Grant and an EPSRC First Grant.}\\
University of Cambridge, UK\\
\url{r.shah@statslab.cam.ac.uk}
\and
Jonas Peters\thanks{JP was supported by  a research grant (18968) from VILLUM FONDEN and the Alan Turing Institute supported JP's research visit in London (July 2017).}\\
University of Copenhagen, Denmark\\
\url{jonas.peters@math.ku.dk}}
\date{\today}

\begin{document}
\maketitle

\begin{abstract}
It is a common saying that testing for conditional independence, i.e., testing whether whether two random vectors $X$ and $Y$ are independent, given $Z$, is a hard statistical problem if $Z$ is a continuous random variable (or vector). In this paper, we prove that conditional independence is indeed a particularly difficult hypothesis to test for. Valid statistical tests are required to have a size that is smaller than a predefined significance level, and different tests usually have power against a different class of alternatives. We prove that a valid test for conditional independence does not have power against \emph{any} alternative. 

Given the non-existence of a uniformly valid conditional independence test, we argue that tests must be designed so their suitability for a particular problem  may be judged easily. To address this need, we propose in the case where $X$ and $Y$ are univariate to nonlinearly regress $X$ on $Z$, and $Y$ on $Z$ and then compute a test statistic based on the sample covariance between the residuals, which we call the generalised covariance measure (GCM). We prove that validity of this form of test relies almost entirely on the weak requirement that the regression procedures are able to estimate the conditional means $X$ given $Z$, and $Y$ given $Z$, at a slow rate. We extend the methodology to handle settings where $X$ and $Y$ may be multivariate or even high-dimensional. While our general procedure can be tailored to the setting at hand by combining it with any regression technique, we develop the theoretical guarantees for kernel ridge regression.
A simulation study shows that the test based on GCM is competitive with state of the art conditional independence tests. 
Code is available as the R package \texttt{GeneralisedCovarianceMeasure} on CRAN.
\end{abstract}

\begin{cbunit}

\section{Introduction}
Conditional independences lie at the heart of several fundamental concepts such as sufficiency \citep{Fisher20} and ancillarity \citep{Fisher34, Fisher35}; see also 
\citet{Sorensen2017}. \citet{Dawid79} states that ``many results and
theorems concerning these concepts are just applications of some simple general properties
of conditional independence''. 
During the last few decades,
conditional independence relations have played an 
increasingly important  
role in 
computational statistics, too, since they are the building blocks of 
graphical models \citep{Lauritzen1996, Koller2009, Pearl2009}.

Estimating conditional independence graphs has been of great interest in high-dimensional statistics, particularly in biomedical applications \citep[e.g.][]{markowetz2007inferring, Dobra2004}. 
To estimate the conditional independence graph corresponding to a random vector $X \in \R^d$, edges may be placed between vertices corresponding to $X_j$ and $X_k$ if a test for whether $X_j$ is conditionally independent of $X_k$ given all other variables indicates rejection.

Conditional independence tests also play a key role in causal inference.
Constraint-based or independence-based methods \citep{Pearl2009, Spirtes2000, Peters2017book} apply a series of conditional independence tests in order to learn causal structure from observational data.
The recently introduced invariant prediction methodology \citep{Peters2016jrssb, Heinze2017} aims to estimate for a specified target variable $Y$, the set of causal variables among potential covariates $X_1,\ldots, X_d$. Given data from different environments labelled by a variable $E$, the method involves testing for each subset $S \subseteq \{1, \ldots, d\}$, the null hypothesis that the environment $E$ is conditionally independent of $Y$, given $X_S$.

Given the importance of conditional independence tests in modern statistics, there has been great deal of work devoted to developing conditional independence tests; we review some important examples of tests in Section~\ref{sec:related_work}. However one issue that conditional independence tests appear to suffer from is that they can fail to control the type I error in finite samples, which can have important consequences in downstream analyses.

In the first part of this paper, we prove this failure of type I error control is in fact unavoidable: conditional independence is not a testable hypotheses.
To fix ideas, consider $n$ i.i.d.\ observations corresponding to a triple of random variables $(X, Y, Z)$ where it is desired to test whether $X$ is conditional independent of $Y$ given $Z$.
We show that provided the joint distribution of $(X, Y, Z) \in \R^{d_X + d_Y + d_Z}$ is absolutely continuous with respect to Lebesgue measure, any test based on the data whose size is less than a pre-specified level $\alpha$, has no power; more precisely, there is no alternative, for which the test has power more than $\alpha$. 
This result is perhaps surprising as it is in stark contrast to unconditional independence testing, for which permutation tests allow for the correct calibration of any hypothesis test.
Our result implies that in order to perform conditional independence testing, some domain knowledge is required in order to select an appropriate conditional independence test for the particular problem at hand.  This would appear to be challenging in practice, as the validity of conditional independence tests typically rests on properties of the entire joint distribution of the data, which may be hard to model. 

Our second main contribution aims to alleviate this issue by providing a family of conditional independence tests whose validity relies on estimating the conditional expectations $\E(X |Z=z)$ and $\E(Y|Z=z)$ via regressions, in the setting where $d_X=d_Y=1$. These need to be estimated sufficiently well using the data such that the product of mean squared prediction errors from the two regressions is $o(n^{-1})$. This is a relatively mild requirement that allows for settings where the conditional expectations are as general as Lipschitz functions, for example, and also encompasses settings where $Z$ is high-dimensional but the conditional expectations have more structure.

Our test statistic, which we call the \emph{generalised covariance measure} (GCM) is based on a suitably normalised version of the empirical covariance between the residual vectors from the regressions. The practitioner is free to choose the regression methods that appear most suitable for the 
problem of interest.
Although domain knowledge is still required to make an appropriate choice, selection of regression methods is a problem statisticians are more familiar with.
We also extend the GCM to handle settings where $X$ and $Y$ are potentially high-dimensional, though in this case our proof of the validity of the test additionally requires the errors $X_j - \E(X_j|Z)$ and $Y_k - \E(Y_k|Z)$ to obey certain moment restrictions for $j=1,\ldots,d_X$ and $k=1,\ldots,d_Y$ and slightly faster rates of convergence for the prediction errors.

As an example application of our results on the GCM, we consider the case where the regressions are performed using kernel ridge regression, and show that provided the conditional expectations are contained in a reproducing kernel Hilbert space, our test statistic has a tractable limit distribution.

The rest of the paper is organised as follows. In Sections~\ref{sec:condind} and \ref{sec:testing}, we first formalise the notion of conditional independence and relevant concepts related to statistical hypothesis testing. In Section~\ref{sec:related_work} we review some popular conditional independence tests,
after which we set out some notation used throughout the paper. In Section~\ref{SEC:NOFREELUNCH} we present our main result on the hardness of conditional independence testing. We introduce the generalised covariance measure in Section~\ref{SEC:GCM} first treating the univariate case with $d_X=d_Y=1$ before extending ideas to the potentially high-dimensional case. In Section~\ref{sec:kernel} we apply the theory and methodology of the previous section to study that particular example of generalised covariance measures based on kernel ridge regression. We present numerical experiments in Section~\ref{sec:experiments} and conclude with a discussion in Section~\ref{sec:discussion}. All proofs are deferred to the appendix and supplementary material.

\subsection{Conditional independence} \label{sec:condind}
Let us consider three random vectors $X$, $Y$ and $Z$ taking values in $\R^{d_X}$, $\R^{d_Y}$ and $\R^{d_Z}$, respectively, and let us assume, for now that their joint distribution is absolutely continuous with respect to Lebesgue measure with density $p$. 
For our deliberations only the continuity in $Z$ is
necessary,
see Remark~\ref{rem:continuity}.
We say that $X$ is conditionally independent of $Y$ given $Z$ and write 
$$
X \independent Y \given Z
$$
if for all
$
x, y, z$ with $p(z) > 0$, we have $p(x,y|z) = p(x|z) p(y|z)$,
see, e.g., \citet{Dawid79}. 
Here and below, statements involving densities should be understood to hold (Lebesgue) almost everywhere.
We now discuss an equivalent formulation of conditional independence that has given rise to several hypothesis tests, including the generalised covariance measure proposed in this paper.
Let therefore 
$
L^2_{X,Z}
$ denote the space of all functions $f:\R^{d_X} \times \R^{d_Z} \rightarrow \R$ such that $\E f(X,Z)^2 < \infty$ and define $L^2_{Y,Z}$ analogously.
\citet{Daudin1980} proves that $X$ and $Y$ are conditionally independent given $Z$ if and only if 
\begin{equation} \label{eq:Daudin}
\E f(X,Z) g(Y,Z) = 0
\end{equation}
for all 
functions 
$f \in L^2_{X,Z}$ and 
$g \in L^2_{Y,Z}$ such that 
$\E[f(X,Z)|Z] = 0$ and 
$\E[g(Y,Z)|Z] = 0$, respectively.

This may be viewed as an extension of the fact that 
 for one-dimensional $X$ and $Y$, the partial correlation coefficient 
$\rho_{X,Y|Z}$ (the correlation between residuals of linear regressions of $X$ on $Z$ and $Y$ on $Z$) is 0 if and only if $X \independent Y \given Z$ in the case where $(X, Y, Z)$ are jointly Gaussian.

\subsection{Statistical hypothesis testing 
and notation} \label{sec:testing}
We now introduce some notation and relevant concepts related to statistical hypothesis testing.
In order to deal with composite null hypotheses where the probability of rejection must be controlled under a variety of different distributions for the data to which our test is applied, we introduce the following notation. We will write $\E_P(\cdot)$ for expectations of random variables whose distribution is determined by $P$, and similarly $\pr_P(\cdot) = \E_P \ind_{\{\cdot\}}$.

Let $\mathcal{P}$ be a potentially composite null hypothesis consisting of a collection of distributions for $(X, Y, Z)$.
For $i=1,2, \ldots$ let $(x_i, y_i, z_i) \in \R^{d_X + d_Y + d_Z}$ be i.i.d.\ copies of $(X, Y, Z)$ and let $\mb X^{(n)} \in \R^{d_X \cdot n}$, $\mb Y^{(n)} \in \R^{d_Y \cdot n}$ and $\mb Z^{(n)} \in \R^{d_Z \cdot n}$ be matrices with $i$th rows $x_i$, $y_i$ and $z_i$ respectively. Let $\psi_n$ be a potentially randomised test
that can be applied to the
data $(\mb X^{(n)}, \mb Y^{(n)}, \mb Z^{(n)})$; formally,
\begin{equation*}
\psi_n: \R^{(d_X + d_Y + d_Z)\cdot n} \times [0,1] \rightarrow \{0,1\}
\end{equation*}
is a measurable function whose last argument is reserved for a random variable $U \sim U[0, 1]$ independent of the data which is responsible for the randomness of the test.

Given a sequence of tests $(\psi_n)_{n=1}^\infty$, the following validity properties will be of interest; %
note the particular names given to these properties differ in literature.
Given a level $\alpha \in (0, 1)$ and null hypothesis $\mathcal{P}$, we say that the test $\psi_n$ has 
$$
\text{\emph{valid level at sample size $n$} if }
\quad
\sup_{P \in \mathcal{P}}
\mathbb{P}_P(\psi_n = 1) \leq \alpha,
$$
where the left-hand side is the \emph{size} of the test; the sequence $(\psi_n)_{n=1}^\infty$ has
\begin{gather*}
\text{ \emph{uniformly asymptotic level} if }
\quad
\limsup_{n \rightarrow \infty} 
\sup_{P \in \mathcal{P}}
\mathbb{P}_P(\psi_n = 1) \leq \alpha,\\
\text{\emph{pointwise asymptotic level} if }
\quad
\sup_{P \in \mathcal{P}}
\limsup_{n \rightarrow \infty} 
\mathbb{P}_P(\psi_n = 1) \leq \alpha.
\end{gather*}
In practice, we would like a test to have at least uniformly asymptotic level. 
Otherwise, 
even for an arbitrarily large sample size $n$,
there can exist null distributions for which
the size exceeds the nominal 
level by some fixed amount.

Given a sequence of tests $(\psi_n)_{n=1}^\infty$ each with valid level $\alpha \in (0, 1)$ and alternative hypotheses $\mathcal{Q}$, it is desirable for the power to be large uniformly over $\mathcal{Q}$, and to have
$\inf_{Q \in \mathcal{Q}} \pr_Q(\psi_n=1) \to 1$. In standard parametric settings, we can certainly achieve this for any fixed alternative hypothesis and indeed have uniform power against a sequence of $\sqrt{n}$ alternatives.
Nonparametric problems are much harder and when $\mathcal{Q}$ contains all distributions outside a small fixed total variation (TV) neighbourhood of $\mathcal{P}$, we have
\[
\liminf_{n \to \infty} \inf_{Q \in \mathcal{Q}} \pr_Q(\psi_n=1) < 1,
\]
\citep[Prop.~2, Thm.~3]{lecam1973convergence, barron1989uniformly}.
To achieve power tending to 1, we need to restrict $\mathcal{Q}$ by imposing certain smoothness conditions for example
 \citep{balakrishnan2017hypothesis}. 

A class of even harder hypothesis testing problems may be defined as those where no test with valid level achieves power at any alternative, so $\sup_{Q \in \mathcal{Q}} \pr_Q(\psi_n=1) \leq \alpha$. In other words,  for all $n$, tests $\psi_n$ and alternative distributions $Q \in \mathcal{Q}$, we have
\[
\pr_Q (\psi_n =1) \leq \sup_{P \in \mathcal{P}} \pr_P(\psi_n=1).
\]
The hypothesis testing problem defined by the pair $(\mathcal{P}, \mathcal{Q})$ is then said to be \emph{untestable} \citep{dufour2003identification}.
In order to have power at even a single alternative, we need to restrict the null $\mathcal{P}$ in some way.
One of the main results of this paper is that conditional independence is untestable.

\subsection{Related work} \label{sec:related_work}

Our hardness result for conditional independence contributes to an important literature on impossibility results in statistics and econometrics starting with the work of \citet{bahadur1956} which shows that there is no non-trivial test for whether a distribution has mean zero. \citet{canay2013} shows that certain problems arising in the context
of identification of some nonparametric models are not testable.
In these examples, the null hypothesis is dense with respect to the TV metric in the alternative hypothesis, a property which implies untestability \citep{romano2004non}. Interestingly, our Proposition~\ref{prop:separation} shows that conditional independence testing is qualitatively different in that
some
distributions in the alternative are in fact well-separated from the null.
It has been suggested for some time that conditional independence testing is a hard problem (see, e.g., \citep{Bergsma2004}, and several talks given by Bernhard Sch\"olkopf). To the best of our knowledge the conjecture that conditional independence is not testable (cf.~Corollary~\ref{cor:nfl} with $M=\infty$) is due to Arthur Gretton.
We also note that when the conditional distribution of $X$ given $Z$ is known, conditional independence is testable \citep[e.g.][]{berrett2018conditional}. 

We now briefly review
several tests for conditional independence that bear some relation to our proposal here.
\paragraph{Extensions of partial correlation}
\citet{Ramsey2014}
suggests 
regressing
$X$ on $Z$ and $Y$ on $Z$ and then testing for independence between the residuals. 
\citet{Fan2015} consider this approach in the setting where $Z$ is potentially high-dimensional and under the null hypothesis of $X \independent Y \given Z$, $X=Z^T\beta_X + \varepsilon_X$, $Y=Z^T\beta_Y + \varepsilon_Y$ with $\varepsilon_X \independent Z$ and $\varepsilon_Y \independent Z$. 
The following simple example however indicates where such methods can fail. 
\begin{example} \label{ex:hetero}
Define $N_X, N_Y$ and $Z$ to be i.i.d.\ random variables with distribution 
$\mathcal{N}(0,1)$ and define $X:= Z\cdot N_X$, $Y:=Z\cdot N_Y$. This
implies $X \independent Y \given Z$. Since $\E[X|Z] = \E[Y|Z] = 0$, the (population) residuals equal 
$R_1 := Z\cdot N_X$
and
$R_2 := Z\cdot N_Y$;
they are uncorrelated
but %
not independent since, e.g.,
$\mathrm{cov}(R_1^2, R_2^2) \neq 0$.
Consider regressing $X$ on $Z$ and $Y$ on $Z$, and then testing for independence 
of the residuals.
If the regression method outputs the true conditional means and 
the independence test has power against the alternative
$\mathrm{cov}(R_1^2, R_2^2) \neq 0$, 
the method will falsely reject the null hypothesis of conditional independence with large probability.
\end{example}
\paragraph{Kernel-based conditional independence tests}
The Hilbert-Schmidt independence criterion (HSIC) equals the square of the Hilbert--Schmidt
norm of the cross-covariance operator, 
and is used in 
unconditional independence testing \citep{Gretton2008}.
\citet{Fukumizu2008} 
extend this idea to conditional independence testing. 
To construct a test for continuous variables $Z$,
their work requires clustering of the values of $Z$ and permuting $X$ and $Y$ values within the same cluster component. 
Another extension is proposed by \citet{Zhang2011uai}.
Their kernel conditional independence (KCI) test is stated to yield pointwise asymptotic level control.

\paragraph{Estimation of expected conditional covariance}
Though typically not thought of as conditional independence tests, there are several approaches to estimating the expected conditional covariance functional $\E\cov(X, Y |Z)$ in the semiparametric statistics literature \citep{robins2008higher, chernozhukov2017double, newey2018cross}. 
From \eqref{eq:Daudin} we see these may be used as conditional independence tests and indeed the GCM test we propose falls under this category.  We delay further discussion of such methods to Section~\ref{sec:semiparametric}.

\subsection{Notation} \label{sec:notation}
We now introduce some notation used throughout the paper.
If $(V_{P,n})_{n \in \mathbb{N}, P \in \mathcal{P}}$ is a family of sequences of random variables whose distributions are determined by $P \in \mathcal{P}$, we use $V_{P,n}=o_\mathcal{P}(1)$ and $V_{P,n} = O_\mathcal{P}(1)$ to mean respectively that for all $\epsilon > 0$,
\begin{gather*}
\sup_{P \in \mathcal{P}} \pr_P(|V_{P,n}| > \epsilon) \to 0, \;\; \text{ and} \\
\text{there exists } M>0 \text{ such that } \sup_{n \in \mathbb{N}} \sup_{P \in \mathcal{P}} \pr_P(|V_{P,n}| > M) < \epsilon.
\end{gather*}
If $(W_{P,n})_{n \in \mathbb{N}, P \in \mathcal{P}}$ is a further family of sequences of random variables, $V_{P,n}=o_\mathcal{P}(W_{P,n})$ and $V_{P,n}=O_\mathcal{P}(W_{P,n})$ mean $V_{P,n} = W_{P,n} R_{P,n}$ and respectively that $R_{P,n}=o_\mathcal{P}(1)$ and $R_{P,n}=O_\mathcal{P}(1)$.
If $\B{A}$ is a $c \times d$ matrix, then $\B{A}_j$ denotes the $j$th column of $\B{A}$, $j \in \{1, \ldots, d\}$.

\section{No-free-lunch in Conditional Independence Testing}\label{SEC:NOFREELUNCH}
In this section we show that, under certain conditions, no non-trivial test  for conditional independence with valid level exists. To state our result, we introduce the following subsets of $\mathcal{E}_0$ defined to be the set of all distributions for $(X, Y, Z)$ absolutely continuous with respect to Lebesgue measure.

Let $\mathcal{P}_0 \subset \mathcal{E}_0$ be the subset of distributions under which $X \independent Y \given Z$. Further, for any $M \in (0, \infty]$, let $\mathcal{E}_{0,M} \subseteq \mathcal{E}_0$ be the subset of all distributions with support contained strictly within an $\ell_\infty$ ball of radius $M$. Here we take $\mathcal{E}_{0,\infty} = \mathcal{E}_0$. We also define $\mathcal{Q}_0=\mathcal{E}_0 \setminus \mathcal{P}_0$ and set $\mathcal{P}_{0,M} = \mathcal{E}_{0,M} \cap \mathcal{P}_0$, and $\mathcal{Q}_{0,M} = \mathcal{E}_{0,M} \cap \mathcal{Q}_0$.
Consider the setup of Section~\ref{sec:testing} with null hypothesis $\mathcal{P} = \mathcal{P}_{0,M}$. Our first result shows that with this null hypothesis, any test $\psi_n$ with valid level at sample size $n$ has no power against any alternative.

\begin{theorem}[No-free-lunch] \label{THM:NFL}
Given any $n \in \mathbb{N}$,
$\alpha \in (0, 1)$,
$M \in (0,\infty]$, 
and any
potentially randomised test $\psi_n$ that has valid level $\alpha$ 
for the 
null hypothesis $\mathcal{P}_{0,M}$,
we have that
$\prob_Q(\psi_n=1) \leq \alpha$ for all $Q \in \mathcal{Q}_{0,M}$. 
Thus $\psi_n$ cannot have power against any alternative.
\end{theorem}
A proof is given in the appendix.
Note that taking $M$ to be finite ensures all the random vectors $(x_i, y_i, z_i)$ are bounded. Thus, for example, averages will converge in distribution to Gaussian limits uniformly over $\mathcal{P}_{0,M}$; however, as the result shows, this does not help in the construction of a non-trivial test for conditional independence.
An immediate corollary to Theorem~\ref{THM:NFL} is that there is no non-trivial test for conditional independence with uniformly asymptotic level.
\begin{corollary} \label{cor:nfl}
For all $M \in (0,\infty]$ and for any sequence $(\psi_n)_{n=1}^\infty$ of tests we have
\[
\sup_{Q \in \mathcal{Q}_{0,M}} 
\limsup_{n \to \infty} \pr_Q(\psi_n = 1) \leq  
\limsup_{n \to \infty} \sup_{P \in \mathcal{P}_{0,M}} \pr_P(\psi_n = 1).
\]
\end{corollary}
This result is in
stark contrast to unconditional independence testing, where a permutation test can always be used to control the size of any testing procedure.
As a consequence, there exist tests with valid level at sample size $n$ and non-trivial power. 
For example, \citet{Hoeffding1948} introduces a rank-based test in the case of univariate random variables and proves that it maintains uniformly asymptotic level and has asymptotic power against each fixed alternative. For the multivariate case,
\citet{Berrett2017} consider a test based on mutual information and prove level guarantees, as well as uniform power results against a wide class of alternatives.
Thus while independence testing remains a hard problem in that it is only possible to have uniform power against certain subsets of alternatives, this is different to conditional independence testing where we can only hope to control the \emph{size} uniformly over certain subsets of the \emph{null} hypothesis.

\begin{remark} \label{rem:continuity}
Inspection of the proof shows that Theorem~\ref{THM:NFL} also holds in the case where the variables $X$ and $Y$ have marginal distributions that are absolutely continuous with respect to  counting measure, for example. Theorem~\ref{THM:NFL} therefore contains an impossibility result for testing the equality of two conditional distributions (by taking $Y$ to be an indicator specifying the distribution).
The continuity of $Z$, however, is necessary. If $Z$ only takes values in $\{1,2\}$, for example, one can reduce the 
problem of conditional independence testing to unconditional independence testing by combining the tests for $X \independent Y \given Z = 1$ and 
$X \independent Y \given Z = 2$.
\end{remark}

The null hypothesis being dense with respect to TV distance among the alternative hypothesis is a sufficient condition for the problem to be untestable \citep{romano2004non}. 
Proposition~\ref{prop:separation}, proved in the supplementary material, illustrates
that this
is not the case here:
at least for $M \in (0,\infty)$, there exists an alternative, for which there is no distribution from the null that is arbitrarily close.
\begin{proposition} \label{prop:separation}
For $P, Q \in \mathcal{E}_0$, the total variation distance is given by
\[
\tv{P-Q} := \sup_{ A \in \mathcal{B}} |\pr_P((X, Y, Z) \in A) - \pr_Q((X, Y, Z) \in A)|,
\]
where $\mathcal{B}$ is the Borel $\sigma$-algebra on $\R^{d_X + d_Y + d_Z}$.
For each $M \in (0, \infty)$, there exists $Q \in \mathcal{Q}_{0,M}$ satisfying
\[
\inf_{P \in \mathcal{P}_{0,M}} \tv{P-Q} \geq 1/24.
\]
\end{proposition}
In Proposition~\ref{prop:KL} in the %
appendix, we also show that the null and alternative hypotheses are well-separated in the sense of KL divergence.
On the other hand, it is known that if a problem is untestable, the convex closure of the null must contain the alternative \citep[][Theorem~5 and Corollary~1, respectively]{Kraft55, Bertanha2018}.
The problem of conditional independence testing therefore has the interesting property of the null being separated from the alternative, but its convex hull is TV-dense in the alternative.

A practical implication of the negative result of Theorem~\ref{THM:NFL} is that domain knowledge is needed to select a conditional independence test appropriate for the data at hand.
However guessing the form of the entire joint distribution in order to apply a test with the appropriate type I error control seems challenging. In Section~\ref{SEC:GCM} we introduce a form of test that instead relies on selecting regression methods that have sufficiently low prediction error when regressing $\mb Y^{(n)}$ and $\mb X^{(n)}$ on $\mb Z^{(n)}$, thereby converting the problem of finding an appropriate test to the more familiar task of prediction. Before discussing this methodology, we first sketch some of the main ideas of the proof of Theorem~\ref{THM:NFL} below.

\subsection{Proof ideas of Theorem~\ref{THM:NFL}}
Consider the case where $d_X=d_Y=d_Z=1$ and where the test is required to be non-randomised.
First suppose that for $Q \in \mathcal{Q}_{0, M}$, we have a test with rejection region $R := \psi_n^{-1}(1) \subseteq \R^{3\cdot n}$ such that $\pr_Q((\mb X^{(n)}, \mb Y^{(n)}, \mb Z^{(n)}) \in R) > \alpha$. Let us suppose for now that $R$ has the particularly simple form of a finite union of boxes.
Our argument now proceeds by showing that 
one can construct a distribution $P \in \mathcal{P}_{0, M}$ from the null such that there is a coupling of $P^n$ and $Q^n$ where samples from each distribution are $\epsilon$ close in $\ell_\infty$-norm. For a sufficiently small $\epsilon$, we will have $\pr_P((\mb X^{(n)}, \mb Y^{(n)}, \mb Z^{(n)}) \in R) > \alpha$ as well, giving the result.

Figure~\ref{fig:proofintuition2} sketches the main components in our construction of $P$,
which is laid out formally in Lemmas~\ref{lem:cond_ind_approx} and \ref{lem:hiding2} in the appendix.
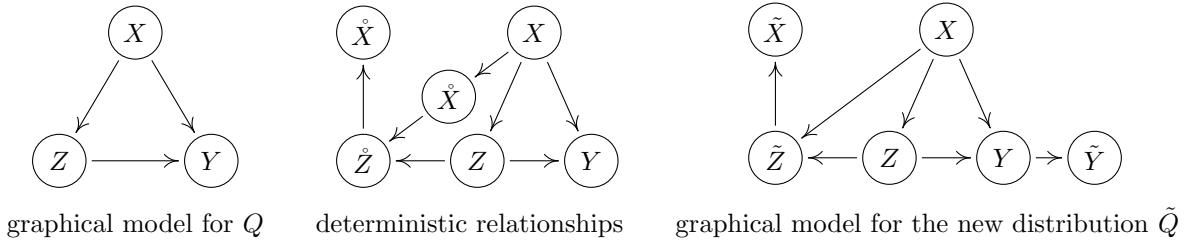
\begin{figure}[ht] 
\centerline{
\begin{tikzpicture}[xscale=2, yscale=1.7, shorten >=1pt, shorten <=1pt]
\small
  \node[observedsmall] at (1.5,2) (x)  {$X$};
  \node[observedsmall] at (2,1) (y)  {$Y$};
  \node[observedsmall] at (1,1) (z)  {$Z$};
  \draw[-arcsq] (x) -- (y);
  \draw[-arcsq] (z) -- (y);
  \draw[-arcsq] (x) -- (z);
  \draw[align=center] (1.5,0.5) node(x2p) 
  {graphical model for $Q$}; %
\end{tikzpicture}
\hfill
\begin{tikzpicture}[xscale=2, yscale=1.7, shorten >=1pt, shorten <=1pt]
\small
  \node[observedsmall] at (2.125,2) (x)  {$X$};
  \node[observedsmall] at (1,2) (tx)  {$\checknew{X}$};
  \node[observedsmall] at (1.5625,1.5) (tx2)  {$\checknew{X}$};
  \node[observedsmall] at (2.5,1) (y)  {$Y$};
  \node[observedsmall] at (1.75,1) (z)  {$Z$};
  \node[observedsmall] at (1,1) (tz)  {$\checknew{Z}$};
  \draw[-arcsq] (z) -- (y);
  \draw[-arcsq] (x) -- (y);
  \draw[-arcsq] (x) -- (z);
  \draw[-arcsq] (x) -- (tx2);
  \draw[-arcsq] (tx2) -- (tz);
  \draw[-arcsq] (tz) -- (tx);
  \draw[-arcsq] (z) -- (tz);
  \draw[align=center] (1.7,0.5) node(x2p) {deterministic relationships};
\end{tikzpicture}
\hfill
\begin{tikzpicture}[xscale=2, yscale=1.7, shorten >=1pt, shorten <=1pt]
\small
  \node[observedsmall] at (2.125,2) (x)  {$X$};
  \node[observedsmall] at (1,2) (tx)  {$\tilde{X}$};
  \node[observedsmall] at (2.5,1) (y)  {$Y$};
  \node[observedsmall] at (3.1,1) (ty)  {$\tilde{Y}$};
  \node[observedsmall] at (1.75,1) (z)  {$Z$};
  \node[observedsmall] at (1,1) (tz)  {$\tilde{Z}$};
  \draw[-arcsq] (z) -- (y);
  \draw[-arcsq] (x) -- (y);
  \draw[-arcsq] (x) -- (z);
  \draw[-arcsq] (x) -- (tz);
  \draw[-arcsq] (tz) -- (tx);
  \draw[-arcsq] (z) -- (tz);
  \draw[-arcsq] (y) -- (ty);
  \draw[align=center] (2,0.5) node(x2p) 
  {graphical model for the new distribution $\tilde{Q}$}; %
\end{tikzpicture}
}
\caption{\label{fig:proofintuition2}Illustration of the main idea of the proof of Theorem~\ref{THM:NFL}. 
Left: 
We start with a distribution $Q \in \mathcal{Q}_{0, M}$
over $(X,Y,Z)$.
In general, the $Q$ is Markov only to a fully connected graphical model. 
Middle:
After discretising $X$ to $\checknew{X}$, we are able to ``hide'' variable $\checknew{X}$ in $\checknew{Z} = f(Z, \checknew{X})$ such that
variable $\checknew{Z}$
is close to 
$Z$ in $\ell_\infty$-norm %
and 
$\checknew{X}$ can be reconstructed from $\checknew{Z}$.
Thus,
$\checknew{X}$, does not contain any ``additional information'' about~$Y$ when conditioning on $\checknew{Z}$. 
Right:
We then consider noisy versions of the variables to guarantee that the new distribution $P$ (over $\tilde{X}, \tilde{Y}, \tilde{Z}$) is absolutely continuous with respect to Lebesgue measure, and has $\tilde{X} \independent \tilde{Y} \given \tilde{Z}$. (The noise in $\tilde{Z}$ is such 
that it still allows us to reconstruct $\checknew{X}$ from~$\tilde{Z}$.) 
}
\end{figure}
The key idea is as follows.
Given $(X, Y, Z) \sim P$, we consider a binary expansion of $(X, Y, Z)$, which we truncate at some point to obtain $(\checknew{X}, \checknew{Y}, \checknew{Z})$. 
We then concatenate the digits of $\checknew{X}$ and $\checknew{Z}$ placing the former at the end of the binary expansion, thereby embedding $\checknew{X}$ within $\checknew{Z}$. 
This way, $\checknew{X}$ can be reconstructed from $\checknew{Z}$, and adding noise gives a distribution that is absolutely continuous with respect to Lebesgue measure. By making the truncation point sufficiently far down the expansions, we can ensure the $\epsilon$ proximity required.

For a general rejection region, we first approximate it using a finite union of boxes $R^{\sharp}$.
The argument sketched above gives us $\pr_P((\mb X^{(n)}, \mb Y^{(n)}, \mb Z^{(n)}) \in R^\sharp) > \alpha$, but in order to conclude the final result, we must argue that we can construct $P$ such that $\pr_P((\mb X^{(n)}, \mb Y^{(n)}, \mb Z^{(n)}) \in R^\sharp \setminus R)$ is sufficiently small. To do this, we consider a large number of potential embeddings for which the supports of the resulting distributions have little overlap. Using a probabilistic argument, we can then show that at least one embedding yields a distribution $P$ such that the above is satisfied.

\section{The Generalised Covariance Measure} \label{SEC:GCM}
We have seen how conditional independence testing is not possible without restricting the null hypothesis. In this section we give a general construction for a conditional independence test based on regression procedures for regressing $\mb Y^{(n)}$ and $\mb X^{(n)}$ on $\mb Z^{(n)}$. In the case where $d_X=d_Y=1$, which we treat in the next section, the basic form of our test statistic is a normalised covariance between the residuals from these regressions. Because of this, we call our test statistic the \emph{generalised covariance measure} (GCM). In Section~\ref{sec:multivariate} we show how to extend the approach to handle cases where more generally $d_X, d_Y\geq 1$.

\subsection{Univariate $X$ and $Y$} \label{sec:univariate}
Given a distribution $P$ for $(X, Y, Z)$, we can always decompose
\begin{align*}
X = f_P(Z) + \varepsilon_P, \qquad
Y = g_P(Z) + \xi_P,
\end{align*}
where $f_P(z) = \E_P(X | Z = z)$ and $g_P(z) = \E_P(Y | Z=z)$. Similarly, for $i=1,2,\ldots$ we define $\varepsilon_{P,i}$ and $\xi_{P,i}$ by $x_i - f_P(z_i)$ and $y_i - g_P(z_i)$ respectively. Also let $u_P(z) = \E_P(\varepsilon_P^2 |Z=z)$ and $v_P(z)= \E_P(\xi_P^2|Z=z)$.

Let $\hat{f}^{(n)}$ and $\hat{g}^{(n)}$ be estimates of the conditional expectations $f_P$ and $g_P$ formed, for example, by regressing $\mb X^{(n)}$ and $\mb Y^{(n)}$ on $\mb Z^{(n)}$. For $i=1,\ldots, n$, we compute the product between residuals from the regressions:
\begin{equation} \label{eq:resid}
R_i = \{x_i - \hat{f}(z_i)\} \{y_i - \hat{g}(z_i)\}.
\end{equation}
Here, and in what follows, we have sometimes suppressed dependence on $n$ and $P$ for simplicity of presentation.
We then define $T^{(n)}$ 
to be a normalised sum of the $R_i$'s:
\begin{equation} \label{eq:T_def}
T^{(n)}  = \frac{
 \sqrt{n}
\cdot
 \frac{1}{n}\sum_{i=1}^n R_i}{\Big(\frac{1}{n}\sum_{i=1}^n R_i^2 - \big(\frac{1}{n} \sum_{r=1}^n R_r \big)^2 \Big)^{1/2}}=:\frac{\tau_N^{(n)}}{\tau_D^{(n)}}.
\end{equation}
Our final test can be based on $|T^{(n)}|$ with large values suggesting rejection. Note that the introduction of notation for the numerator and denominator in the definition of $T^{(n)}$ are for later use in Theorem~\ref{thm:power}.

In the case where $\hat{f}$ and $\hat{g}$ are formed through linear regressions, the test is similar to one based on partial correlation, and would be identical were the denominator in \eqref{eq:T_def} to be replaced by the the product of the empirical standard deviations of the vectors $(x_i-\hat{f}(z_i))_{i=1}^n$ and $(y_i-\hat{g}(z_i))_{i=1}^n$. This approach however would fail for Example~\ref{ex:hetero} despite $f$ and $g$ being linear (in fact both equal to the zero function) as the product of the variances of the residuals would not in general equal the variance of their product. 
Indeed, the reader may convince herself using \texttt{pcor.test} from the R package \texttt{ppcor} \citep{ppcor}, for example, that common tests for vanishing partial correlation do not yield the correct size in this case.

The following result gives conditions under which when the null hypothesis of conditional independence holds, we can expect the asymptotic distribution of $T^{(n)}$ to be a standard normal.
\begin{theorem} \label{thm:univariate}
Define the following quantities:
\begin{gather*}
A_f \;:=\; \frac{1}{n}\sum_{i=1}^n \{f_P(z_i) - \hat{f}(z_i)\}^2, \qquad
B_f \;:=\; \frac{1}{n}\sum_{i=1}^n \{f_P(z_i) - \hat{f}(z_i)\}^2 v_P(z_i), \\
A_g \;:=\; \frac{1}{n}\sum_{i=1}^n \{g_P(z_i) - \hat{g}(z_i)\}^2, \qquad
B_g \;:=\; \frac{1}{n}\sum_{i=1}^n \{g_P(z_i) - \hat{g}(z_i)\}^2 u_P(z_i).
\end{gather*}

We have the following results:
\begin{enumerate}[(i)]
\item If for $P \in \mathcal{P}_0$, $A_f A_g = o_P(n^{-1})$, $B_f=o_P(1)$, $B_g=o_P(1)$ and also $0<\E_P(\varepsilon_P^2 \xi_P^2) < \infty$, then
\[
 \sup_{t \in \R} |\pr_P(T^{(n)}\leq t) -\Phi(t)|\to 0.
\]
\item Let $\mathcal{P} \subset \mathcal{P}_0$ be a class of distributions such that $A_f A_g = o_\mathcal{P}(n^{-1})$, $B_f=o_\mathcal{P}(1)$, $B_g=o_\mathcal{P}(1)$.
If in addition $\inf_{P \in \mathcal{P}}\E(\varepsilon_P^2\xi_P^2) \geq c_1$ and $\sup_{P \in \mathcal{P}} \E_P\{|\varepsilon_P\xi_P|^{2+\eta} \} \leq c_2$ for some $c_1, c_2 > 0$ and $\eta > 0$, then
\[
 \sup_{P \in \mathcal{P}} \sup_{t \in \R} |\pr_P(T^{(n)}\leq t) -\Phi(t)|\to 0.
\]
\end{enumerate}
\end{theorem}
\begin{remark} \label{rem:univariate}
Applying the Cauchy--Schwarz inequality and Markov's inequality, we see the requirement that $A_f A_g = o_P(n^{-1})$ is fulfilled if 
\begin{equation} \label{eq:exp_pred_err}
n\;\; \E_P \bigg(\frac{1}{n}\sum_{i=1}^n \{f_P(z_i) - \hat{f}(z_i)\}^2\bigg) \;\; \E_P \bigg(\frac{1}{n}\sum_{i=1}^n \{g_P(z_i) - \hat{g}(z_i)\}^2\bigg) \to 0.
\end{equation}
Thus if in addition we have $\E_P B_f, \, \E_P B_g \to 0$, this is sufficient for all conditions required in (i) to hold.

If $\E_P B_f$, $\E_P B_g$ and the left-hand side of \eqref{eq:exp_pred_err} converges to 0 uniformly over all $P \in \mathcal{P}$, then the conditions in (ii) will hold provided the moment condition on $\varepsilon_P\xi_P$ is also satisfied.
\end{remark}
A proof is given in the supplementary material.
We see that under conditions largely to do with the mean squared prediction error (MSPE) of $\hat{f}$ and $\hat{g}$, $T^{(n)}$ can be shown to be asymptotically standard normal (i), and if the prediction error is uniformly small, the convergence to the Gaussian limit is correspondingly uniform (ii). A key point is that the requirement on the predictive properties of $\hat{f}$ and $\hat{g}$ is reasonably weak: for example, provided their MSPEs are $o(n^{-1/2})$, we have that the condition on $A_f A_g$ is satisfied. 
If in addition $\max_{i=1}^n |v_P(z_i)|$ and $\max_{i=1}^n |u_P(z_i)|$ are $O_P(\sqrt{n})$, then the conditions on $B_f$ and $B_g$ will be automatically satisfied. The latter conditions would hold if $ \E_P u_P^2(Z) < \infty$ and $\E_P v_P^2(Z) < \infty$, for example.

Note that the rate of convergence requirement on $A_f$ and $A_g$ is a slower rate of convergence than the rate obtained when estimating Lipschitz regression functions when $d_Z=1$, for example. Furthermore, we show in Section~\ref{sec:kernel} that $f$ and $g$ being in a reproducing kernel Hilbert space (RKHS) is enough for them to be estimable at the required rate.

In the setting where $Z$ is high-dimensional and $f$ and $g$ are sparse and linear, standard theory for the Lasso \citep{tibshirani96regression, buhlmann2011statistics} shows that it may be used to obtain estimates $\hat{f}$ and $\hat{g}$ satisfying the required properties under appropriate sparsity conditions. %
In fact, in this case our test statistic is closely related to that involved in the ANT procedure of \citet{Ren2015} and the so-called RP test introduced in \citet{Shah2018}, which amount to a regularised partial correlation. A difference is that the denominator in \eqref{eq:T_def} means the GCM test would not require $\varepsilon_P \independent \xi_P$ unlike the ANT test and the RP test.

We now briefly sketch the reason for the relatively weak requirement on the MSPEs. In the following we suppress dependence on $P$ for simplicity of presentation. We have
\begin{align}
\frac{1}{\sqrt{n}}\sum_{i=1}^n R_i &= \frac{1}{\sqrt{n}} \sum_{i=1}^n \{f(z_i) - \hat{f}(z_i) + \varepsilon_i\}\{g(z_i) -\hat{g}(z_i) + \xi_i\} \nonumber \\
&= (b + \nu_g + \nu_f) + \frac{1}{\sqrt{n}}\sum_{i=1}^n \varepsilon_i \xi_i, \label{eq:T_num}
\end{align}
where
\begin{gather*}
b := \frac{1}{\sqrt{n}}\sum_{i=1}^n \{f(z_i) - \hat{f}(z_i)\}\{g(z_i) - \hat{g}(z_i)\}, \\
\nu_g := \frac{1}{\sqrt{n}} \sum_{i=1}^n \varepsilon_i \{g(z_i) - \hat{g}(z_i)\} \qquad
\nu_f := \frac{1}{\sqrt{n}} \sum_{i=1}^n \xi_i \{f(z_i) - \hat{f}(z_i)\}.
\end{gather*}
The summands in the final term in \eqref{eq:T_num} are i.i.d.\ with zero mean provided $P \in \mathcal{P}_0$, so the central limit theorem dictates that these converge to a standard normal. We also see that the simple form of the GCM gives rise to the term $b$ involving a product of bias-type terms from estimating $f$ and $g$, so each term is only required to converge to 0 at a slow rate such that their product is of smaller order than the variance of the final term.
The summands in $\nu_g$ are, under the null, mean zero conditional on $(\mb Y^{(n)}, \mb Z^{(n)})$. This term and similarly $\nu_f$ are therefore both relatively well-behaved, and give rise to the weak conditions on $B_f$ and $B_g$.

\subsubsection{Power of the GCM} \label{sec:power}
We now present a result on the power of a version of the GCM. We may view our test statistic as a normalised version of the conditional covariance $\E_P (\varepsilon_P \xi_P) = \E_P\cov_P(X, Y | Z)$, 
where $\cov_P(X, Y | Z) = \E_P (XY | Z) - 
\E_P (X | Z)\E_P (Y | Z)$.
This is always zero under the null, see Equation~\eqref{eq:Daudin}, and does not necessarily need to be non-zero under an alternative; we can only hope to have power against alternatives where this conditional covariance is non-zero.

Control of the term $b$ in \eqref{eq:T_num} under the alternative can proceed in exactly the same way as under the null. However  control of the terms $\nu_f$ and $\nu_g$ typically requires additional conditions (for example Donsker-type conditions) on the estimators $\hat{f}$ and $\hat{g}$ as under the alternative both the errors $\varepsilon_i$ and $\hat{g}$ can depend on $\mb Y^{(n)}$. 
A notable exception is when $f$ and $g$ are sparse linear functions; in this setting alternative arguments can be used to show the GCM with Lasso regressions has optimal power when $Z$ has a sparse inverse covariance \citep{Ren2015, Shah2018}.
To state a general result avoiding additional conditions, here we will suppose that $\hat{f}$ and $\hat{g}$ have been constructed from an auxiliary training sample, independent of the data $(\mb X^{(n)}, \mb Y^{(n)}, \mb Z^{(n)})$ (e.g.\ through sample-splitting); see, for example, \citep{robins2008higher, Zheng2011}.
A drawback however, compared to the original GCM, is that the corresponding prediction error terms $A_f$ and $A_g$ are here out-of-sample prediction errors. These are typically more sensitive to the distribution of $Z$ and larger than the in-sample prediction errors featuring in Theorem~\ref{thm:univariate}. For this reason we consider the sample splitting approach to be more of a tool to facilitate theoretical analysis and would usually recommend using the original GCM in practice due to its typically better type I error control. 
\begin{theorem} \label{thm:power}
Consider the setup of Theorem~\ref{thm:univariate} but with the following differences: $\hat{f}$ and $\hat{g}$ have been constructed using auxiliary data independent of $(\mb X^{(n)}, \mb Y^{(n)}, \mb Z^{(n)})$; the null hypothesis $\mathcal{P}_0$ is replaced by $\mathcal{E}_0$ the set of all distributions absolutely continuous with respect to Lebesgue measure; and conditions involving $\varepsilon_P \xi_P$ are replaced by those involving the centred version $\varepsilon_P \xi_P - \E_P(\varepsilon_P \xi_P)$.
Define
\[
\rho_P = \E_P\cov_P(X, Y | Z) \qquad \text{and} \qquad \sigma_P = \sqrt{\var_P(\varepsilon_P \xi_P)}.
\]
Then under the 
conditions of (i) in Theorem~\ref{thm:univariate} we have
\[
\sup_{t \in \R} \abs{\pr_P\left(\frac{\tau_N^{(n)} - \sqrt{n}\rho_P}{\tau_D^{(n)}}  \leq t\right) - \Phi(t)} \to 0, \;\;\;  \tau_D^{(n)} - \sigma_P = o_P(1),
\]
with
$\tau_N^{(n)}$ and
$\tau_D^{(n)}$ defined as in~\eqref{eq:T_def}.

Under the conditions of (ii) in Theorem~\ref{thm:univariate} we have
\[
\sup_{P \in \mathcal{P}} \sup_{t \in \R} \abs{\pr_P\left(\frac{\tau_N^{(n)} - \sqrt{n}\rho_P}{\tau_D^{(n)}}  \leq t\right) - \Phi(t)} \to 0, \;\;\;  \tau_D^{(n)} - \sigma_P = o_\mathcal{P}(1).
\]
\end{theorem}
A proof is given in the supplementary material. We see that we achieve optimal $\sqrt{n}$ rates for estimating $\rho_P$.

\subsubsection{Relationship to semiparametric models} \label{sec:semiparametric}
In viewing $\tau_N^{(n)} / \sqrt{n}$ as an estimator of the functional $\rho_P$, our GCM test connects to a vast literature in semiparametric statistics. In particular, the requirement of estimating nonparametric quantities (in our case $\sqrt{A_f}$ and $\sqrt{A_g}$) at a $o(n^{-1/4})$ rate is common for estimators of functionals based on estimating equations involving influence functions \citep{bickel1993efficient}. Our requirement on prediction error necessitates that at least one of $f_P$ and $g_P$ is H\"older $\beta$-smooth with $\beta / (2\beta + d_Z) \geq 1/4$. Estimators of the expected conditional covariance functional requiring minimal possible smoothness conditions may be derived using the theory of higher order influence functions \citep{robins2008higher, robins2009semiparametric, li2011higher, robins2017minimax}; 
these estimators are however significantly more complicated.
\citet{newey2018cross} study another approach to estimation of the functional based on a particular spline-based regression method. The work of \citet{chernozhukov2017double} uses related ideas to ours here to obtain $1/\sqrt{n}$ convergent estimates and confidence intervals for parameters such as average treatment effects in causal inference settings. A distinguishing feature of our work here is that we only require in-sample prediction error bounds under the null of conditional independence, which is advantageous in our setting for the reasons mentioned in the previous section.

\subsection{Multivariate $X$ and $Y$} \label{sec:multivariate}
We now consider the more general setting where $d_X,\, d_Y \geq 1$, and will assume for technical reasons that $d_X d_Y \geq 3$. We let $T^{(n)}_{jk}$ be the univariate GCM based on data $(\mb X^{(n)}_j, \mb Y^{(n)}_k, \mb Z^{(n)})$ 
and regression methods $\hat{f}_j$ and $\hat{g}_k$. 
(As described in Section~\ref{sec:notation}, the subindex selects a column.)
More generally, we will add subscripts $j$ and $k$ to certain terms defined in the previous subsection to indicate that the quantities are based on $X_j$ and $Y_k$ rather than $X$ and $Y$. Thus, for example, $\varepsilon_{P,j}$ is the difference of $X_j$ and its conditional expectation given~$Z$.

We define our aggregated test statistic to be
\begin{equation*}
S_n = \max_{j=1,\ldots,d_X,\,k=1,\ldots,d_Y} |T^{(n)}_{jk}|.
\end{equation*}
There are other choices for how to combine the test statistics in 
$\mb T^{(n)} := (T^{(n)}_{jk})_{j,k} \in \R^{d_X \cdot d_Y}$
into a single test statistic. Under similar conditions to those in Theorem~\ref{thm:univariate}, one can show that if $d_X$ and $d_Y$ are fixed, $\mb T^{(n)}$ will converge in distribution to a multivariate Gaussian limit with a covariance that can be estimated. The continuous mapping theorem can then be used to deduce the asymptotic limit distribution of the sum of squares of $T^{(n)}_{jk}$, for example. However, one advantage of the maximum is that the bias component of $S_n$ will be bounded by the maximum of the bias terms in $T^{(n)}_{jk}$. A sum of squares-type statistic would have a larger bias component, and tests based on it may not maintain the level for moderate to large $d_X$ or $d_Y$. Furthermore, $S_n$ will tend to exhibit good power against alternatives where conditional independence is only violated for a few pairs $(X_j, Y_k)$, i.e., when the set of $(j,k)$ such that $X_j \notindependent Y_k | Z$ is small.

In order to understand what values of $S_n$ indicate rejection, we will compare $S_n$ to
\begin{equation*}
\hat{S}_n = \max_{j=1,\ldots,d_X,\,k=1,\ldots,d_Y} |\hat{T}^{(n)}_{jk}|
\end{equation*}
where $\hat{\mb T}^{(n)} \in \R^{d_X \cdot d_Y}$ is mean zero multivariate Gaussian with a covariance matrix $\hat{\mbb \Sigma} \in \R^{d_X \cdot d_Y \times d_X \cdot d_Y}$ determined from the data as follows.
Let $\mb R_{jk} \in \R^n$ be the vector of products of residuals \eqref{eq:resid} involved in constructing the test statistic $T^{(n)}_{jk}$. We set $\hat{\Sigma}_{jk,lm}$ to be the sample correlation between $\mb R_{jk} \in \R^n$ and $\mb R_{lm}$:
\[
\hat{\Sigma}_{jk,lm} = \frac{\mb R_{jk}^T \mb R_{lm}/n - \bar{\mb R}_{jk} \bar{\mb R}_{lm}}{\big(\|\mb R_{jk}\|_2^2/n - \bar{\mb R}_{jk}^2\big)^{1/2} \big(\|\mb R_{lm}\|_2^2/n - \bar{\mb R}_{lm}^2\big)^{1/2}}.
\]
Here $\bar{\mb R}_{jk}$ is the sample mean of the components of $\mb R_{jk}$.

Let $\hat{G}$ be the quantile function of $\hat{S}_n$. This is a random function that depends on the data $(\mb X^{(n)}, \mb Y^{(n)}, \mb Z^{(n)})$ through $\hat{\mbb\Sigma}$. Note that given the $\mb R_{jk}$, we can approximate $\hat{G}$ to any degree of accuracy via Monte Carlo.

The ground-breaking work of \citet{chernozhukov2013gaussian} gives conditions under which $\hat{G}$ can well-approximate the quantile function of a version of $S_n$ where all bias terms, that is terms corresponding to $b$, $\nu_g$ and $\nu_f$ are all equal to 0. We will require that those conditions are met by $\varepsilon_{P,j} \xi_{P,k}$ for all $j=1,\ldots,d_X$, $k=1,\ldots,d_Y$. Below, we lay out these conditions, which take two possible forms. Let $d=\max(d_X, d_Y)$; note that $d$ and $\mathcal{P}$ are permitted to change with $n$, though we suppress this in the notation.
\begin{itemize}
\item[(A1a)] $\max_{r=1,2} \E_P(|\varepsilon_{P,j} \xi_{P,k}|^{2+r}/C_n^r) + \E_P(\exp(|\varepsilon_{P,j} \xi_{P,k}|/C_n)) \leq 4$;
\item[(A1b)] $\max_{r=1,2} \E_P(|\varepsilon_{P,j} \xi_{P,k}|^{2+r}/C_n^{r/2}) + \E_P(\max_{j,k}|\varepsilon_{P,j} \xi_{P,k}|^4/C_n^2) \leq 4$;
\item[(A2)] $C_n^2 (\log(d n))^7/n \leq Cn^{-c}$ and $\E_P(\varepsilon_{P,j}^2 \xi_{P,k}^2)\geq c_1$.
\end{itemize}
The result below shows that under the moment conditions above, provided the prediction error following the regressions goes to zero sufficiently fast, $\hat{G}$ closely approximates the quantile function of $S_n$ and therefore may be used to correctly calibrate our test. 
\begin{theorem} \label{thm:multivariate}
Suppose that for $\mathcal{P} \subset \mathcal{P}_0$, the following is true: there are constants $C, c, c_1 >0$ such that for each $n$ and $P \in \mathcal{P}$, there exists $C_n \geq 1$ such that one of (A1a) and (A1b) hold, and that (A2) holds. Suppose that 
\begin{gather}
\max_{j,k} A_{f,j}A_{g,k} = o_\mathcal{P}(n^{-1} \log(d)^{-2}), \label{eq:A_cond}
\end{gather}
Suppose further that there exist sequences 
$(\tau_{f,n})_{n \in \mathbb{N}}$, 
$(\tau_{g,n})_{n \in \mathbb{N}}$ 
such that
\begin{gather}
\max_{i,j} |\varepsilon_{P,ij}| = O_\mathcal{P}(\tau_{g,n}), \qquad \max_k A_{g,k} = 
o_\mathcal{P}(\tau_{g,n}^{-2}\log(d)^{-4})
\label{eq:tau_g_cond} \\
\max_{i,k} |\xi_{P,ik}| = O_\mathcal{P}(\tau_{f,n}), \qquad \max_j A_{f,j} = 
o_\mathcal{P}(\tau_{f,n}^{-2}\log(d)^{-4}) .\label{eq:tau_f_cond}
\end{gather}
Then
\[
\sup_{P \in \mathcal{P}} \sup_{\alpha \in (0, 1)} |\pr_P\{S_n \leq \hat{G}(\alpha)\} - \alpha| \to 0
\]
\end{theorem}
A proof is given in the supplementary material.
\begin{remark}
If the errors $\{\varepsilon_{P,j}\}_{j=1}^{d_X}$ and $\{\xi_{P,k}\}_{k=1}^{d_Y}$ are all sub-Gaussian with parameters bounded above by some constant $M$ uniformly across $P \in \mathcal{P}$, we may easily see that both (A1a) and (A1b) are satisfied with $C_n$ a constant; see \citet{chernozhukov2013gaussian} for further discussion.

If additionally we have $A_{f,j}, A_{g,k} =  
o_{\mathcal{P}}\left(\log(d)^{-1}\min\{n^{-1/2}, \log(d)^{-4}\}\right)$, \eqref{eq:A_cond}, \eqref{eq:tau_g_cond} and \eqref{eq:tau_f_cond} will all be satisfied.
\end{remark}

Theorem~\ref{thm:multivariate} allows for $d_X$ and $d_Y$ to be large compared to $n$.
However the result can be of use even when faced with univariate data. In this case, or more generally when $d_X$ and $d_Y$ are small, one can consider mappings $f_X : \R^{d_X + d_Z} \to \R^{\tilde{d}_X}$ and $f_Y:\R^{d_Y + d_Z} \to \R^{\tilde{d}_Y}$ where $\tilde{d}_X$ and $\tilde{d}_Y$ are potentially large. Provided these mappings are not determined from the data, 
we will have for $\tilde{X} := f_X(X, Z)$ and $\tilde{Y} := f_Y(Y, Z)$ that $\tilde{X} \independent \tilde{Y} \given Z$ if $X \independent Y \given Z$ (see equation~\eqref{eq:Daudin}). 
Thus we may apply the methodology above to the mapped data, potentially allowing the test to have power against a more diverse set of alternatives. In view of Theorem~\ref{thm:power}, successful mappings should have the equivalent of $\rho_P$ large, but also $\E(\tilde{X} | Z=\cdot)$ and $\E(\tilde{Y}| Z=\cdot)$ should not be so complex that it is impossible to estimate them well. We leave further investigation of this topic to further work.

\section{GCM Based on Kernel Ridge Regression} \label{sec:kernel}
We now apply the results of the previous section to a GCM based on estimating the conditional expectations via kernel ridge regression.
For simplicity, we consider only the univariate case where $d_X=d_Y=1$. In the following, we make use of the notation introduced in Section~\ref{sec:univariate}.

Given $\mathcal{P} \subset \mathcal{P}_0$, suppose that the conditional expectations $f_P, g_P$ satisfy $f_P, g_P \in \mathcal{H}$ for some RKHS $(\mathcal{H}, \|\cdot \|_{\mathcal{H}})$ with reproducing 
kernel $k: \R^{d_{Z}} \times \R^{d_{Z}} \rightarrow \mathbb{R}$. 
Let $K \in \R^{n \times n}$ have $ij$th entry $K_{ij} = k(z_i, z_j)/n$ and denote the eigenvalues of $K$ by $\hat{\mu}_1 \geq \hat{\mu}_2 \geq \cdots \geq \hat{\mu}_n \geq 0$.
We will assume that under each $P \in \mathcal{P}$, $k$ admits an eigen-expansion of the form
\begin{equation} \label{eq:Mercer}
k(z, z') = \sum_{j=1}^\infty \mu_{P,j} e_{P,j}(z) e_{P,j}(z')
\end{equation}
with orthonormal eigenfunctions $\{e_{P,j}\}_{j=1}^\infty$, so $\E_P e_{P,j}e_{P,k}=\ind_{\{k=j\}}$, and summable eigenvalues $\mu_{P,1} \geq \mu_{P,2} \geq \cdots \geq 0$. Such an expansion is guaranteed under mild conditions by Mercer's theorem.

Consider forming estimates $\hat{f}=\hat{f}^{(n)}$ and $\hat{g}=\hat{g}^{(n)}$ through kernel ridge regressions of  $\mb X^{(n)}$ and $\mb Y^{(n)}$ on $\mb Z^{(n)}$ in the following way. For $\lambda > 0$, let
\[
\hat{f}_\lambda = \argmin_{h \in \mathcal{H}} \bigg\{ \frac{1}{n}\sum_{i=1}^n \{x_i-h(z_i)\}^2 + \lambda \|h\|_{\mathcal{H}}^2\bigg\}.
\]
We will consider selecting a final tuning parameter $\hat{\lambda}$ in the following data-dependent way:
\[
\hat{\lambda} = \argmin_{\lambda >0} \bigg\{\frac{1}{n} \sum_{i=1}^n \frac{\hat{\mu}_i^2}{(\hat{\mu}_i + \lambda)^2} + \lambda\bigg\}.
\]
The term minimised on the RHS is an upper bound on the mean-squared prediction error omitting constant factors depending on $\sigma^2$ (defined below in Theorem~\ref{THM:KERNEL}) and $\|f_P\|_{\mathcal{H}}^2$ or $\|g_P\|_{\mathcal{H}}^2$. Because of the hidden dependence on these quantities, this is not necessarily a practically effective way of selecting $\lambda$: our use of it here is simply to facilitate theoretical analysis.
Finally define $\hat{f} = \hat{f}_{\hat{\lambda}}$, and define $\hat{g}$ analogously. We will write $T^{(n)}$ for the test statistic formed as in \eqref{eq:T_def} with these choices of $\hat{f}$ and $\hat{g}$.

\begin{theorem}\label{THM:KERNEL}
Let $\mathcal{P}$ be such that $u_P(z), v_P(z) \leq \sigma^2$ for all $z$ and $P \in \mathcal{P}$.
\begin{enumerate}[(i)]
\item For any $P \in \mathcal{P}$, $\sup_{t \in \R}|\pr_P(T^{(n)} \leq t)-\Phi(t)|\to 0$.
\item Suppose $\sup_{P \in \mathcal{P}}\E_P\{|\varepsilon_P\xi_P|^{2+\eta}\}\leq c$ and $\inf_{P \in \mathcal{P}} \E_P(\varepsilon_P^2 \xi_P^2) > c^{-1}$ for some $c\geq 0$ and $\eta > 0$. Suppose further that $\sup_{P \in \mathcal{P}} \max(\|f_P\|_{\mathcal{H}}, \|g_P\|_{\mathcal{H}}) < \infty$  and
\begin{equation} \label{eq:mu_cond}
\lim_{\lambda \downarrow 0} \sup_{P \in \mathcal{P}} \sum_{j=1}^\infty \min(\mu_{P,j}, \lambda) = 0.
\end{equation}
Then
\[
\sup_{P \in \mathcal{P}} \sup_{t \in \R} |\pr_P(T^{(n)} \leq t) - \Phi(t)| \to 0.
\]
\end{enumerate}
\end{theorem}
A proof is given in the supplementary material.
\begin{remark}
An application of the dominated convergence theorem shows that a sufficient condition for \eqref{eq:mu_cond} to hold is that $\sum_{j=1}^\infty \sup_{P \in \mathcal{P}} \mu_{P,j} <\infty$.
\end{remark}
The proof proceeds by first showing that the ridge regression estimators $\hat{f}$ and $\hat{g}$ satisfy $A_f A_g = o_P(n^{-1})$ and then applies Theorem~\ref{thm:univariate}.
The requirement that $f_P$ and $g_P$ lie in an RKHS satisfying \eqref{eq:Mercer} is a rather weak regularity condition on the conditional expectations. For example, taking the first-order Sobolev kernel shows that it is enough that the conditional expectations are Lipschitz when $d_Z=1$, $\mathbb{P}(Z \in [0,1]) = 1$ and the marginal density of $Z$ is bounded above \citep{bach2017equivalence}.
However, the uniformity offered by (ii) above requires $L:=\sup_{P \in \mathcal{P}} \max(\|f_P\|_{\mathcal{H}}, \|g_P\|_{\mathcal{H}}) < \infty$ and a large value of $L$ will require a large sample size in order for $T^{(n)}$ to have a distribution close to a standard normal. We investigate this, and evaluate the empirical performance of the GCM in the next section.

\section{Experiments} \label{sec:experiments}
Section~\ref{SEC:GCM} proposes the generalised covariance measure (GCM). 
Although we provide detailed computations for kernel ridge regression in Section~\ref{sec:kernel}, the technique can be combined with any regression method. In practice, the choice may depend on external knowledge of the specific application the user has in mind.
In this section, we study the empirical performance of the GCM with boosted regression trees as the regression method. 
In particular, we use 
the R package \texttt{xgboost} \citep{chen2018xgboost, chen2016xgboost} with a ten-fold cross-validation scheme over the parameter \texttt{maxdepth}.

\subsection{No-free-lunch in Conditional Independence Testing} \label{sec:exp:nfl}
Theorem~\ref{THM:NFL} \linebreak states that if a conditional independence test has power against an alternative at a given sample size, then there is a distribution from the null that is rejected with probability larger than the significance level.
We now illustrate the no-free-lunch theorem
empirically.

Let us 
fix an RKHS~$\mathcal{H}$ that corresponds to a Gaussian kernel with bandwidth~$\sigma = 1$.
We now compute for different sample sizes the rejection rates for data sets generated from the following 
model:
$Z = N_Z$,
$Y = f_a(Z) + N_Y$, and
$X = f_a(Z) + N_X$,
with 
$N_X, N_Y, N_Z \sim \mathcal{N}(0,1)$, i.i.d.,
and 
$f_a(z) := \exp(-z^2/2) \sin(az)$
defining a function $f_a \in \mathcal{H}$.
Figure~\ref{fig:exp:nfl-fa}
shows a plot of $f_a$ for $a=6$ and $a=18$.
\begin{figure}
\centerline{\includegraphics[width = 0.48\textwidth]{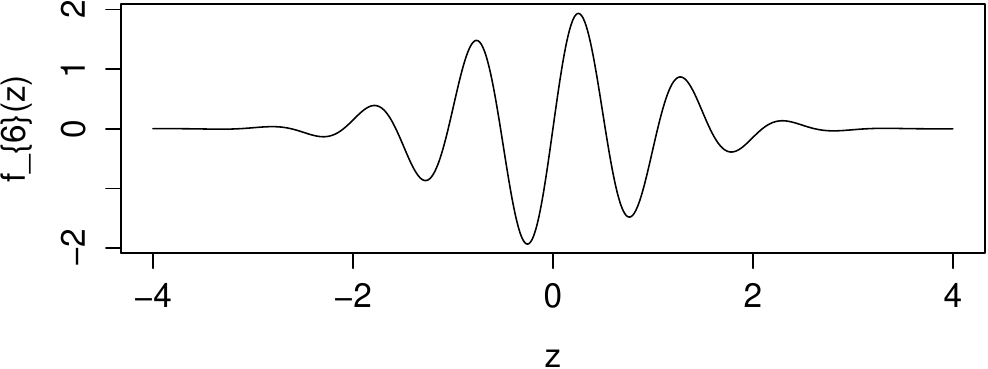}
\hfill
\includegraphics[width = 0.48\textwidth]{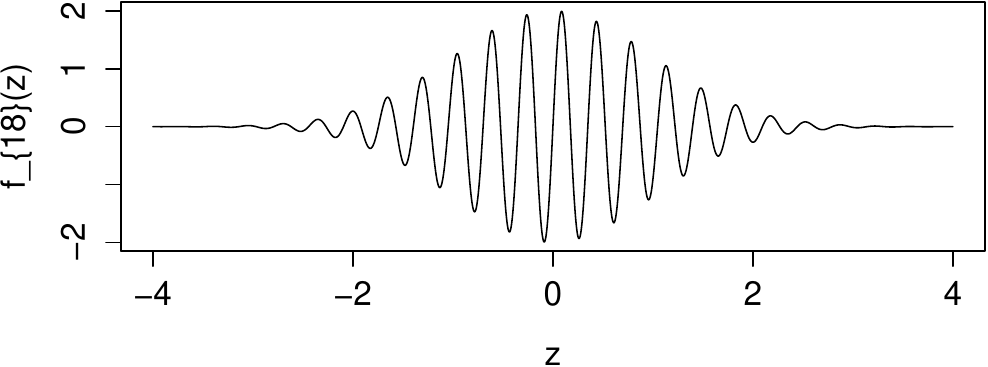}
}
\caption{\label{fig:exp:nfl-fa}Graphs of the function $f_a$ for $a=6$ (left) and $a=18$ (right). This function is used as the conditional mean that needs to be estimated from data. The RKHS norm increases exponentially with $a$, see \eqref{eq:exp:rkhsnorm}.}
\end{figure}
Clearly, for any $a$, we have
$X \independent Y\given Z$, but
for large values of $a$ the independence will be harder to detect from data.
We now fix three different sample sizes $n=100$, $n=1000$, and $n=10000$. For any of such sample size $n$, we can find an $a$, i.e., a distribution from the null, such that the probability of (falsely) rejecting 
$X \independent Y\given Z$
is larger than the prespecified level~$\alpha$.
Figure~\ref{fig:exp:nfl}
shows the results for 
the GCM test with boosted regression trees
and the significance level $\alpha = 0.05$:
for any sample size, there exists a distribution from the null, for which the test rejects the null hypothesis of conditional independence.
\begin{figure} 
\centerline{\includegraphics[width = 0.9\textwidth]{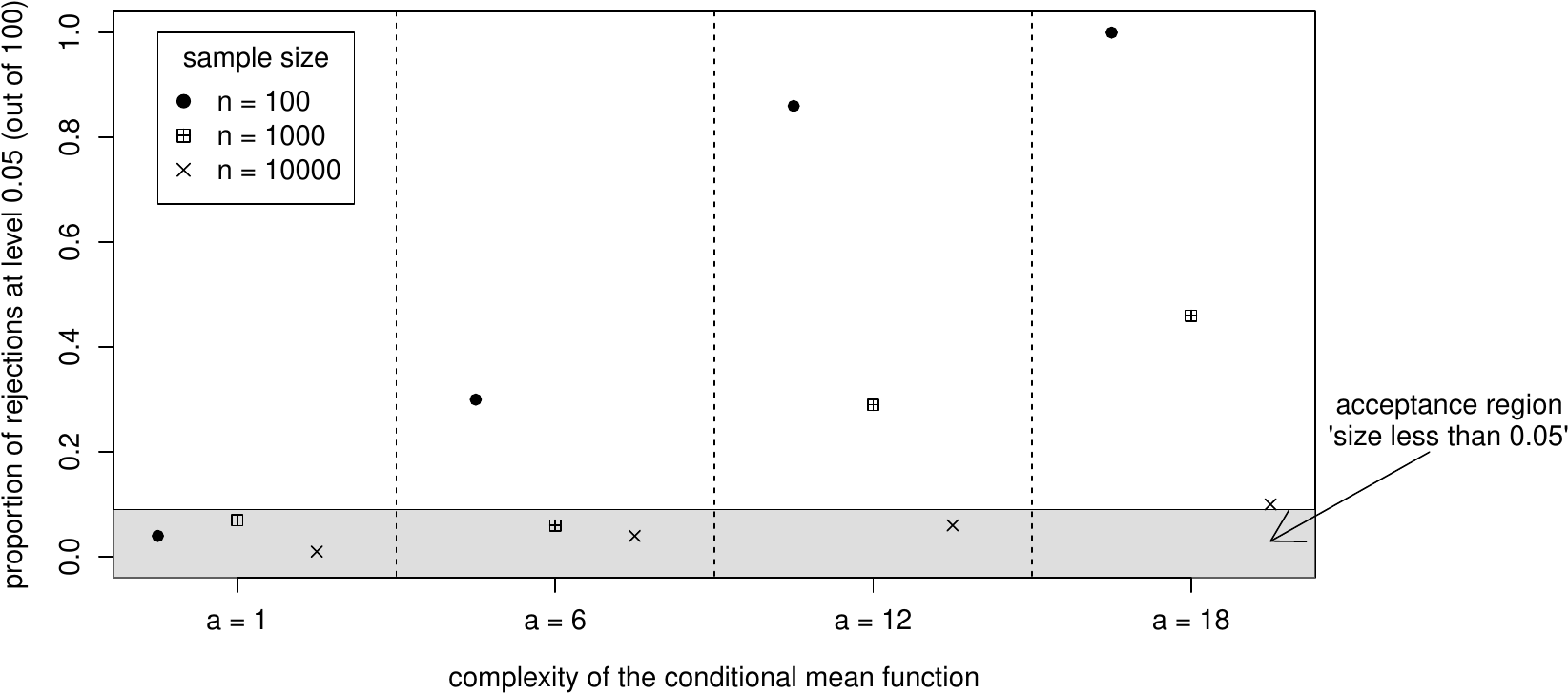}}
\caption{\label{fig:exp:nfl}Illustration of the no-free-lunch theorem, see Section~\ref{sec:exp:nfl}. No sample size is large enough to ensure the correct level for all distributions from the null: there is always a distribution from the null which yields a type I
error that is larger than the prespecified significance level of $0.05$. 
The shaded area indicates the area in which we accept the null hypothesis that the size of the test is less than $0.05$.
}
\end{figure}
For $n=100$, we can choose $a=6$, for $n=1000$, we choose $a=12$, and for $n=10000$, $a=18$.

This sequence of distributions violates one of the assumptions that we require for the GCM test to obtain
uniform asymptotic level guarantee.
Intuitively, 
for large $a$,
the conditional expectations 
$z \mapsto \E[X|Z=z]$ and
$z \mapsto \E[Y|Z=z]$
are too complex to be estimated reliably from the data.
More formally, 
the RKHS norm 
of the functions~$f_a$ are defined as:
\begin{align} \label{eq:exp:rkhsnorm}
\|f_a\|^2_{\mathcal{H}}
= \int_{-\infty}^{\infty} F_a(\omega)^2 \exp(\sigma^2 \omega^2/2) \,d\omega
= \sqrt{8\pi} \cdot \big(\exp(a^2) + \exp(-a^2)\big),
\end{align}
where 
$$
F_a(\omega) = \exp \big(-(\omega - a)^2/2 \big) + \exp \big(-(\omega + a)^2/2 \big)
$$ 
is the Fourier transform of $f_a$.
Equation~\eqref{eq:exp:rkhsnorm} shows that 
a null hypothesis $\mathcal{P}$
containing all of the above models for $a > 0$,
violates 
one 
of the assumptions in Theorem~\ref{THM:KERNEL}:
for this choice of RKHS and null hypothesis there is no $M$ such that
$\sup_{P \in \mathcal{P}} \max(\|f_P\|_{\mathcal{H}}, \|g_P\|_{\mathcal{H}}) < M$. 
(Note that not all sequences of functions with growing 
RKHS norm also yield a violation of level guarantees:
some functions with large RKHS norm, e.g., modifications of constant functions, can be easily learned from data.) 
Other conditional independence tests fail on the examples in Figure~\ref{fig:exp:nfl}, too, for a similar reason. However most of these other methods are less transparent in the underlying assumptions, since they do not come with uniform level guarantees.

\subsection{On Level and Power}
It is of course impossible to provide an exhaustive simulation-based level and power analysis.
We therefore concentrate on a small choice of distributions from the null and the alternative. 
In the following, we compare the GCM with three other conditional independence tests: KCI \citep{Zhang2011uai} with its implementation from \texttt{CondIndTests} \citep{Heinze2017}, and the residual prediction test \citep{Heinze2017, Shah2018}.
We also compare to a test that
performs the same regression as GCM, but then tests for independence between the residuals, rather than vanishing correlation, using HSIC \citep{Gretton2008}. 
(This procedure is similar to the one that \citet{Fan2015} propose to use in the case of additive noise models.)
As we discuss in Example~\ref{ex:hetero}, we do not expect this test to hold level in general.
We then consider the following distributions from the null:
\begin{enumerate}[(a)]
\item $Z \sim \mathcal{N}(0,1)$, 
$X = f_a(Z) + 0.3 \mathcal{N}(0,1)$,
$Y = f_a(Z) + 0.3 \mathcal{N}(0,1)$,
$a=2$;
\item the same as (a) but with $a=4$;
\item 
$Z_1, Z_2 \sim \mathcal{N}(0,1)$ independent, 
$X = f_1(Z_1) - f_1(Z_2) + 0.3 \mathcal{N}(0,1)$,
$Y = f_1(Z_1) + f_1(Z_2) + 0.3 \mathcal{N}(0,1)$;
\item 
$Z \sim \mathcal{N}(0,1)$, 
$X_1 = f_1(Z) + 0.3 \mathcal{N}(0,1)$,
$X_2 = f_1(Z) + X_1 + 0.3 \mathcal{N}(0,1)$,
$Y_1 = f_1(Z) + 0.3 \mathcal{N}(0,1)$,
$Y_2 = f_1(Z) + Y_1 + 0.3 \mathcal{N}(0,1)$; and
\item $Z \sim \mathcal{N}(0,1)$, 
$Y = f_2(Z) \cdot  \mathcal{N}(0,1)$,
$X = f_2(Z) \cdot \mathcal{N}(0,1)$.
\end{enumerate}
In the remainder of this section, we refer to these settings as 
(a) ``$a=2$'',
(b) ``$a=4$'',
(c) ``biv.\ $Z$'',
(d) ``biv.\ $X, Y$'',
and 
(e) ``multipl.\ noise'', respectively.
For each of the sample sizes $50$, $100$, $200$, $300$, and $400$,
we first generate $100$ data sets, and 
then 
compute rejection rates of the 
considered conditional independence
tests.
The results are shown in Figure~\ref{fig:exp:exp3level}. 
For rejection rates below $0.11$ the hypothesis ``the size of the test is less than $0.05$'' is not rejected at level 0.01 (pointwise). 
The GCM indeed has promising behaviour in terms of type I error control. 
As expected, however, it requires the sample size to be big enough to obtain a reliable estimate for the conditional mean.
\begin{figure}[h]
\centerline{
\includegraphics[width = 0.22\textwidth]{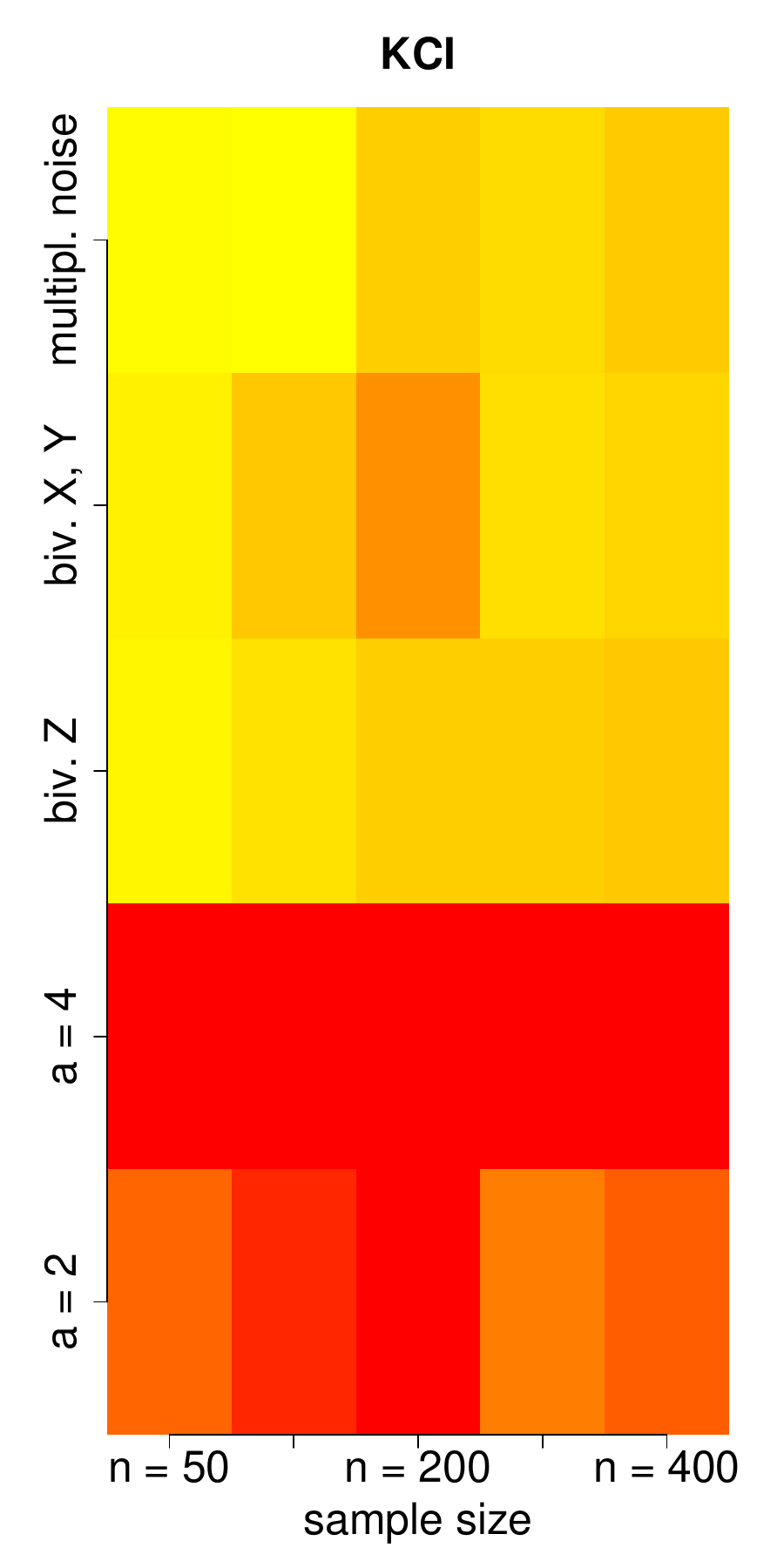}
\includegraphics[width = 0.22\textwidth]{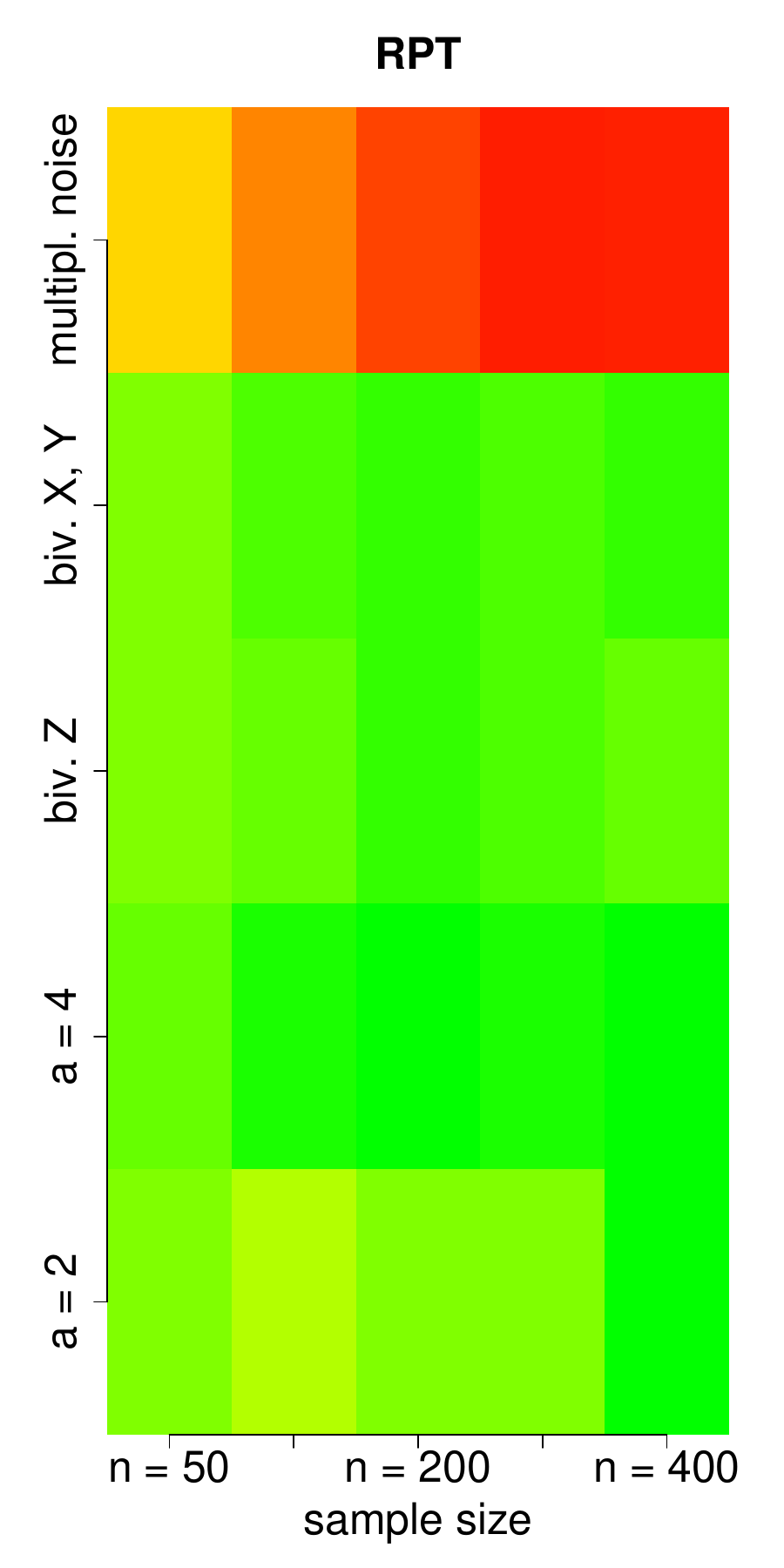}
\includegraphics[width = 0.22\textwidth]{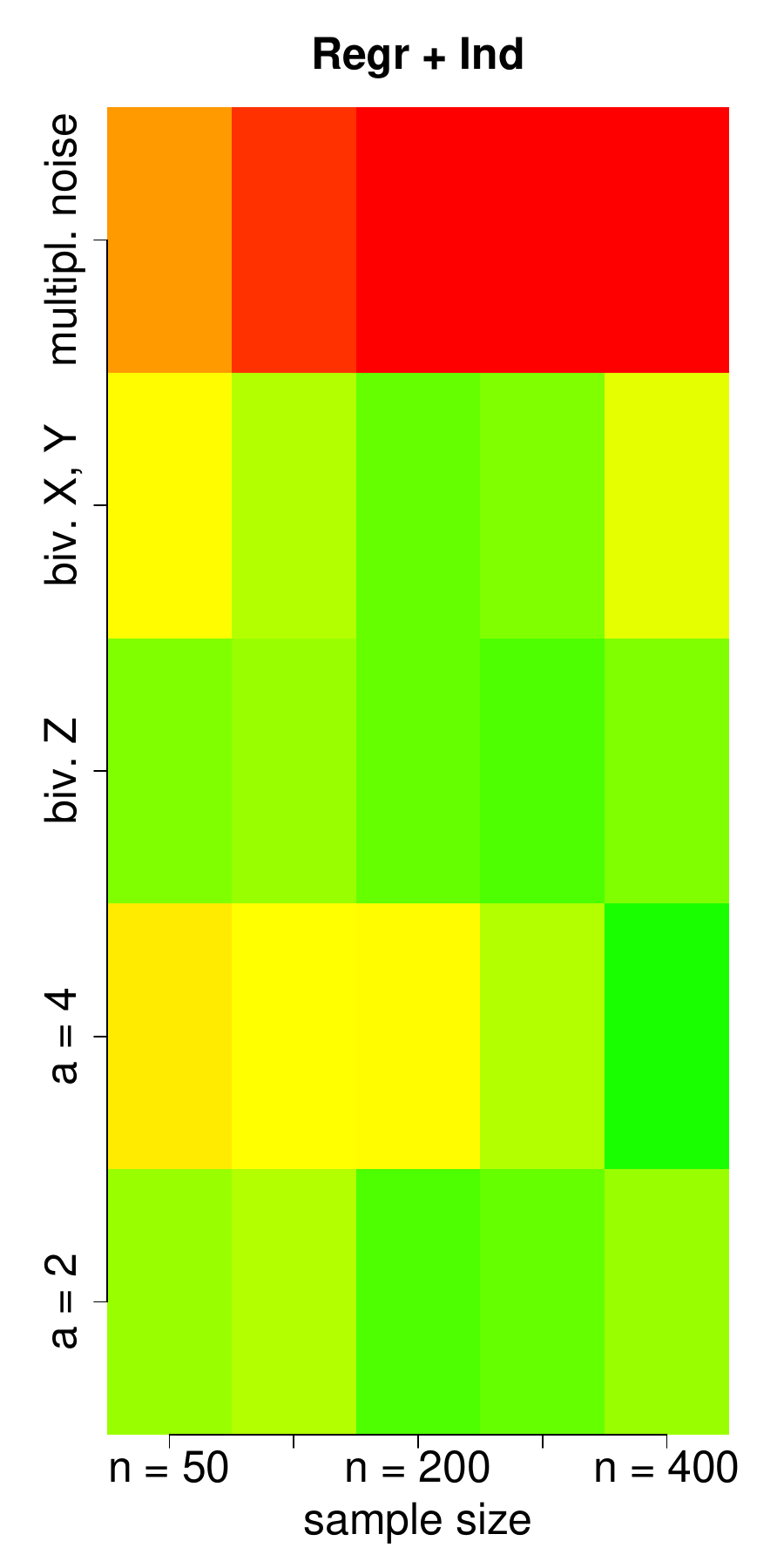}
\includegraphics[width = 0.22\textwidth]{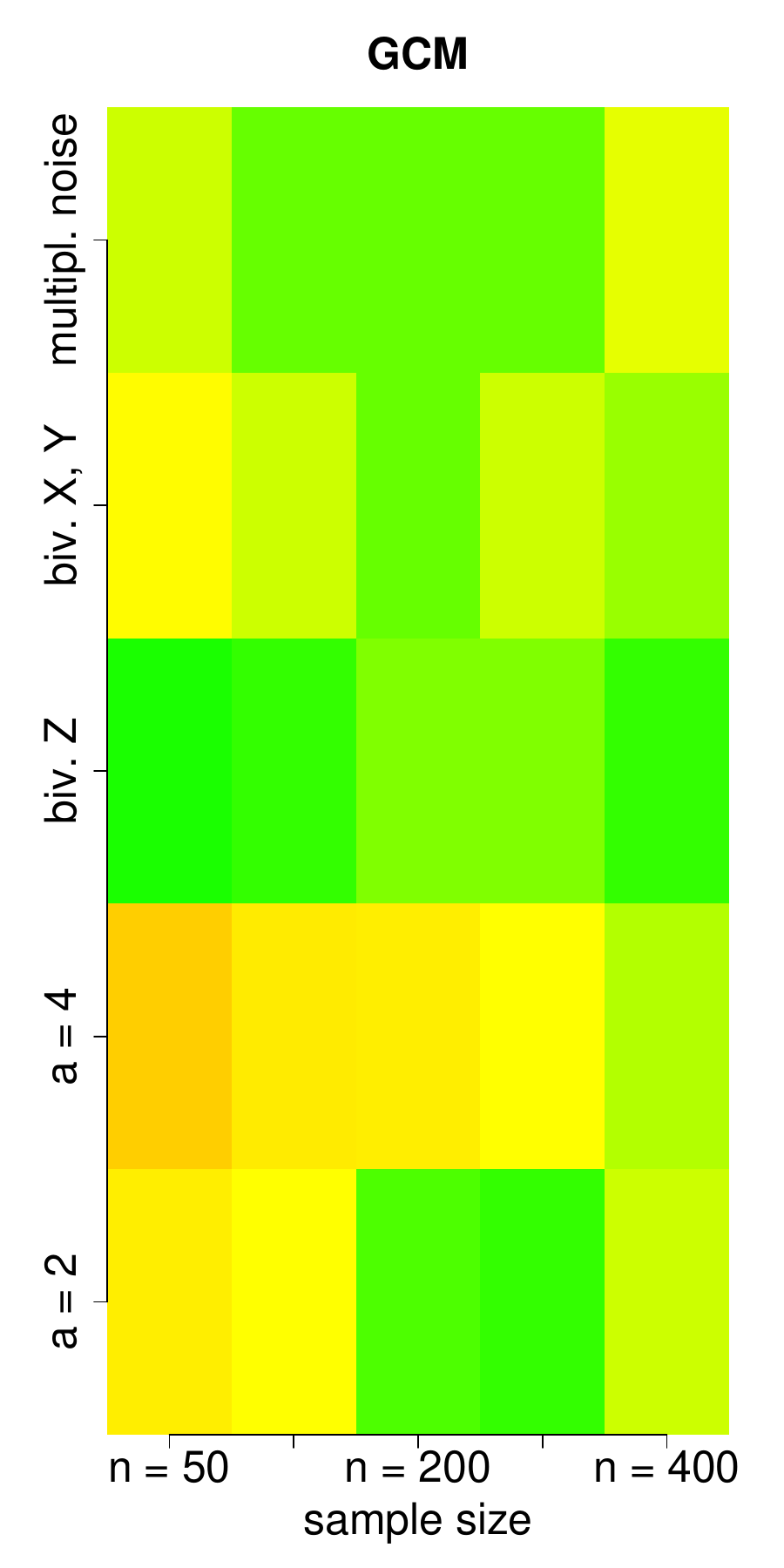}
\hfill
\raisebox{0.5\height}{
\includegraphics[width = 0.1\textwidth]{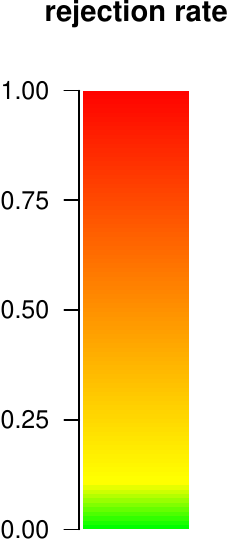}
}
}
\caption{\label{fig:exp:exp3level}Level analysis: the GCM can hold the level if the sample size is large enough to reliably estimate the conditional mean. Testing for independence between residuals does not hold the level (third plot).}
\end{figure}

We then investigate the tests' power 
by altering the 
data generating processes (a)--(e), described
above. 
Each equation for $Y$ receives an additional term $+ 0.2X$, which 
yields
$X \notindependent Y \given Z$
(for (d), we add the term $+ 0.2X_2$ to the equation of $Y_2$). 
Figure~\ref{fig:exp:exp3power} 
shows empirical rejection rates.
\begin{figure}[h]
\centerline{
\includegraphics[width = 0.22\textwidth]{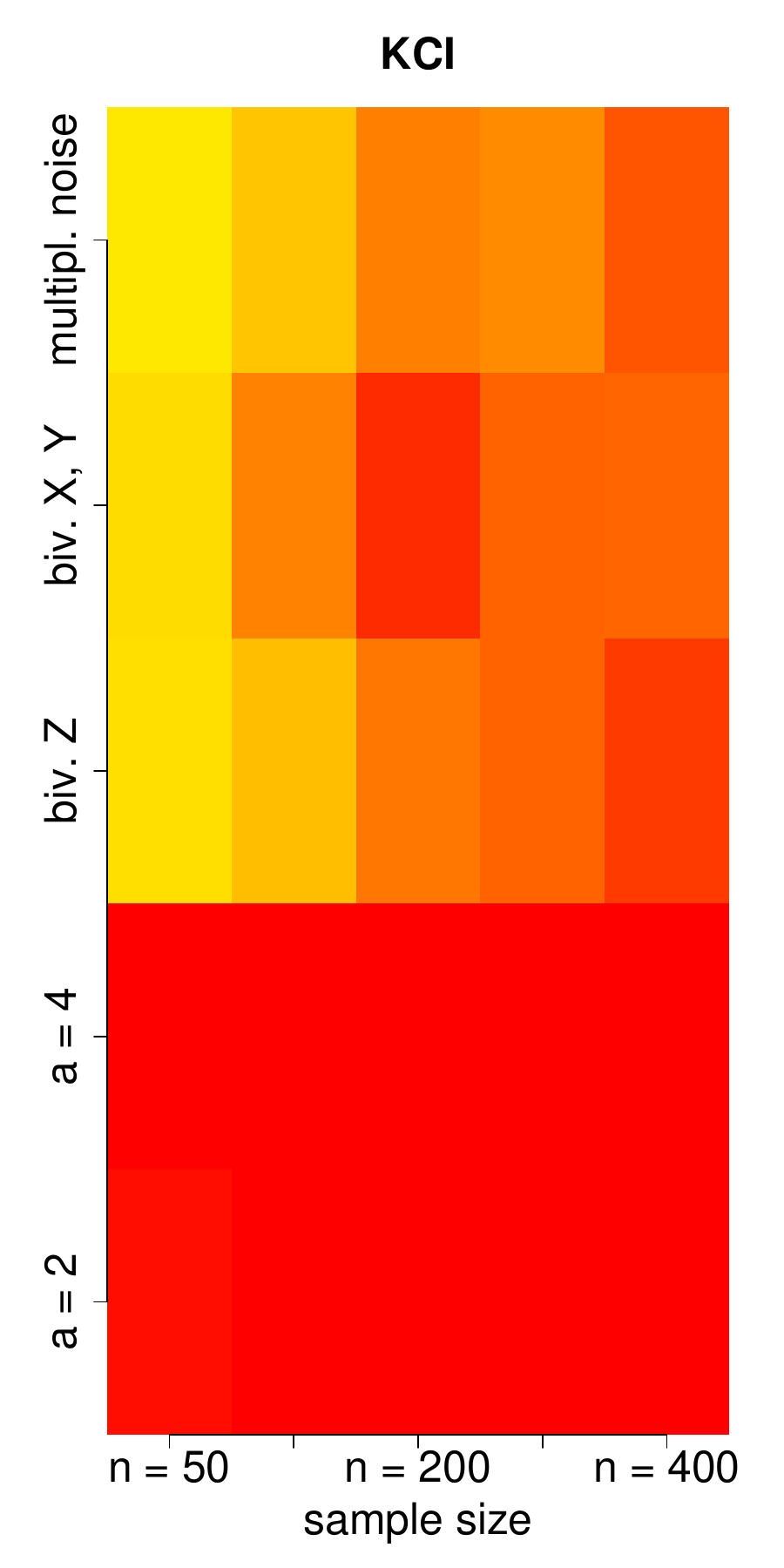}
\includegraphics[width = 0.22\textwidth]{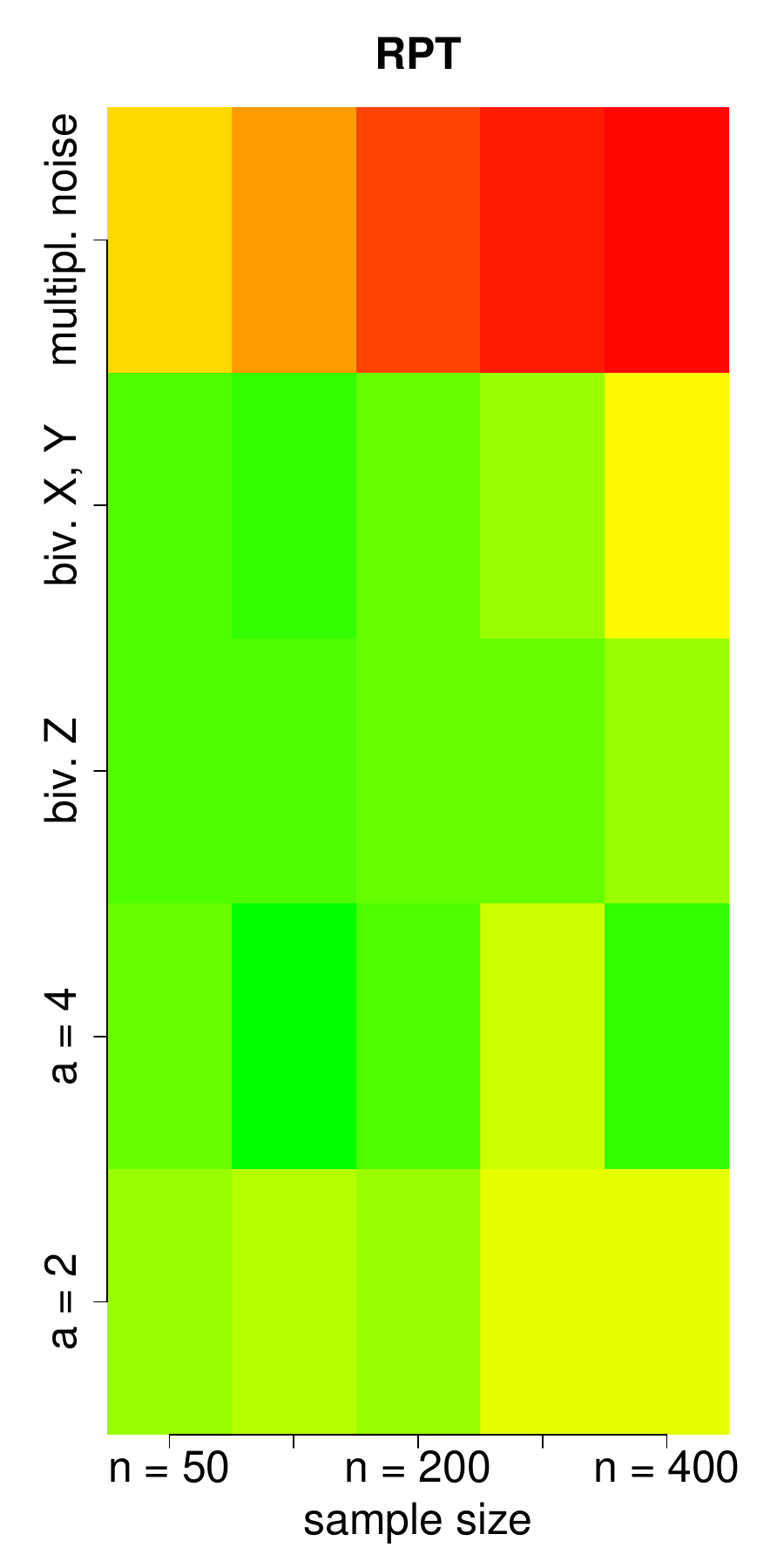}
\includegraphics[width = 0.22\textwidth]{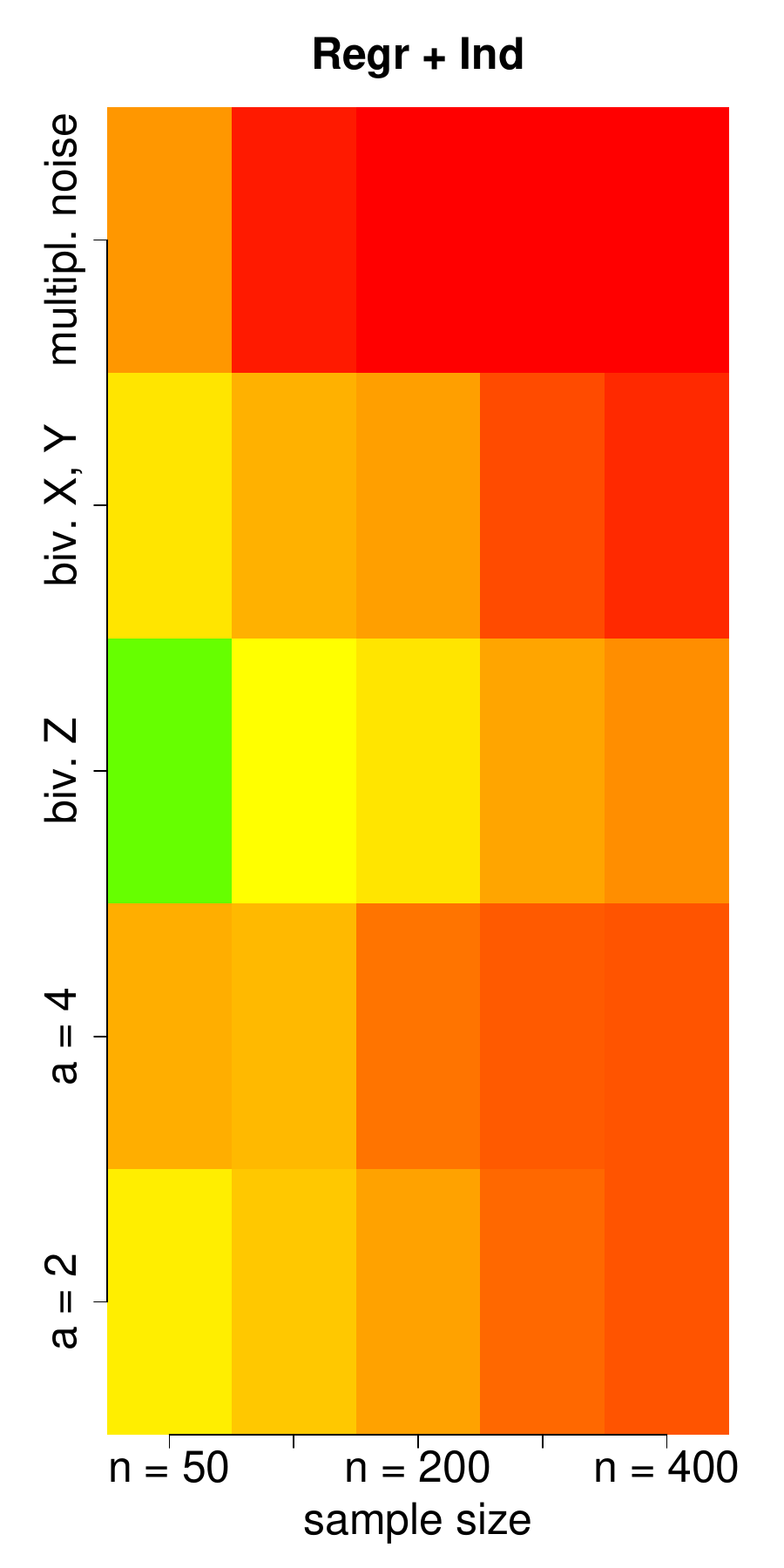}
\includegraphics[width = 0.22\textwidth]{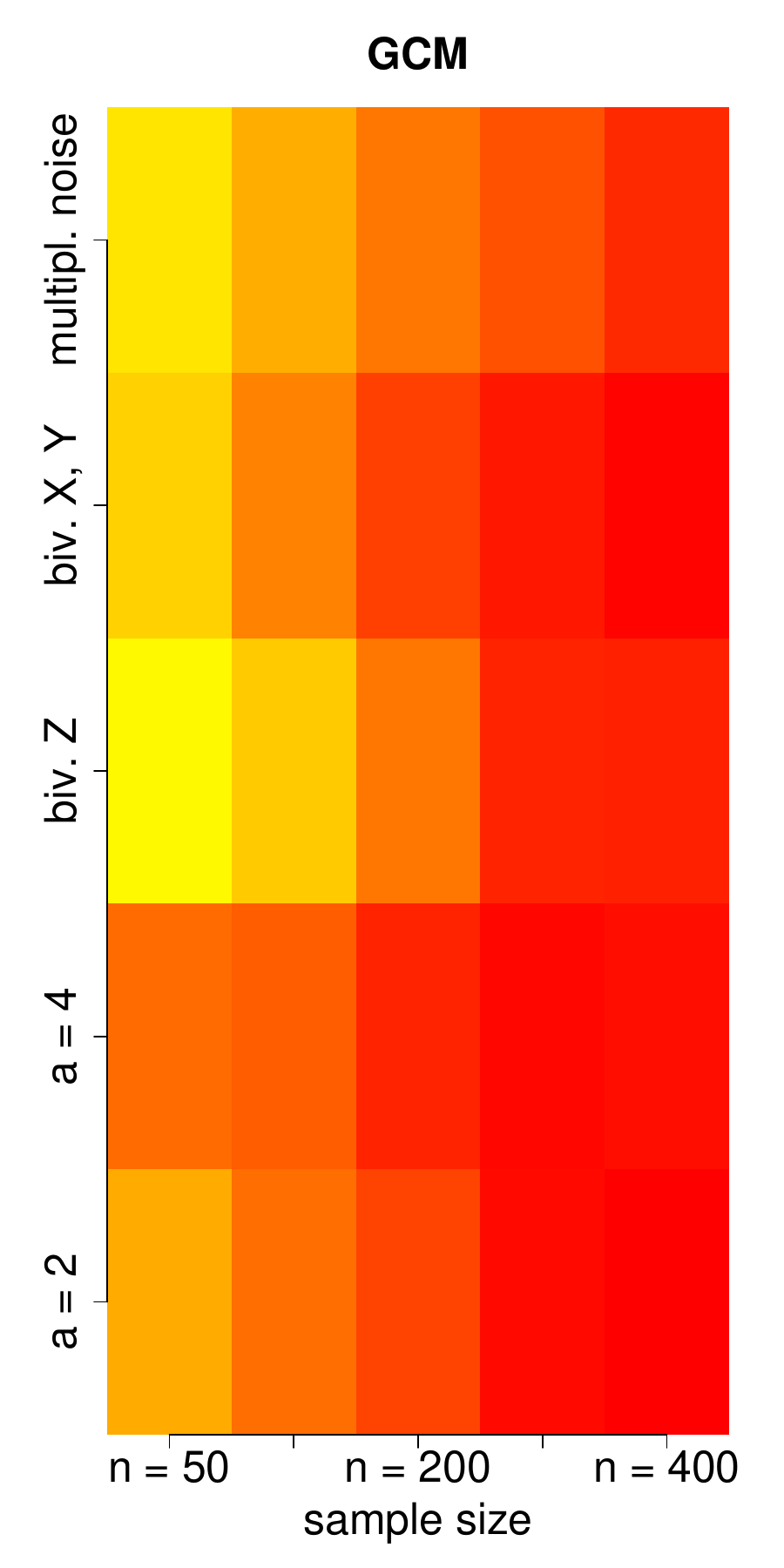}
\hfill
\raisebox{0.5\height}{
\includegraphics[width = 0.1\textwidth]{experiment3mult-colorbar-cut}
}
}
\caption{\label{fig:exp:exp3power}Power analysis: most of the methods are able to detect if the distribution does not satisfy conditional independence, in particular if the sample size increases. 
}
\end{figure}
All methods, except for RPT, 
are able to correctly reject the hypothesis that the distribution is from the null, particularly with increasing sample size. 
In our experimental setup,
it is the level analysis, 
that poses a greater challenge for the methods other than GCM.

\section{Discussion} \label{sec:discussion}
A key result of this paper is that conditional independence testing is hard: non-trivial tests that maintain valid level over the entire class of distributions satisfying conditional independence and that are absolutely continuous with respect to Lebesgue measure cannot exist.
In unconditional independence testing, control of type I error is straightforward and research efforts have focussed on power properties of tests. Our result indicates that in conditional independence testing, the basic requirement of type I error control deserves further attention.
We argue that as domain knowledge is necessary in order to select a conditional independence test appropriate for a particular setting, there is a need to develop conditional independence tests whose suitability is reasonably straightforward to judge.

In this work we have introduced the GCM framework to address this need.
The ability for the GCM to maintain the correct level relies almost exclusively on the predictive properties of the regression procedures upon which it is based. Selecting a good regression procedure, whilst mathematically an equally impossible problem, can at least be usefully informed by domain knowledge. We hope to see further applications of GCM-based tests in the future. %
On the theoretical side, it would be interesting to understand more precisely the tradeoff between the type I and type II errors in conditional independence testing. Often, work on testing fixes a null and then considers what sorts of classes of alternative distributions it is possible, or impossible to maintain power against. In the context of conditional independence testing, the problem set is even richer, in that one must also consider subclasses of null distributions, and can then study power properties associated with that null.

\section*{Acknowledgements}
We thank Kacper Chwialkowski, Kenji Fukumizu, Arthur Gretton, Bernhard Sch\"olkopf, and Ilya Tolstikhin for helpful discussions, initiated by BS, on the hardness of conditional independence testing, and KC and AG for helpful comments on the manuscript.
BS has raised the point that for finitely many data conditional independence testing may be arbitrarily hard in several of his talks, e.g., at the Machine Learning Summer School in T\"ubingen in 2013.
We are very grateful to Matey Neykov, Anton Lundborg and Cyrill
Scheidegger for kindly pointing out some errors in earlier versions of this manuscript, and suggesting potential fixes.
We also thank Peter B\"uhlmann for helpful discussions regarding the aggregation of tests via taking the maximum test statistic. Finally, we thank four anonymous referees and an associate editor for helpful comments that have improved the manuscript.
\appendix

\section{Proof of Theorem~\ref{THM:NFL}}
The proof of Theorem~\ref{THM:NFL} relies heavily on Lemma~\ref{lem:cond_ind_approx} in Section~\ref{sec:aux_lem}, which 
shows that given any distribution $Q$ where $(X, Y, Z) \sim Q$, one can construct $(\tilde{X}, \tilde{Y}, \tilde{Z})$ with $\tilde{X} \independent \tilde{Y} \given \tilde{Z}$ where $(\tilde{X}, \tilde{Y}, \tilde{Z})$ and $(X, Y, Z)$ are arbitrarily close in $\ell_\infty$-norm with arbitrarily high probability.

In the proofs of Theorem~\ref{THM:NFL} and Lemma~\ref{lem:cond_ind_approx} below, we often suppress dependence on $n$ to simplify the presentation. Thus for example, we write $\mb X$ for $\mb X^{(n)}$.
We use the following notation. We write $s=(d_X + d_Y + d_Z)$ and will denote by $V \in \R^s$ the triple $(X, Y, Z)$. Furthermore, $\mb V := (\mb X, \mb Y, \mb Z)$. We denote by $p_{X,Y,Z}$ the density of $(X, Y, Z)$ with respect to Lebesgue measure. We will use $\mu$ to denote Lebesgue measure on $\R^{ns+1}$ and write $\triangle$ for the symmetric difference operator.

\subsection{Proof of Theorem~\ref{THM:NFL}}
Suppose, for a contradiction, that there exists a $Q$ with support strictly contained in an $\ell_\infty$-ball of radius $M$ under which $X \notindependent Y \given Z$ but 
$\prob_Q(\psi_n(\mb V; U) = 1) = \beta > \alpha$.
We will henceforth assume that $V \sim Q$ and $\mb V := (\mb X, \mb Y, \mb Z)$ are i.i.d.\ copies of $V$. Thus we may omit the subscript $Q$ applied to probabilities and expectations in the sequel.
Denote the rejection region by
\[
R = \{(\mb x, \mb y, \mb z; u) \in \R^{ns} \times [0,1] :\psi_n(\mb x, \mb y, \mb z; u)=1\}.
\]
Our proof strategy is as follows. Using Lemma~\ref{lem:cond_ind_approx} we will create $\tilde{V} := (\tilde{X}, \tilde{Y}, \tilde{Z})$ such that  $\tilde{X} \independent \tilde{Y} \given \tilde{Z}$ but $\tilde{V}$ is suitably close to $V$ such that a corresponding i.i.d.\ sample $\tilde{\mb V} := (\tilde{\mb X}, \tilde{\mb Y}, \tilde{\mb Z})  \in \mathbb{R}^{ns}$ satisfies $\prob((\tilde{\mb V}, U) \in R) > \alpha$, 
contradicting that $\psi_n$ has valid level $\alpha$.
How close $\tilde{V}$ needs to be to $V$ in order for this argument to work depends on the rejection region $R$. As an arbitrary Borel subset of $\R^{ns} \times [0,1]$, $R$ can be arbitrarily complex. In order to get a handle on it we will construct an approximate version $R^{\sharp}$ of $R$ that is a finite union of boxes; see Lemma~\ref{lem:borel_approx}.

Let $\eta = (\beta - \alpha)/7 >0$.
Since $\{(x, y, z) : p_{X,Y,Z}(x, y, z) > m\} =:B_m \downarrow \emptyset$ as $m \uparrow \infty$, there exists $M_1$ such that $\pr((X, Y, Z) \in B_{M_1}^c) > 1-\eta/n$. Let $\Omega_1$ be the event that $(x_i, y_i, z_i) \in B_{M_1}^c$ 
for all $i=1,\ldots,n$. 
(Here and below, an event refers to an element in the underlying $\sigma$-algebra. Recall that
$x_i$, $y_i$, and $z_i$ denote 
rows of $\mb X$, $\mb Y$, and $\mb Z$, respectively, i.e., they are random vectors.)
Then by a union bound we have 
$\pr(\Omega_1) \geq 1 - \eta$.

Let $M_2$ be such that $\prob(\|\mb V\|_\infty > M_2) < \eta$ and let $\Omega_2$ be the event that $\|\mb V\|_\infty \leq M_2$.
Further define
\[
\check{R} = \{(\mb x, \mb y, \mb z, u) \in R: \|(\mb x, \mb y, \mb z)\|_\infty \leq M_2\}.
\]
Note that
\begin{align} \label{eq:check_R_bd}
\pr((\mb V, U) \in \check{R}) \geq \beta - \pr((\mb V, U) \in R\setminus\check{R}) > \beta - \eta.
\end{align}

Let $L=L(\eta)$ be as defined in Lemma~\ref{lem:cond_ind_approx} (taking $\delta = \eta$).
From Lemma~\ref{lem:borel_approx}
applied to $\check{R}$, we know there exists a finite union $R^{\sharp}$ of hypercubes each of the form
\[
\prod_{k=1,\ldots, ns+1} (a_k, b_k]
\]
such that $\mu(R^{\sharp} \triangle \check{R}) < \eta / \max(L, M_1^n)$. Now on the region $B_{M_1}^c$ defining $\Omega_1$ we know that the density of $(\mb V, U)$ is bounded above by $M_1^n$. Thus we have
that
\begin{equation} \label{eq:R_sharp_bd}
\pr(\{(\mb V, U) \in \check{R} \setminus R^{\sharp}\} \cap \Omega_1) < \eta.
\end{equation}

Now for $r \geq 0$ and $\mb v \in \R^{ns+1}$ let $B_r(\mb v)\subset \R^{ns+1}$ denote the $\ell_\infty$ ball with radius $r>0$ and center $\mb v$. Define
\[
R^r = \{\mb v \in R : B_r(\mb v) \subseteq R^{\sharp}\}.
\]
Then since $R^r \uparrow R^{\sharp}$ 
as $r \downarrow 0$, there exists $r_0 > 0$ such that $\mu(R^{\sharp} \setminus R^{r_0}) < \eta / M_1^n$.

For $\epsilon=r_0$ and $B=R^{\sharp} \setminus \check{R}$, 
the statement of Lemma~\ref{lem:cond_ind_approx} 
provides us 
with $\tilde{\mb V} := (\tilde{\mb X}, \tilde{\mb Y}, \tilde{\mb Z})$ which 
satisfies
$\pr((\tilde{\mb V}, U) \in R^{\sharp} \setminus \check{R}) < L \mu(R^{\sharp} \setminus \check{R})< \eta$ 
and 
with which we argue as follows. Let $\Omega_3$ be the event that $\|\mb V - \tilde{\mb V}\|_\infty < r_0$, so 
$\pr(\Omega_3) \geq 1 - \eta$.
\begin{align*}
\pr((\tilde{\mb V}, U) \in R) &\geq \pr((\tilde{\mb V}, U) \in \check{R}) 
\geq \pr((\tilde{\mb V}, U) \in R^{\sharp}) - \pr((\tilde{\mb V}, U) \in R^{\sharp} \setminus \check{R}) \\
&> \pr(\{(\tilde{\mb V}, U) \in R^{\sharp}\} \cap \Omega_3) - \eta 
> \pr((\mb V, U) \in R^{r_0}) - 2\eta \\
&\geq \pr((\mb V, U) \in R^{\sharp}) - \pr(\{(\mb V, U) \in R^{\sharp} \setminus R^{r_0}\} \cap \Omega_1) - \pr(\Omega_1^c) - 2\eta\\
&> \pr((\mb V, U) \in R^{\sharp}) - 4\eta.
\end{align*}
Now
\begin{align*}
\pr((\mb V, U) \in R^{\sharp}) &\geq \pr((\mb V, U) \in \check{R}) - \pr(\{(\mb V, U) \in \check{R} \setminus R^{\sharp}\} \cap \Omega_1) - \pr(\Omega_1^c) \\ 
&> \pr((\mb V, U) \in \check{R}) - 2\eta > \beta - 3\eta
\end{align*}
using \eqref{eq:R_sharp_bd} and \eqref{eq:check_R_bd}. Putting things together, we have
$\pr((\tilde{\mb V}, U) \in R) > \beta - 7\eta > \alpha$, 
completing the proof.

\subsection{Auxilliary Lemmas} \label{sec:aux_lem}
\begin{lemma} \label{lem:cond_ind_approx}
Let $(X, Y, Z)$ have a 
$(d_X + d_Y + d_Z)$-dimensional
distribution in $\mathcal{Q}_{0, M}$ for some  $M \in (0, \infty]$. Let $(\mb X^{(n)}, \mb Y^{(n)}, \mb Z^{(n)})$ be a sample of $n$ i.i.d.\ copies of $(X, Y, Z)$.
Given $\delta >0$, there exists $L=L(\delta)$ such that for all $\epsilon > 0$ and all Borel subsets $B \subseteq \R^{n \cdot (d_X + d_Y + d_Z)} \times [0,1]$, it is possible to construct 
$n$ i.i.d.\ random vectors $(\tilde{\mb X}^{(n)}, \tilde{\mb Y}^{(n)}, \tilde{\mb Z}^{(n)})$
with 
distribution $P \in \mathcal{P}_{0, M}$ where the following properties hold:
\begin{enumerate}[(i)]
\item $\pr(\|(\mb X^{(n)}, \mb Y^{(n)}, \mb Z^{(n)}) - (\tilde{\mb X}^{(n)}, \tilde{\mb Y}^{(n)}, \tilde{\mb Z}^{(n)})\|_\infty < \epsilon) > 1- \delta$;
\item If $U \sim U[0, 1]$ independently of $(\tilde{\mb X}^{(n)}, \tilde{\mb Y}^{(n)}, \tilde{\mb Z}^{(n)})$ then
\[
\pr((\tilde{\mb X}^{(n)}, \tilde{\mb Y}^{(n)}, \tilde{\mb Z}^{(n)}, U) \in B) \leq L \mu(B).
\]
\end{enumerate}
\end{lemma}
\begin{proof}
We will first describe the construction of $\tilde{V} := (\tilde{X}, \tilde{Y}, \tilde{Z})$ from $V:=(X, Y, Z)$. The corresponding $n$-sample $\tilde{\mb V} := (\tilde{\mb X}^{(n)}, \tilde{\mb Y}^{(n)}, \tilde{\mb Z}^{(n)})$ will have observation vectors formed in the same way from the corresponding observation vectors in $\mb V$.
The proof proceeds in three steps. We begin by creating a bounded version $\check{V}=(\check{X}, \check{Y}, \check{Z})$ of $V$ supported on a grid $2^{-r}\mathbb{Z}$,
for which we can control an upper bound on the probability mass function. Next, we apply Lemma~\ref{lem:hiding2} to obtain transforms $\checknew{V}^{(1)}, \ldots, \checknew{V}^{(K!)}$ of $\check{V}$ for arbitrarily large $K$ where $\check{X}$ has been `embedded' in the last component. Then we create noisy versions $\{\tilde{V}^{(m)}\}_{m=1}^{K!}$ by adding uniform noise such that truncation of their binary expansions yields the discrete versions $\{\checknew{V}^{(m)}\}_{m=1}^{K!}$. Each of these are potential candidates for the random vector $\tilde{V}$, but we must ensure that the corresponding $n$-fold product obeys (ii). This is problematic as the embedding procedure necessarily creates near-degenerate random vectors that fall within small regions with large probability. To overcome this issue, we employ in the final step, a probabilistic argument that exploits the property, supplied by Lemma~\ref{lem:hiding2}, that the $K!$ embeddings have supports with little overlap.

\emph{Step 1:}
Define $s := d_X + d_Y + d_Z$.
Since $\{(x, y, z) \in \R^s : p_{X,Y,Z}(x, y, z) > t\} =:B_t \downarrow \emptyset$ as $t \uparrow \infty$, there exists $M_1$ such that the event $\Lambda_1 = \{(X, Y, Z) \in B_{M_1}^c\}$ has 
$\prob(\Lambda_1) \geq 1 - \delta/(2n)$. 
Next, let $M_2 < M$ be such that $\pr(\|V\|_\infty > M_2) < \delta /(2n)$, and let $\Lambda_2$ be the event that $\|V\|_\infty \leq M_2$. For later use, we define  the events
\begin{align*} %
\Omega_1 =\{(x_i, y_i, z_i) \in B_{M_1}^c \text{ for all } i=1,\ldots,n\} \quad \text{and} \quad \Omega_2 = \{\|\mb V\|_\infty \leq M_2\}.
\end{align*}
Note that union bounds give $\pr((\Omega_1 \cap \Omega_2)^c) < \delta$.

Let $E^{(1)}$ be uniformly distributed on $[-M_2,M_2]^{s}$. 
Let $r \in \mathbb{N}$ be such that 
$2^{-r} < \min(\epsilon/3, (M - M_2)/3, 1/n)$ 
and define
\[
\check{V} := (\check{X}, \check{Y}, \check{Z}) :=2^{-r}\floor{2^r(V\ind_{\Lambda_1 \cap \Lambda_2}  + E^{(1)}\ind_{(\Lambda_1 \cap \Lambda_2)^c})}.
\]
Here, the floor function is 
applied
componentwise.
Note that $\check{V}$ takes values in a grid $(2^{-r} \mathbb{Z})^s$ and satisfies
\begin{equation} \label{eq:check_V_l_infty_bd}
\|(\check{V} - V) \ind_{\Lambda_1 \cap \Lambda_2}\|_\infty \leq 2^{-r} < \epsilon/3.
\end{equation}
The choice of $r$ ensures that $\check{V} \in(-M', M')^s$ where $M' = M - 2(M-M_2)/3$. Furthermore, the inclusion of the 
$\ind_{\Lambda_1 \cap \Lambda_2}$ 
term and $E^{(1)}$ ensures that the probability it takes any given value is bounded above by $M_3 2^{-sr}$ where
$M_3:= \max(M_1, (M_2/2)^{-s})$ is independent of $\epsilon$. 
Indeed for any fixed $k \in \mathbb{Z}^s$, writing $A=[k2^{-r},(k+1)2^{-r})$ we have 
\[
\prob(V \in A| \Lambda_1 \cap \Lambda_2) \leq M_12^{-rs} \;\;\text{ and }\;\;\prob(E^{(1)} \in A | (\Lambda_1 \cap \Lambda_2)^c) = 2^{-rs}/(2M_2)^s.
\]
As $\prob(\check{V} = k 2^{-r})$ is a convex combination of these probabilities, it must be at most their maximum.

\emph{Step 2:}
We can now apply Lemma~\ref{lem:hiding2} with $W=(\check{Y}, \check{Z})$ and $N=\check{X}$. This gives us $K!$ 
random vectors $\checknew{V}^{(1)},\ldots, \checknew{V}^{(K!)}$ where $K > 2^r>n$; 
for each $m=1,\ldots,K!$, 
$\checknew{V}^{(m)} = (\checknew{X}^{(m)}, \checknew{Y}^{(m)}, \checknew{Z}^{(m)})$ satisfies
\begin{enumerate}[(a)]
\item $\pr(|\checknew{V}^{(m)}_s - \check{V}_s| \leq 2^{-r}) = 1$ and $\checknew{V}^{(m)}_j = \check{V}_j$ for $j \leq s-1$;
\item $\checknew{X}^{(m)}$ may be recovered from $\checknew{Z}^{(m)}$ via $\checknew{X}^{(m)} = \checknew{g}_m(\checknew{Z}^{(m)})$ for some function $\checknew{g}_m$;
\item $\checknew{V}^{(m)}_s$ takes values in $K^{-2}2^{-r}\mathbb{Z}$ and the probability it takes any given value is bounded above by $2^{-sr}K^{-1} M_3$;
\end{enumerate}
and, additionally,
\begin{enumerate}[(a)]
\item[(d)] 
the supports $\checknew{S}_1,\ldots,\checknew{S}_{K!}$ 
of $\checknew{V}^{(1)}, \ldots, \checknew{V}^{(K!)}$
obey the following structure:
there exists a collection of  $K^2$ disjoint sets $\{\checknew{G}_{jk}\}_{j,k=1}^K$ and an enumeration $\pi_1, \ldots, \pi_{K!}$ of the permutations of $\{1,\ldots, K\}$ \\
such that 
$\checknew{S}_{m} = \cup_{k=1}^K \checknew{G}_{k\pi_m(k)}$ 
and $\pr(\checknew{V}^{(m)} \in \checknew{G}_{j \pi_m(j)}) = K^{-1}$ for all $m=1,\ldots,K!$ and $j=1,\ldots,K$. 
\end{enumerate}

We now create a noisy version of the $\checknew{V}^{(m)}$ that obeys similar properties to the above, but is absolutely continuous with respect to Lebesgue measure. To this end, we introduce $E^{(2)}=(E_X, E_Y, E_Z) \in (0,1)^s$ with independent $U(0,1)$ components. Then let $\tilde{V}^{(m)} \in \R^s$ be defined by
\[
\tilde{V}^{(m)}_j= \begin{cases}
K^{-2} 2^{-r} E^{(2)}_s + \checknew{V}^{(m)}_s \; \text{ for } j=s\\
2^{-r} E^{(2)}_j + \checknew{V}^{(m)}_j \; \text{ otherwise.}
\end{cases}
\]
This obeys
\begin{enumerate}[(a')]
\item $\pr(\|\tilde{V}^{(m)} - \checknew{V}^{(m)}\|_\infty \leq 2^{-r}) = 1$;
\item 
$\checknew{X}^{(m)}$ may be recovered from $\tilde{Z}^{(m)}$ 
via $\tilde{X}^{(m)} = g_k(\tilde{Z}^{(m)})$ 
for some function $g_m$, which depends on $K$ and $r$;
\item the density of $\tilde{V}^{(m)}$ with respect to Lebesgue measure is bounded above by $KM_3$
(indeed, using (c), we have that it is bounded by 
$2^{rs} K^2 \cdot 2^{-sr}K^{-1} M_3 = KM_3$);
\item the supports $S_1,\ldots,S_{K!}$ of the $\{\tilde{V}^{(m)}\}_{m=1}^{K!}$ obey property (d) with the disjoint sets $\checknew{G}_{jk}$ above replaced by the Minkowski sum
$G_{jk}=\checknew{G}_{jk} + 2^{-r}((0,1)^{s-1} \times (0, K^{-2}))$.
\end{enumerate}

Note that (b') holds as 
we can 
first construct 
$\checknew{Z}^{(m)}$
from
$\tilde{Z}^{(m)}$
and then apply (b).
The former is done 
by removing the additive noise component by
truncating the binary expansion appropriately:
$\checknew{Z}^{(m)} := K^{-2} 2^{-r} \floor{K^2 2^r \tilde{Z}^{(m)}}$.
A consequence of this property is that decomposing $(\tilde{X}^{(m)}, \tilde{Y}^{(m)}, \tilde{Z}^{(m)})=\tilde{V}^{(m)}$, we have $\tilde{X}^{(m)} \independent \tilde{Y}^{(m)} \given \tilde{Z}^{(m)}$. To see this we argue as follows. Let us write $p_A$ and $p_{A |B}$ for the densities of $A$ and $A$ given $B$ respectively when $A$ and $B$ are random vectors. Suppressing dependence on $m$ temporarily, we have that for any $\tilde{z}$ with $p_{\tilde{Z}}(\tilde{z}) > 0$,
\begin{align*}
p_{\tilde{X},\tilde{Y},\tilde{Z}}(\tilde{x},\tilde{y} | \tilde{z}) 
&= p_{E_X, \tilde{Y}|\tilde{Z}}(\tilde{x} - g(\tilde{z}),\tilde{y} | \tilde{z}) 
= p_{E_X}(\tilde{x} - g(\tilde{z})) \, p_{\tilde{Y}|\tilde{Z}}(\tilde{y} | \tilde{z}) \\
&= p_{\tilde{X} | \tilde{Z}}(\tilde{x} |\tilde{z}) \, p_{\tilde{Y}|\tilde{Z}}(\tilde{y} | \tilde{z}),
\end{align*}
so $\tilde{X} \independent \tilde{Y} \given \tilde{Z}$.

Property (d') follows as the support of each $\checknew{V}^{(m)}$ is contained in $2^{-r}(\mathbb{Z}^{s-1} \times K^{-2}\mathbb{Z})$ (see (a) and (c)).

From (a'), by the triangle inequality we have that
\begin{equation} \label{eq:V_l_inf_bd}
\pr(\|(\tilde{V}^{(m)} - V)\ind_{\{\Lambda_1 \cap \Lambda_2\}}\|_\infty \leq \epsilon) = 1
\end{equation}
and $\tilde{V}^{(m)} \in (-M, M)^s$.
Let $\{\tilde{\mb V}^{(m)}\}_{m=1}^{K!}$ be the corresponding $n$-sample versions of $\{\tilde{V}^{(m)}\}_{m=1}^{K!}$. Then for any $m$, \eqref{eq:V_l_inf_bd} gives 
$\pr(\|(\tilde{\mb V}^{(m)} - \mb V)\ind_{\Omega_1 \cap \Omega_2}\|_\infty \leq \epsilon) = 1$. 
Thus $\pr(\|\tilde{\mb V}^{(m)} - \mb V\|_\infty \leq \epsilon) > 1-\delta$.
We see that any $\tilde{\mb V}^{(m)}$ satisfies all requirements of the result except potentially (ii).

\emph{Step 3:}
In order to pick an $m$ for which (ii) is satisfied, we use the so-called probabilistic method. First note we may assume $0<\mu(B) <\infty$ or otherwise any $m$ will do.
Define $T_m := S_m^n \times [0,1]$ to be the support set of $(\tilde{\mb V}^{(m)}, U)$ where $U \sim U[0,1]$ independently of $\{\tilde{\mb V}^{(m)}\}_{m=1}^{K!}$.

For $j \in \{1, \ldots, K\}$
let $G_j := \cup_{k=1}^K G_{jk}$. 
Let $\mathcal{J}$ be the set of $n$-tuples $(j_1,\ldots,j_n)$ of distinct elements of $\{1,\ldots,K\}$. Now define
\begin{equation}\label{eq:defineC}
C := \bigcup_{(j_1,\ldots,j_n) \in \mathcal{J}} \prod_{l=1}^n G_{j_l}
\end{equation}
and let $D = C \times [0,1]$. 
Fix $(j_l)_{l=1}^n \in \mathcal{J}$ and $(k_l)_{l=1}^n \in \{1,\ldots,K\}^n$, and set $G:=\prod_{l=1}^n G_{j_l k_l}$. Then $G$ has non-empty intersection with a given $S_m^n$ if and only if $k_l = \pi_m(j_l)$ for all $l$. Thus if $(k_l)_{l=1}^n \notin \mathcal{J}$, $G$ is disjoint from all $S_m^n$. 
On the other hand if $(k_l)_{l=1}^n \in \mathcal{J}$, the number of $S_m^n$ that intersect $G$ is $(K-n)!$, the number of permutations of $\{1,\ldots,K\}$ whose outputs are fixed at $n$ points.
We therefore have that all but at most $(K-n)!$ of the support sets $T_m$ are disjoint from $G \times [0, 1]$, whence
\[
\sum_{m=1}^{K!} \mu\left\{(G \times [0,1]) \cap B \cap T_m\right\} \leq (K-n)! \mu\left\{(G\times [0,1]) \cap B \right\}.
\]
Now the set $C$ is the disjoint union of all sets $\prod_{l=1}^n G_{j_l k_l}$ with $(j_l)_{l=1}^n \in \mathcal{J}$ and $(k_l)_{l=1}^n \in \{1,\ldots,K\}^n$. Thus summing over all such sets we obtain
\[
\sum_{m=1}^{K!} \mu(D \cap B \cap T_m) \leq (K-n)!\mu(D \cap B) \leq (K-n)! \mu(B).
\]
This gives that there exists at least one $m=m^*$ with
\begin{equation*}
\mu(B \cap D \cap T_{m^*}) \leq \frac{(K-n)!}{K!} \mu(B).
\end{equation*}

Next, observe that the number of cells $\prod_{l=1}^n G_{k_l}$ where at least two of $k_1,\ldots,k_n$ are the same is $K^n-K(K-1)\cdots(K-n+1)$. 
As 
$\pr(\tilde{V}^{(m)} \in G_j)=K^{-1}$ for all $j$, using~\eqref{eq:defineC} we have
\begin{align*}
\pr((\tilde{\mb V}^{(m)},U) \notin D) = K^{-n} \{K^n-K(K-1)\cdots(K-n+1)\} = O(K^{-1}).
\end{align*}
for every $m$.
Putting things together, we have that there must exist an $m^*$ with
\begin{align*}
\pr((\tilde{\mb V}^{(m^*)}, U) \in B) &\leq \pr((\tilde{\mb V}^{(m^*)}, U) \in B \cap D) + \pr ((\tilde{\mb V}^{(m^*)}, U) \notin D) \\
&\leq K^n M_3^n \frac{(K-n)!}{K!} \mu(B)  + O(K^{-1}) \leq 2 M_3^n \mu(B)
\end{align*}
for $K$ sufficiently large, which can be arranged by taking $r$ sufficiently large.
\end{proof}

\begin{lemma}\label{lem:hiding2}
Let $W \in \R^l$, $N \in \R^d$ be random vectors. 
Suppose that $N$ is bounded and 
there is some $r \in \mathbb{N}$
such that both $W$ and $N$ have components taking values in the grid $2^{-r} \mathbb{Z}$. 
Suppose further that the probability that $(N, W)$ takes any particular value is bounded by $2^{-(m+d)r} M$ for some $M>0$.
Then there exists a 
$K \in \mathbb{N}$ with $K > 2^r$ and a
set of $K!$ functions $\{f_1,\ldots,f_{K!}\}$ where
\[
(W_l, N) \; \mapsto \; f_m(W_l, N) =: \checknew{W}_l^{(m)} \in \R
\]
such that for each $m=1,\ldots,K!$,
\begin{enumerate}[(i)]
\item $\prob(|W_l - \checknew{W}_l^{(m)} | \leq 2^{-r}) = 1$;
\item 
there is some function $g_m : \R \to \R^d$
such that
$N = g_m(\checknew{W}_l^{(m)})$;
\item 
$\checknew{W}_l^{(m)}$
has components taking values in a grid $K^{-2} 2^{-r} \mathbb{Z}$.
\end{enumerate}
Moreover, defining $\checknew{W}^{(m)} = (W_1, \ldots, W_{l-1}, \checknew{W}_l^{(m)})$,
\begin{enumerate}[(i)]
\item[(iv)] the probability that $(N, \checknew{W}^{(m)})$ takes any value is bounded above by $K^{-1}2^{-(m+d)r} M$;
\item[(v)] the supports $S_1,\ldots,S_{K!}$ of $(N,\checknew{W}^{(1)}), \ldots, (N,\checknew{W}^{(K!)})$ obey the following structure: there exists $K^2$ disjoint sets $\{G_{jk}\}_{j,k=1}^K$ and an enumeration $\pi_1,\ldots,\pi_{K!}$ of the permutations of $\{1,\ldots,K\}$ such that 
for all $m$, 
$S_m= \cup_{k=1}^K G_{k\pi_m(k)}$, and 
for all 
$m$ and $k$, 
$\pr((N, \checknew{W}^{(m)}) \in G_{k\pi_m(k)}) = K^{-1}$.
\end{enumerate}
\end{lemma}
\begin{proof}
As $N$ is bounded, by replacing $g_m(\cdot)$ by $g_m(\cdot) + v$ where $v \in \R^d$ is appropriately chosen with components in $2^{-r} \mathbb{Z}$, we may assume that $N$ has non-negative components. Let $t \in \mathbb{N}$ be such that $2^{t} > 2^r \max(1,\|N\|_\infty)$.
We shall prove the result with $K=2^{dt}$.

Define the random variable $\checknew{N}$
by
\[
\checknew{N} = 2^{r} \sum_{j=0}^{d-1} 2^{tj}N_{j+1}.
\]
This is a concatenation of the binary expansions 
of $2^r N_j \in \{0,1, 2, \ldots,2^t-1\}$
for $j=1,\ldots,d$. 
Observe that $\checknew{N} \in \{0, 1, \ldots, K-1\}$ and 
that
$N_j$ may be recovered from $\checknew{N}$ by examining its binary expansion. Indeed, $2^r N_j$ is the residue modulo $2^{t}$ of $\floor{\checknew{N} / 2^{t(j-1)}}$.

For $j=\{0,1, \ldots, K-1\}$, let $\tilde{N}_j$ be the residue
of $\checknew{N} + j$ modulo $K$, 
so $\tilde{N}_j \in \{0, 1, \ldots, K-1\}$. Also, for $k=\{0,1, \ldots, K-1\}$ let $\tilde{N}_{j,k} = \tilde{N}_j + Kk$. Note that $\tilde{N}_{j,k}$ takes values in $\{0, \ldots, K^2-1\}$.
Let the random variable $E$ be uniformly distributed on $\{0,1,\ldots,K-1\}$ independently of all other quantities.
Now let $\pi_1,\ldots,\pi_{K!}$ be an enumeration of the permutations of $\{0,\ldots,K-1\}$. Finally, let $\checknew{N}_m = \tilde{N}_{E,\pi_m(E)}$ for $m=1,\ldots,K!$.

One important feature of this construction is that we can recover $E$ from $\checknew{N}_m$
(and $m$)
via 
$\pi_m(E)= \floor{\checknew{N}_m/K}$, 
and thereby determine 
$E$,
which then reveals $\checknew{N}$ 
and each of the individual $N_j$. 
In summary, this gives us $K!$ different embeddings of the vector $N$ into a single random variable.

We may now define $f_m$ by
\[
f_m(W_l, N) = W_l + 2^{-r} K^{-2} \checknew{N}_m.
\]
It is easy to see that (i) and (iii) are satisfied. To deduce (ii), observe that we may recover $W_l$ via $2^{-r}\floor{2^r f_m(W_l, N)}=:c_m(f_m(W_l, N))$, and thus also determine $\checknew{N}_m$ which, as discussed above, also gives us $N$ and $E$. Let $g_m$ and $h_m$ be the functions that when applied to $f_m(W_l,N)$, yield $N$ and $E$ respectively. Let us introduce the notation that for a vector $v \in \R^s$, $v_{-j} \in \R^{s-1}$ for $j=1,\ldots,s$ is the subvector of $v$ where the $j$th component is omitted. Then we have
\begin{align*}
&\pr(N=n, \checknew{W}^{(m)} = \checknew{w}) \\
&\qquad = \pr(W_{-l} =\checknew{w}_{-l}, W_l=c_m(\checknew{w}_l), N=n, E=h_m(\checknew{w}_l))\ind_{\{g_m(\checknew{w}_l)=n\}} \\
&\qquad = K^{-1} \pr(W_{-l} =\checknew{w}_{-l}, W_l=c_m(\checknew{w}_l), N=n)\ind_{\{g_m(\checknew{w}_l)=n\}} \\
&\qquad \leq K^{-1} 2^{-(l+d)r}M,
\end{align*}
using the independence of $E$ in the second line above. This gives (iv).

Note that the supports of the $(N, \tilde{N}_{j,k})$ are all disjoint as $(j,k)$ can be recovered from $(N, \tilde{N}_{j,k})$. For $j,k=0,1,\ldots,K-1$ let $G_{jk}$ be the support set of
\[
(N, \tilde{W}_{j,k}) := (N, W_1, \ldots, W_{l-1}, W_l + 2^{-r}K^{-2} \tilde{N}_{j,k}).
\]
From the above, we see that the $\{G_{jk}\}_{j,k=0}^{K-1}$ are all disjoint. Property (v) follows from noting that 
$\checknew{W}^{(m)} = \tilde{W}_{E,\pi_m(E)}$.
\end{proof}

The following well-known result appears for example in \citet[][Theorem~2.19]{weaver2013measure}.
\begin{lemma} \label{lem:borel_approx}
Given any bounded Borel subset $B$ of $\R^d$ and any $\epsilon > 0$, there exists a finite union of boxes of the form
\[
B^{\sharp} = \bigcup_{i=1}^N \prod_{k=1}^d (a_{i,k}, b_{i,k}]
\]
such that $\mu(B \triangle B^{\sharp}) \leq \epsilon$, where $\mu$ denotes Lebesgue measure and $\triangle$ denotes the symmetric difference operator.
\end{lemma}

\section{KL-Separation of null and alternative} \label{SEC:KL}
For distributions $P_1, P_2 \in \mathcal{E}_0$, let $\mathrm{KL}(P_1||P_2)$ denote the KL divergence from $P_2$ to $P_1$, which we define to be $+\infty$ when $P_1$ is not absolutely continuous with respect to $P_2$.
The following result, which is proved in the supplementary material,
shows that it is possible to choose $Q \in \mathcal{Q}_0$ that is arbitrarily far away from $\mathcal{P}_0$ (by picking $\sigma^2>0$ to be sufficiently small).
\begin{proposition} \label{prop:KL}
Consider the distribution $Q$ over the triple $(X,Y,Z)\in \R^3$ defined in the following way:  
$Y = X + N$ with 
$X \sim \mathcal{N}(0,1)$, 
$N \sim \mathcal{N}(0,\sigma^2)$ and $X \independent N$. The variable $Z\sim \mathcal{N}(0,1)$ is independent of $(X,Y)$.
Thus $X \notindependent Y \given Z$, i.e., $Q \in \mathcal{Q}_0$, and we have
\[
\inf_{P \in \mathcal{P}_0} \mathrm{KL}(Q||P) 
= \frac{1}{2} \log \bigg(\frac{1+\sigma^2}{\sigma^2}\bigg)
= \inf_{P \in \mathcal{P}_0} \mathrm{KL}(P||Q).
\]
\end{proposition}

\end{cbunit}

\begin{cbunit}
\newpage

\noindent {\Large \bf Supplementary material}\\

This supplementary material contains the proofs of results omitted in the appendices A and B in the main paper.
\section{Proof of additional results in Sections~\ref{SEC:NOFREELUNCH} and \ref{SEC:KL}}

\subsection{Results on the separation between the null and alternative hypotheses}
In this section we provide proofs of Propositions~\ref{prop:separation} and~\ref{prop:KL}.

\subsubsection{Proof of Proposition~\ref{prop:separation}}
The proof of Proposition~\ref{prop:separation} makes frequent use of the following lemma.
\begin{lemma} \label{lem:tv}
Let probability measures $P$ and $Q$ be defined on $[-M, M]$. Then
\[
|\E_P W - \E_Q W| \leq 2M \tv{P - Q}.
\]
\end{lemma}
\begin{proof}
First note that
\[
\E W = \int_0^{2M} \pr(W +M \geq t) dt - M.
\]
Thus
\begin{align*}
\E_P W - \E_Q W = \int_0^{2M} \pr_P(W +M \geq t)  - \pr_Q(W+M \geq t) dt \leq 2M \tv{P - Q}.
\end{align*}
Repeating the argument with $P$ and $Q$ interchanged gives $\E_Q W - \E_P W \leq 2M \tv{P - Q}$ and hence the result.
\end{proof}

We now turn to the proof of Proposition~\ref{prop:separation}.
By scaling we may assume without loss of generality that $M=1$. We will frequently use the following fact. Let $P_1, P_2$ be two probability laws on $\R^d$ and suppose $U \sim P_1$ and $V \sim P_2$. Then for any measurable function $f:\R^d \to \R^c$, if $\tilde{P}_1$ and $\tilde{P}_2$ are the laws of $f(U)$ and $f(V)$ respectively, $\tv{\tilde{P}_1 - \tilde{P}_2} \leq \tv{P_1 - P_2}$. It thus suffices to consider the case $d_X=d_Y=d_Z=1$.

Let $(X, Y, Z) \in \R^3$ is a random triple be a random triple. We define $Q \in \mathcal{Q}_{0,1}$ in the following way: if $(X, Y, Z) \sim Q$, $Z \independent (X, Y)$, $Z \sim U(-1, 1)$, and $(X, Y)$ is uniformly distributed on $(0, 1)^2 \cup (-1, 0)^2$. We note, for later use, that $\E_Q X = \E_Q Y = 0$ and $\E_Q XY = 1/4$.

Next take $P \in \mathcal{P}_{0,M}$ and let functions $f,g:(-1,1) \to (-1,1)$ be defined by $f(z) = \E_P(X|Z=z)$ and $g(z) = \E_P(Y|Z=z)$. Let $\psi: (-1,1)^3 \to \R$ be given by
\[
\psi(x, y, z) = \{x-f(z)\}\{y-g(z)\}.
\]
Note that $\psi$ as $x - f(z), \, y-g(z) \in (-2, 2)$, $\psi$ takes values in $(-4,4)$. Also $\E_P \psi(X, Y, Z) = 0$ and $\E_Pf(Z) g(Z) = \E_P XY$ as $\E_P (X-f(Z))g(Z) = 0$ and $\E_P f(Z)(Y-g(Z)) = 0$.

Let $\tv{Q - P} = \delta$. Then
\begin{align*}
\E_Q \psi(X, Y, Z) &= \E_Q XY + \E_Q f(Z) g(Z) \\
&\geq 1/4 + \E_P f(Z) g(Z) - 2\delta \\
&= 1/4 + \E_P XY - 2\delta \\
&\geq 1/4 + \E_Q XY - 4\delta = 1/2 - 4\delta,
\end{align*}
where we have used Lemma~\ref{lem:tv} in the second and last lines above. We now apply Lemma~\ref{lem:tv} once more to give
\[
1/2 - 4\delta \leq |\E_Q \psi(X, Y, Z) - \E_P \psi(X, Y, Z)| \leq 8 \delta,
\]
whence $\delta \geq 1/24$.
\subsubsection{Proof of Proposition~\ref{prop:KL}}
\begin{proof}
Let $q_{XYZ}$ be the density of $Q$ and denote by $q_{XY}$ the marginal density of $(X, Y)$ under $Q$ and similarly for $q_X$ etc.
To prove (i), we will consider minimising 
$\mathrm{KL}(Q||P)$ over distributions
$P \in \mathcal{P}_0$. Let $\mathcal{P}'_0$ be the set of densities of distributions in $\mathcal{P}_0$ with respect to Lebesgue measure with the added restriction that for $p \in \mathcal{P}'_0$, with a slight abuse of notation we have $\mathrm{KL}(q_{XYZ}||p) <\infty$. The proof will show that $\mathcal{P}'_0$ is non-empty. Given any $p \in \mathcal{P}'_0$, the conditional independence implies the factorisation 
$
p(x,y,z) = p_1(x|z) p_2(y|z) p_3(z)
$, 
i.e., 
we can minimise over the individual 
(conditional) densities 
$p_j$, $j \in \{1, 2, 3\}$. Adding and subtracting terms that do not depend on $p$, we obtain
\begin{align*}
& \argmin_{P \in \mathcal{P}_0} \mathrm{KL}(Q||P) \\
= &\argmin_{p \in \mathcal{P}'_0} \bigg\{
- \int q_{XYZ}(x,y,z) \log p(x,y,z) \,dx\,dy\,dz
+ \int q_{XYZ}(x,y,z) \log q_{XYZ}(x,y,z) \,dx\,dy\,dz \bigg\}\\
= &\argmin_{p \in \mathcal{P}'_0} \bigg\{
- \int q_{XY}(x,y)q_Z(z) \log p_1(x|z) \,dx \, dy \, dz
- \int q_{XY}(x,y)q_Z(z) \log p_2(y|z) \,dx \, dy \, dz \\
& \qquad \qquad \qquad \qquad - \int q_{XY}(x,y)q_Z(z) \log p_3(z) \,dx \, dy \, dz \bigg\} \\
= &\argmin_{p \in \mathcal{P}'_0} \bigg\{
- \int q_X(x)q_Z(z) \log \{p_1(x|z) q_Z(z)\} \,dx  \, dz
- \int q_Y(y)q_Z(z) \log \{p_2(y|z) q_Z(z)\} \, dy \, dz \\
& \qquad \qquad \qquad \qquad - \int q_Z(z) \log p_3(z) \, dz\bigg\}.
\end{align*}
Note that the fact that $\mathcal{P}_0'$ contains only densities $p$ such that $\mathrm{KL}(q_{XYZ}||p) <\infty$ ensures that all of the integrands above are integrable, which permits the use of the additivity property of integrals.
From Gibbs' inequality, we know the expression in the last display is minimised by
$p_1^*(x|z) = q_X(x)$,
$p_2^*(y|z) = q_Y(y)$, and
$p_3^*(z) = q_Z(z)$.
This implies 
$$
\inf_{P \in \mathcal{P}_0} \mathrm{KL}(Q||P) = \int q_{XY}(x,y) \log \frac{q_{XY}(x,y)}{q_X(x)q_Y(y)} = I_Q(X;Y) = -\frac{1}{2} \log (1-\rho_Q^2),
$$
where $I_Q(X, Y)$ denotes the mutual information between $X$ and $Y$, and $\rho_Q$ is the correlation coefficient of the bivariate Gaussian $(X,Y)$, both under $Q$.

For the second equality sign, 
we now minimise 
$\mathrm{KL}(P||Q)$ with respect to $P$. Let us redefine $\mathcal{P}'_0$ to now be the set of densities of distributions in $\mathcal{P}_0$ with the additional restriction that for $p \in \mathcal{P}'_0$,  we have $\mathrm{KL}(p||q_{XYZ}) <\infty$. For such $p$, we have
\begin{align} \label{eq:varBayes}
\mathrm{KL}(p||q_{XYZ}) 
& = \int p_3(z) \int p_1(x|z) p_2(y|z) \log \frac{p_1(x|z) p_2(y|z)}{q_{XY}(x,y)} \,dx\,dy\,dz
+ \int p_3(z) \log \frac{p_3(z)}{q_Z(z)} \,dz,
\end{align}
again noting that finiteness of $\mathrm{KL}(p||q_{XYZ})$ ensures Fubini's theorem and additivity may be used.
We will first consider, for a fixed $z$, the optimisation over 
$p_1$ and $p_2$.
From variational Bayes methods we know that 
\begin{align*}
p_1^*(x|z) &\propto \exp \bigg(\int_y p_2^*(y|z) \log q_{XYZ}(x,y,z) \,dy\bigg) \propto \exp\bigg(\int_y p_2^*(y|z) \log q_{XY}(x,y) \, dy\bigg) \\
p_2^*(y|z) &\propto \exp\bigg(\int_x p_1^*(x|z) \log q_{XYZ}(x,y,z) \,dx\bigg) \propto \exp\bigg(\int_x p_1^*(x|z) \log q_{XY}(x,y) \,dx\bigg).
\end{align*}
Straightforward calculations (Section~10.1.2 of \citet{Bishop2006}) show that
$p^*_1(x|z) = p^*_1(x)$ and $p^*_2(y|z) = p^*_2(y)$  are Gaussian densities with mean zero and variances
$\Sigma_{11}:=(\Sigma^{-1}_Q)_{11}^{-1} = \sigma^2/(\sigma^2 + 1)$ and
$\Sigma_{22}:=(\Sigma^{-1}_Q)_{22}^{-1} = \sigma^2$, respectively, where $\Sigma_Q$ is the covariance matrix of the bivariate distribution for $(X,Y)$ under $Q$.
It then follows from~\eqref{eq:varBayes} that 
$p_3^*(z) = q_Z(z)$. 
Thus we have,
\begin{align*}
\mathrm{KL}(p^*||q_{XYZ}) 
& = \int p^*(x,y) \log \{p^*(x,y)/q_{XY}(x,y)\} \,dx\,dy\\
 &= \frac{1}{2} \left(\mathrm{tr}(\Sigma_Q^{-1}\Sigma) - 2 + \log \frac{\mathrm{det}\Sigma_Q}{\mathrm{det} \Sigma} \right)
 = 
\frac{1}{2}\log\bigg(\frac{1+\sigma^2}{\sigma^2}\bigg).
\end{align*}
\end{proof}

\section{Proofs of Results in Section~\ref{SEC:GCM}}
In this section, we will use $a \lesssim b$ as shorthand for $a \leq Cb$ for some constant $C\geq 0$, where what $C$ is constant with respect to will be clear from the context.
\subsection{Proof of Theorem~\ref{thm:univariate}}
First observe that by scaling we may assume without loss of generality that $\E_P((\varepsilon_P \xi_P)^2) = 1$ for all $P$.

We begin by proving (i). We shall suppress the dependence on $P$ and $n$ at times to lighten the notation.
Recall the decomposition,
\begin{align}
\tau_N= \frac{1}{\sqrt{n}}\sum_{i=1}^n R_i &= \frac{1}{\sqrt{n}} \sum_{i=1}^n \{f(z_i) - \hat{f}(z_i) + \varepsilon_i\}\{g(z_i) -\hat{g}(z_i) + \xi_i\} \nonumber \\
&= (b + \nu_f + \nu_g) + \frac{1}{\sqrt{n}}\sum_{i=1}^n \varepsilon_i \xi_i, \label{eq:T_num2}
\end{align}
where
\begin{gather*}
b := \frac{1}{\sqrt{n}}\sum_{i=1}^n \{f(z_i) - \hat{f}(z_i)\}\{g(z_i) - \hat{g}(z_i)\}, \\
\nu_g := \frac{1}{\sqrt{n}} \sum_{i=1}^n \varepsilon_i \{g(z_i) - \hat{g}(z_i)\}, \qquad
\nu_f := \frac{1}{\sqrt{n}} \sum_{i=1}^n \xi_i \{f(z_i) - \hat{f}(z_i)\}.
\end{gather*}
Now
\[
\E(\varepsilon_i\xi_i) = \E\{\E(\varepsilon_i\xi_i |\mb Y, \mb Z)\} = \E\{\xi_i \E(\varepsilon_i |\mb Z)\} = 0.
\]
Thus the summands $\varepsilon_i \xi_i$ in the final term of \eqref{eq:T_num2} are i.i.d.\ mean zero with finite variance, so the central limit theorem dictates that this converges to a standard normal distribution. By the Cauchy--Schwarz inequality, we have
\begin{equation} \label{eq:A_fA_g}
|b| \leq  \sqrt{n}A_f^{1/2} A_g^{1/2} \inprob 0.
\end{equation}

We now turn to $\nu_f$ and $\nu_g$. Conditional on $\B{Y}$ and $\B{Z}$, $\nu_g$ is a sum of mean-zero independent terms and
\[
\var(\varepsilon_i \{g(z_i) - \hat{g}(z_i)\} | \B{Y}, \B{Z}) = \{g(z_i) - \hat{g}(z_i)\}^2 u(z_i).
\]
Thus our condition on $B_g$ gives us that $\E(\nu_g^2 |\mb Y, \mb Z) \inprob 0$.
Thus given $\epsilon > 0$
\begin{align}
\prob(\nu_g^2 \geq \epsilon) &= \pr(\nu_g^2 \wedge \epsilon \geq \epsilon) \notag\\
&\leq \epsilon^{-1} \E\{\E(\nu_g^2 \wedge \epsilon | \B{Y}, \B{Z})\} \notag\\
&\leq \epsilon^{-1} \E\{\E(\nu_g^2 | \B{Y}, \B{Z}) \wedge \epsilon \} \to 0 \label{eq:nu_g_bd}
\end{align}
using bounded convergence (Lemma~\ref{lem:bdd_conv}).

Using Slutsky's lemma, we may conclude that $\tau_N \indist \mathcal{N}(0, 1)$. We now argue that the denominator $\tau_D$ will converge to 1 in probability, which will give us $T_n \indist \mathcal{N}(0, 1)$ again by Slutsky's lemma.

First note that from the above we have in particular that $(b + \nu_f + \nu_g)/n \inprob 0$. Thus $\sum_{i=1}^n R_i/n \inprob 0$ by the weak law of large numbers (WLLN). It suffices therefore to show that $\sum_{i=1}^n R_i^2/n \inprob 1$. Now
\begin{align*}
|R_i^2 - \varepsilon_i^2\xi_i^2| \leq& [\{f(z_i) - \hat{f}(z_i)\}^2 + 2|\varepsilon_i\{f(z_i) - \hat{f}(z_i)\}|] \; [\{g(z_i) - \hat{g}(z_i)\}^2 + 2|\xi_i\{g(z_i) - \hat{g}(z_i)\}|] \\
&+ \varepsilon_i ^2 [\{g(z_i) - \hat{g}(z_i)\}^2 + 2|\xi_i\{g(z_i) - \hat{g}(z_i)\}|]\\
&+ \xi_i ^2 [\{f(z_i) - \hat{f}(z_i)\}^2 + 2|\varepsilon_i\{f(z_i) - \hat{f}(z_i)\}|] \\
=& \text{I}_i + \text{II}_i + \text{III}_i.
\end{align*}
Multiplying out and using the inequality $2|ab| \leq a^2 + b^2$ we have
\begin{align*}
\text{I}_i \leq & 3\{f(z_i) - \hat{f}(z_i)\}^2 \{g(z_i) - \hat{g}(z_i)\}^2 + 
\varepsilon_i^2 \{g(z_i) - \hat{g}(z_i)\}^2 + \xi_i^2 \{f(z_i) - \hat{f}(z_i)\}^2 \\
&+ 4|\varepsilon_i \xi_i \{f(z_i) - \hat{f}(z_i)\}\{g(z_i) - \hat{g}(z_i)\} |
\end{align*}
Now
\[
\frac{1}{n}\sum_{i=1}^n \{f(z_i) - \hat{f}(z_i)\}^2 \{g(z_i) - \hat{g}(z_i)\}^2 \leq n A_f A_g \inprob 0.
\]
Next note that for any $\epsilon > 0$,
\[
 \E\left( \epsilon \wedge \frac{1}{n}\sum_{i=1}^n\varepsilon_i^2 \{g(z_i) - \hat{g}(z_i)\}^2 \Big|\mb Y, \mb Z \right) = \E(\nu_g^2 \wedge \epsilon | \mb Y, \mb Z),
\]
so
\begin{equation} \label{eq:g_sum_0}
\frac{1}{n} \sum_{i=1}^n\varepsilon_i^2 \{g(z_i) - \hat{g}(z_i)\}^2 \inprob 0
\end{equation}
by the same argument as used to show $\nu_g \inprob 0$. Similarly, we also have that the corresponding term involving $f$, $\hat{f}$ and $\xi_i^2$ tends to 0 in probability. For the final term in $\text{I}_i$, we have
\begin{align*}
\frac{1}{n} \sum_{i=1}^n |\varepsilon_i \xi_i \{f(z_i) - \hat{f}(z_i)\}\{g(z_i) - \hat{g}(z_i)\} | \leq \bigg(\frac{1}{n} \sum_{i=1}^n \varepsilon_i^2 \xi_i^2 \bigg)^{1/2} \bigg(\frac{1}{n} \sum_{i=1}^n \{f(z_i) - \hat{f}(z_i)\}^2 \{g(z_i) - \hat{g}(z_i)\}^2 \bigg)^{1/2}.
\end{align*}
The first term above converges to $\{\E(\varepsilon_i^2 \xi_i^2)\}^{1/2}$ by the WLLN and the final term is bounded above by $\sqrt{nA_f A_g} \inprob 0$.

Turning now to $\text{II}_i$, we have
\[
\frac{1}{n}\sum_{i=1}^n \varepsilon_i^2 |\xi_i \{g(z_i) - \hat{g}(z_i)\}| \leq \bigg(\frac{1}{n}\sum_{i=1}^n \varepsilon_i^2\xi_i^2\bigg)^{1/2} \bigg(\frac{1}{n}\sum_{i=1}^n \varepsilon_i^2 \{g(z_i) - \hat{g}(z_i)\}^2\bigg)^{1/2} \inprob 0
\]
by WLLN and \eqref{eq:g_sum_0}.
Similarly we also have $\sum_{i=1}^n \text{III}_i/n \inprob 0$. As $\sum_{i=1}^n \varepsilon_i^2 \xi_i^2 /n \inprob \E(\varepsilon_i^2\xi_i^2) = 1$ by WLLN, we have $\tau_D \inprob 1$ as required.

The uniform result (ii) follows by an analogous 
argument to the above, the only differences being that all convergence in probability statements must be uniform, and the convergence in distribution via the central limit theorem must also be uniform over $\mathcal{P}$. These stronger properties follow easily from the stronger assumptions given in the statement of the result; that they suffice for uniform versions of the central limit theorem, WLLN and the particular applications of Slutsky's lemma required here to hold is shown in Lemmas~\ref{lem:uniform_CLT}, \ref{lem:uniform_WLLN} and \ref{lem:uniform_Slutsky} below.

\subsection{Uniform convergence results}
\begin{lemma}\label{lem:uniform_CLT}
Let $\mathcal{P}$ be a family of distributions for a random variable $\zeta \in \R$ and suppose $\zeta_1,\zeta_2,\ldots$ are i.i.d.\ copies of $\zeta$. For each $n \in \mathbb{N}$ let $S_n = n^{-1/2} \sum_{i=1}^n \zeta_i $. Suppose that for all $P\in \mathcal{P}$ we have $\E_P(\zeta) =0$, $\E_P(\zeta^2)=1$ and $\E_P(|\zeta|^{2+\eta}) < c$ for some $\eta, c > 0$. We have that
\[
\lim_{n \to \infty} \sup_{P \in \mathcal{P}} \sup_{t \in \R} | \pr_P(S_n\leq t) - \Phi(t)| = 0.
\]
\end{lemma}
\begin{proof}
For each $n$, let $P_n \in \mathcal{P}$ satisfy
\begin{equation} \label{eq:sup_eq}
\sup_{P \in \mathcal{P}} \sup_{t \in \R} | \pr_P(S_n\leq t) - \Phi(t)| \leq  \sup_{t \in \R} | \pr_{P_n}(S_n\leq t) - \Phi(t)| + n^{-1}.
\end{equation}
By the central limit theorem for triangular arrays \citep[Proposition 2.27]{vaart_1998}, we have
\[
\lim_{n \to \infty} \sup_{t \in \R} |\pr_{P_n}(S_n \leq t) - \Phi(t)| = 0,
\]
thus taking limits in \eqref{eq:sup_eq} immediately yields the result.
\end{proof}

\begin{lemma}\label{lem:uniform_WLLN}
Let $\mathcal{P}$ be a family of distributions for a random variable $\zeta \in \R$ and suppose $\zeta_1,\zeta_2,\ldots$ are i.i.d.\ copies of $\zeta$. For each $n \in \mathbb{N}$ let $S_n = n^{-1} \sum_{i=1}^n \zeta_i $. Suppose that for all $P\in \mathcal{P}$ we have $\E_P(\zeta) =0$ and $\E_P(|\zeta|^{1+\eta}) < c$ for some $\eta, c > 0$. We have that for all $\epsilon > 0$,
\[
\lim_{n \to \infty} \sup_{P \in\mathcal{P}} \pr_P(|S_n| > \epsilon ) = 0.
\]
\end{lemma}
\begin{proof}
Given $M>0$ (to be fixed at a later stage), let $\zeta^< := \zeta \ind_{\{|\zeta|\leq M\}}$ and $\zeta^> := \zeta \ind_{\{|\zeta| > M\}}$. Define $\zeta_i^<$ and $\zeta_i^>$ analogously, and let $S_n^<$ be the average of $\zeta_1^<, \ldots, \zeta_n^<$ with $S_n^>$ defined similarly. Note that by Chebychev's inequality, we have
\[
\sup_{P \in \mathcal{P}} \pr_P(|S_n^< - \E_P \zeta^<| \geq t) \leq \frac{M^2}{nt^2}.
\]
Also, by Markov's inequality and then the triangle inequality, we have for each $n$ that
\begin{align} \label{eq:S_n_Markov}
\sup_{P \in \mathcal{P}} \pr_P(|S_n^>| \geq t) \leq \frac{\sup_{P \in \mathcal{P}}\E_P |S_n^>|}{t} \leq \frac{\sup_{P \in \mathcal{P}}\E_P|\zeta^>|}{t}.
\end{align}
We now proceed to bound $\E_P|\zeta^>|$ in terms of $M$. We have
\begin{align*}
\E_P(|\zeta^>|) &= \E_P(|\zeta| \ind_{\{|\zeta| > M\}}) \\
&\leq \{\E_P(|\zeta|^{1+\eta})\}^{1/(\eta+1)} \{\pr_P(|\zeta|>M)\}^{\eta/(1 + \eta)} \\
&\leq c^{1/(\eta+1)} \{\pr_P(|\zeta|>M)\}^{\eta/(1 + \eta)}
\end{align*}
using H\"older's inequality. By Markov's inequality, we have
\[
\pr_P(|\zeta|>M) \leq \frac{\E_P|\zeta|}{M} \leq \frac{c}{M}.
\]
We may therefore conclude that $\sup_{P \in \mathcal{P}}\E_P(|\zeta^>|) \to 0$ as $M \to \infty$. Thus from \eqref{eq:S_n_Markov}, we have
\[
\sup_n \sup_{P \in \mathcal{P}} \pr_P(|S_n^>| \geq t) \to 0 \;\; \text{ as } M \to \infty,
\]
for each fixed $t$. Note further that as $\E_P \zeta = 0$ we have
\[
\sup_{P \in \mathcal{P}} |\E_P \zeta^<|  = \sup_{P \in \mathcal{P}} |\E_P \zeta^>| \leq \sup_{P \in \mathcal{P}} \E_P | \zeta^>| \to 0 \;\; \text{ as } M \to \infty.
\]
Now given $\epsilon, \delta > 0$, let $M$ be such that $\sup_{P \in \mathcal{P}} |\E_P \zeta^<| < \epsilon / 3$ and $\sup_n \sup_{P \in \mathcal{P}} \pr_P(|S_n^>| \geq \epsilon/3) < \delta / 2$. Next we choose $N \in \mathbb{N}$ such that $9 M^2 / (N \epsilon^2) < \delta / 2$. Then for all $n \geq N$ and all $P \in \mathcal{P}$, we have
\begin{align*}
\pr_P(|S_n| > \epsilon ) &\leq \pr_P(|S_n^<| > 2\epsilon/3 ) + \pr_P(|S_n^>| > \epsilon/3 ) \\
&\leq \pr_P(|S_n^< - \E_P \zeta^<| > \epsilon/3 ) + \delta/2 \leq \delta.
\end{align*}
\end{proof}

\begin{lemma}\label{lem:uniform_Slutsky}
Let $\mathcal{P}$ be a family of distributions that determines the law of a sequences $(V_n)_{n \in \mathbb{N}}$ and $(W_n)_{n \in \mathbb{N}}$ of random variables. Suppose
\[
\lim_{n \to \infty} \sup_{P \in \mathcal{P}} \sup_{t \in \R} |\pr_P(V_n\leq t) - \Phi(t)| =0.
\]
Then we have the following.
\begin{enumerate}[(a)]
\item
\[\text{If }\; W_n= o_{\mathcal{P}}(1) \;\text{ we have }\;
\lim_{n \to \infty} \sup_{P \in \mathcal{P}} \sup_{t \in \R} |\pr_P(V_n + W_n \leq t) - \Phi(t)| =0.
\]
\item \[\text{If } \; W_n= 1 + o_{\mathcal{P}}(1) \; \text{ we have } \;
\lim_{n \to \infty} \sup_{P \in \mathcal{P}} \sup_{t \in \R} |\pr_P(V_n / W_n \leq t) - \Phi(t)| =0.
\]
\end{enumerate}
\end{lemma}
\begin{proof}
We prove (a) first. Given $\epsilon > 0$, let $N$ be such that for all $n \geq N$ and for all $P \in \mathcal{P}$
\[
\sup_{t \in \R} |\pr_P(V_n\leq t) - \Phi(t)| < \epsilon/3 \;\;\text{ and } \;\; \pr(|W_n| > \epsilon/3) < \epsilon/3.
\]
Then
\begin{align*}
\pr_P(V_n + W_n \leq t) - \Phi(t) &\leq \pr_P(V_n \leq t + \epsilon/3) -\Phi(t) + \pr(|W_n| > \epsilon / 3) \\
&\leq \epsilon/3 + \Phi(t+\epsilon/3) - \Phi(t) + \epsilon / 3 <\epsilon,
\end{align*}
and
\begin{align*}
\pr_P(V_n + W_n \leq t) - \Phi(t) &\geq \pr_P(V_n \leq t - \epsilon/3) -\Phi(t) \\
&\geq -\epsilon/3 + \Phi(t-\epsilon/3) - \Phi(t)  > -\epsilon.
\end{align*}
Thus for all $n \geq N$ and for all $P \in \mathcal{P}$,
\[
\sup_{t \in \R} |\pr_P(V_n + W_n \leq t) - \Phi(t)| < \epsilon
\]
as required.
Turning to $(b)$, first let $\zeta \sim \mathcal{N}(0, 1)$ and note that for any sequence $(\epsilon_n)_{n \in \mathbb{N}}$ with $\epsilon_n \downarrow 0$, we have $(1+\epsilon_n)^{-1}\zeta$ converges in distribution to $\zeta$, whence $\sup_{t} |\Phi(t(1 + \epsilon_n)) - \Phi(t)| \to 0$ and similarly $\sup_{t} |\Phi(t(1 - \epsilon_n)) - \Phi(t)| \to 0$; note that the fact that we may take a supremum over $t$ here follows from \citet[Lemma~2.11]{vaart_1998}. Now, given $\epsilon > 0$, let $\delta$ be such that for all $0 \leq \delta' \leq \delta$, $\sup_t |\Phi(t(1 + \delta')) - \Phi(t(1 - \delta'))| \leq \epsilon / 3$. Next choose $N$ such that for all $n \geq N$ and for all $P \in \mathcal{P}$
\[
\sup_{t \in \R} |\pr_P(V_n\leq t) - \Phi(t)| < \epsilon/3 \;\;\text{ and } \;\; \pr(|W_n| > \delta) < \epsilon/3.
\]
Then for all $n \geq N$ and for all $P \in \mathcal{P}$, for $t \geq 0$
\begin{align*}
\pr_P(V_n / W_n \leq t) - \Phi(t) &\leq \pr_P(V_n \leq t(1 + \delta)) -\Phi(t) + \pr(|W_n| > \delta) \\
&\leq \epsilon/3 + \Phi(t(1+\delta)) - \Phi(t) + \epsilon / 3 <\epsilon.
\end{align*}
Similarly when $t \leq 0$, replacing $1+\delta$ with $1-\delta$ in the above argument gives the equivalent result. The inequality $\pr_P(V_n / W_n \leq t) - \Phi(t) > -\epsilon$ may be proved analogously. Putting things together then gives part (b) of the result.
\end{proof}

\subsection{Proof of Theorem~\ref{thm:power}}
The proof of this result is very similar to that of Theorem~\ref{thm:univariate}, and we will adopt the same notation here. We shall suppress the dependence on $P$ and $n$ at times to lighten the notation. We shall denote the auxiliary dataset by $\mb A$.
We begin by proving (i). We have
\begin{align}
\tau_N - \sqrt{n} \rho_P = (b + \nu_f + \nu_g) + \frac{1}{\sqrt{n}}\sum_{i=1}^n (\varepsilon_i \xi_i - \rho).
\end{align}
Thus the summands $\varepsilon_i \xi_i - \rho$ in the final term are i.i.d.\ mean zero with finite variance, so the central limit theorem dictates that this converges to a standard normal distribution.

Control of the term $b$ is identical to that in the proof of Theorem~\ref{thm:univariate}.
Turning to $\nu_f$ and $\nu_g$, Conditional on $\B{Z}$ and the auxiliary dataset $\mb A$, $\nu_g$ is a sum of mean-zero independent terms and
\[
\var(\varepsilon_i \{g(z_i) - \hat{g}(z_i)\} | \B{A}, \B{Z}) = \{g(z_i) - \hat{g}(z_i)\}^2 u(z_i).
\]
That $\nu_g \inprob 0$ follows exactly as in the argument preceding \eqref{eq:nu_g_bd}, and similarly for $\nu_f$.
Using Slutsky's lemma, we may conclude that $\tau_N -\sqrt{n} \rho_P \indist \mathcal{N}(0, 1)$.

The argument that $\tau_D \inprob \sigma$ proceeds similarly to that in the proof of Theorem~\ref{thm:univariate}, but with conditioning on $\mb X$ or $\mb Y$ replaced by conditioning on $\mb A$. The uniform result (ii) follows by an analogous argument, see the comments at the end of the proof of Theorem~\ref{thm:univariate}.

\subsection{Proof of Theorem~\ref{thm:multivariate}}
The proof of Theorem~\ref{thm:multivariate} relies heavily on results from \citet{chernozhukov2013gaussian} which we state in the next section for convenience, after which we present the proof Theorem~\ref{thm:multivariate}.

\subsubsection{Results from \citet{chernozhukov2013gaussian}}
In the following, $W \sim \mathcal{N}_p(0, \mbb\Sigma)$ where $\Sigma_{jj}  =1$ for $j=1,\ldots,p$ and $p \geq 3$. Set $V = \max_{j=1,\ldots,p}|W_j|$.
In addition, let $\tilde{w}_1,\ldots,\tilde{w}_n \in \R^p$ be independent random vectors having the same distribution as a random vector $\tilde{W}$ with $\E \tilde{W} = 0$ and covariance matrix $\mbb\Sigma$.

Consider the following conditions.
\begin{itemize}
\item[(B1a)] $\max_{k=1,2} \E(|\tilde{W}_j|^{2+k}/C_n^k) + \E(\exp(|\tilde{W}_j|/C_n)) \leq 4$ for all $j$;
\item[(B1b)] $\max_{k=1,2} \E(|\tilde{W}_j|^{2+k}/C_n^{k/2}) + \E(\max_{j=1,\ldots,p}|\tilde{W}_j|^4/C_n^2) \leq 4$ for all $j$;
\item[(B2)] $C_n^2 (\log(pn))^7/n \leq Cn^{-c}$ for some constants $C, c >0$.
\end{itemize}
We will assume that $\tilde{W}$ satisfies one of (B1a) and (B1b), and also satisfies (B2), for some $C_n \geq 1$. Here the constants implicit in the relationship $a \lesssim b$ will be permitted to depend on $c$ and $C$ (but not $C_n$).

Let
\[
\tilde{V} = \max_{j=1,\ldots,p} \bigg| \frac{1}{\sqrt{n}} \sum_{i=1}^n \tilde{w}_{ij} \bigg|.
\]

The labels of the corresponding results in \citet{chernozhukov2013gaussian} are given in brackets. A slight difference between our presentation of these results here and the statements in \citet{chernozhukov2013gaussian} is that we consider the maximum absolute value rather than the maximum.

\begin{lemma}[Lemma 2.1] \label{lem:anti_conc}
There exists an absolute constant $C'>0$ such that for all $t \geq 0$,
\[
\sup_{w \geq 0} \pr(|V - w| \leq t) \leq C't (\sqrt{2\log(p)}+1).
\]
\end{lemma}

\begin{theorem}[Corollary 2.1] \label{thm:Chernozhukov_main}
Then there exists a constant $c'>0$ such that
\[
\sup_{t \in \R} |\pr(\tilde{V} \leq t) - \pr(V \leq t)| \lesssim n^{-c'}.
\]
\end{theorem}

The following  result includes a slight variant of Lemma 3.2 of \citet{chernozhukov2013gaussian} whose proof follows in exactly the same way.
\begin{lemma}[Lemmas 3.1 and 3.2] \label{lem:Gauss_comp}
Let $U \in \R^p$ be a centred Gaussian random vector with covariance matrix $\mbb\Theta \in \R^{p\times p}$ and let $\Delta_0 = \max_{j,k=1,\ldots,p}|\Sigma_{jk} - \Theta_{jk}|$. Define $q(\theta) := \theta^{1/3}(1 \vee \log(p/\theta))^{2/3}$.
There exists a constant $C'>0$ such that the following hold.
\begin{enumerate}[(i)]
\item 
\[
\sup_{t \in \R} |\pr(\max_{j=1,\ldots,p} |U_j| \leq t) - \pr(V \leq t)| \leq C' q(\Delta_0).
\]
\item Writing $G_{\mbb\Sigma}$ and $G_{\mbb\Theta}$ for the quantile functions of $V$ and 
$\max_{j=1,\ldots,p} |U_j|$ 
respectively,
\begin{align*}
G_{\mbb\Theta}(\alpha) \leq G_{\mbb\Sigma}(\alpha + C'q(\Delta_0)) \qquad \text{and} \qquad G_{\mbb\Sigma}(\alpha) \leq G_{\mbb\Theta}(\alpha + C'q(\Delta_0))
\end{align*}
for all $\alpha \in (0,1)$.
\end{enumerate}
\end{lemma}

\begin{lemma}[Lemma C.1 and the proof of Corollary 3.1] \label{lem:Sigma_conc}
Let $\tilde{\mbb\Sigma}$ be the empirical covariance matrix formed using $\tilde{w}_1, \ldots, \tilde{w}_n$, so $\tilde{\mbb\Sigma} = \sum_{i=1}^n \tilde{w}_i\tilde{w}_i^T / n$. Then
\begin{align*}
\log(p)^2 \E \|\tilde{\mbb\Sigma} - \mbb\Sigma\|_\infty  &\lesssim n^{-c'} \\
\log(p)^2 \E \bigg(\max_{j=1,\ldots,p} \bigg| \frac{1}{n} \sum_{i=1}^n \tilde{w}_{ij}\bigg| \bigg) &\lesssim n^{-c'}.
\end{align*}
for some constant $c' >0$.
\end{lemma}
Note the second inequality does not appear in \citet{chernozhukov2013gaussian} but follows easily in a similar manner to the first inequality.

\subsubsection{Proof of Theorem~\ref{thm:multivariate}} \label{sec:mvt_proof}
We will assume, without loss of generality, that $\var_P(\varepsilon_{P,j}\xi_{P,k}) = 1$ for all $P \in \mathcal{P}$. Furthermore, we will suppress dependence on $P$ and $n$ at times in order to lighten the notation. We will use $C'$ to denote a positive constant that may change from line to line.

Note that we have $\E(\varepsilon_j\xi_k) = 0$. Let us decompose $\sqrt{n}\bar{\mb R}_{jk} = \delta_{jk} + \tilde{T}_{jk}$ where
\[
\tilde{T}_{jk} = \frac{1}{\sqrt{n}} \sum_{i=1}^n \varepsilon_{ij} \xi_{ik}.
\]
Furthermore, let us write the denominator in the definition of $T_{jk}$ as $(\|\mb R_{jk}\|_2^2/n - \bar{\mb R}_{jk}^2)^{1/2} = 1 + \Delta_{jk}$
We thus have $T_{jk} = (\tilde{T}_{jk} + \delta_{jk}) / (1 + \Delta_{jk})$. Let $\tilde{S}_n = \max_{j,k} |\tilde{T}_{jk}|$.

Let $\mbb\Sigma \in \R^{d_X \cdot d_Y \times d_X \cdot d_Y}$ be the matrix with columns and rows indexed by pairs $jk=(j,k) \in \{1,\ldots,d_X\}\times \{1,\ldots,d_Y\}$ and entries given by $\Sigma_{jk,lm} = \E(\varepsilon_j\varepsilon_l\xi_k\xi_m)$. Let $W \in \R^{d_X\cdot d_Y}$ be a centred Gaussian random vector with covariance $\mbb\Sigma$ and let $V_n$ be the maximum of the absolute values of components of $W$. We write $G$ for the quantile function of $V_n$. Note that from Lemma~\ref{lem:anti_conc}, we have in particular that $V_n$ has no atoms, so $\pr(V_n \leq G(\alpha)) = \alpha$ for all $\alpha \in [0,1]$.

Let
\[
\kappa_P = \sup_{t \geq 0} |\pr_P(S_n \leq t) - \pr(V_n \leq t)|.
\]
We will first obtain a bound on
\[
v_P(\alpha) = |\pr_P(S_n \leq \hat{G}(\alpha)) - \alpha|
\]
in terms of $\kappa_P$, and later bound $\kappa_P$ itself.
Fixing $P \in \mathcal{P}$ and suppressing dependence on this, we have
\begin{align*}
v(\alpha) &\leq |\pr(S_n \leq \hat{G}(\alpha)) - \pr(S_n \leq G(\alpha))| + |\pr(S_n \leq G(\alpha)) - \pr(V_n \leq G(\alpha))| \\
&\leq \pr(\{S_n \leq \hat{G}(\alpha)\} \bigtriangleup \{S_n \leq G(\alpha)\}) + \kappa
\end{align*}
where we have used the fact that $|\pr(A) - \pr(B)| \leq \pr(A \bigtriangleup B)$. Now from Lemma~\ref{lem:Gauss_comp} we know that on the event $\{\|\mbb\Sigma - \hat{\mbb\Sigma}\|_\infty \leq u_\Sigma\}$, we have $G(\alpha - Cq(u_\Sigma)) \leq \hat{G}(\alpha) \leq G(\alpha + Cq(u_\Sigma))$. Thus
\begin{align*}
& \pr(\{S_n \leq \hat{G}(\alpha)\} \bigtriangleup \{S_n \leq G(\alpha)\})  \leq \pr\{G(\alpha - C'q(u_\Sigma)) \\
\leq \; &  S_n \leq G(\alpha + C'q(u_\Sigma))\} + \pr(\|\mbb\Sigma - \hat{\mbb\Sigma}\|_\infty > u_\Sigma) \\
\leq \; & 2\kappa + \pr\{G(\alpha - C'q(u_\Sigma)) \leq  V_n \leq G(\alpha + C'q(u_\Sigma))\} + \pr(\|\mbb\Sigma - \hat{\mbb\Sigma}\|_\infty > u_\Sigma) \\
= \; & 2\kappa + 2C'q(u_\Sigma) + \pr(\|\mbb\Sigma - \hat{\mbb\Sigma}\|_\infty > u_\Sigma).
\end{align*}
This gives
\[
v(\alpha) \lesssim \kappa + q(u_\Sigma) + \pr(\|\mbb\Sigma - \hat{\mbb\Sigma}\|_\infty > u_\Sigma).
\]

Now let $\Omega$ be the event that $\max_{j,k} |\delta_{jk}| \leq u_\delta$ and $\max_{j,k} |\Delta_{jk}| \leq u_{\Delta}$.
\begin{align*}
\kappa &\leq \sup_{t \geq 0} \{|\pr(\tilde{S}_n \leq t(1 + u_\Delta) + u_\delta) - \pr(V_n \leq t)| + |\pr(\tilde{S}_n \leq t(1 - u_\Delta) - u_\delta) - \pr(V_n \leq t)|\} + \pr(\Omega^c) \\
&\leq \sup_{t \geq 0} \{|\pr(\tilde{S}_n \leq t) - \pr\{V_n \leq (t - u_\delta)/(1+u_\Delta)\}| + |\pr(\tilde{S}_n \leq t) - \pr\{V_n \leq (t + u_\delta)/(1-u_\Delta)\}|\} + \pr(\Omega^c).
\end{align*}
Now
\begin{align*}
&|\pr(\tilde{S}_n \leq t) - \pr\{V_n \leq (t - u_\delta)/(1+u_\Delta)\}| \\
\leq & |\pr(\tilde{S}_n \leq t) - \pr(V_n \leq t)| + |\pr(V_n \leq t - u_\delta) - \pr((1+u_\Delta)V_n \leq t - u_\delta)| + |\pr(t-u_\delta \leq V_n \leq t)| \\
\leq & \text{I} + \text{II} + \text{III}.
\end{align*}
From Theorem~\ref{thm:Chernozhukov_main}, we have $\text{I} = o(1)$. Lemma~\ref{lem:Gauss_comp} and Lemma~\ref{lem:anti_conc} give
\[
\text{II} + \text{III} \lesssim q(u_\Delta) + u_\delta \sqrt{\log(d)}.
\]
Similarly,
\[
|\pr(\tilde{S}_n \leq t) - \pr\{V_n \leq (t + u_\delta)/(1-u_\Delta)\}| \lesssim q(u_\Delta) + u_\delta \sqrt{\log(d)} + o(1).
\]
Putting things together, we have
\begin{align*}
v(\alpha) \lesssim &\; \pr(\|\mbb\Sigma - \hat{\mbb\Sigma}\|_\infty > u_\Sigma) + q(u_\Sigma) +  \pr(\max_{j,k} |\delta_{jk}| > u_\delta) + u_\delta\sqrt{\log(d)} \\\
& + \pr(\max_{j,k} |\Delta_{jk}| > u_\Delta) + q(u_\Delta) + o(1).
\end{align*}
We thus see that writing $a_n = \log(d)^{-2}$, if $\max_{j,k} |\delta_{jk}|= o_\mathcal{P}(a_n^{1/4})$, $\max_{j,k}|\Delta_{jk}|=o_\mathcal{P}(a_n)$ and $\|\mbb\Sigma - \hat{\mbb\Sigma}\|_\infty = o_\mathcal{P}(a_n)$, then we will have $\sup_{P \in \mathcal{P}} \sup_{\alpha \in (0,1)} v_P(\alpha) \to 0$. These remaining properties are shown in Lemma~\ref{lem:Sigma_bd}.

\subsection{Auxiliary Lemmas}

\begin{lemma} \label{lem:bdd_conv}
Let $\mathcal{P}$ be a family of distributions determining the distribution of a sequence of random variables $(W_n)_{n \in \mathbb{N}}$. Suppose $W_n = o_\mathcal{P}(1)$ and $|W_n|<C$ for some $C > 0$. Then $\sup_{P \in \mathcal{P}} \E_P |W_n| \to 0$.
\end{lemma}
\begin{proof}
Given $\epsilon > 0$, there exists $N$ such that $\sup_{P \in \mathcal{P}}\pr_P(|W_n| > \epsilon) < \epsilon$ for all $n \geq N$. Thus for such $n$,
\[
\E_P |W_n| = \E_P (|W_n| \ind_{\{|W_n| \leq \epsilon\}}) + \E_P (|W_n| \ind_{\{|W_n| >\epsilon\}}) \leq \epsilon + C \epsilon.
\]
As $\epsilon$ was arbitrary, we have $\sup_{P \in \mathcal{P}}\E_P |W_n| \to 0$.
\end{proof}

\begin{lemma} \label{lem:Sigma_bd}
Consider the setup of Theorem~\ref{thm:multivariate} and its proof (Section~\ref{sec:mvt_proof}). Let $a_n = \log(d)^{-2}$. We have that
\begin{enumerate}[(i)]
\item $\max_{j,k} |\delta_{jk}|= o_\mathcal{P}(a_n^{1/4})$;
\item $\max_{j,k} |\Delta_{jk}| = o_\mathcal{P}(a_n)$;
\item $\|\mbb\Sigma - \hat{\mbb\Sigma}\|_\infty = o_\mathcal{P}(a_n)$.
\end{enumerate}
\end{lemma}
\begin{proof}
The arguments here are similar to those in the proof of Theorem~\ref{thm:univariate}, but with the added complication of requiring uniformity over expressions corresponding to different components of $X$ and $Y$. We will at times suppress the dependence of quantities on $P$ to lighten notation.

We begin by showing (i). Let us decompose each $\delta_{jk}$ as $\delta_{jk} = b_{jk} + \nu_{g,jk} + \nu_{f,jk}$, these terms being defined as the analogues of $b$, $\nu_g$ and $\nu_f$ but corresponding to the regression of $\mb X_j^{(n)}$ and $\mb Y_j^{(n)}$ on to $\mb Z^{(n)}$.

By the Cauchy--Schwarz inequality, we have $b_{jk} \leq \sqrt{n}A_{f,j}^{1/2}A_{g,k}^{1/2} = o_\mathcal{P}(a_n^{1/4})$ using \eqref{eq:A_cond}.
Let us write $\omega_{ik} = g_k(z_i) - \hat{g}_k(z_i)$. In order to control $\max_{j,k} |\nu_{g,jk}|$ we will use Lemma~\ref{lem:nu_ineq}. Given $\epsilon > 0$, we have
\begin{align} \label{eq:nu_bd}
 \pr_P(\max_{j,k} |\nu_{g,jk}| / a_n^{1/4} \geq \epsilon) \lesssim  \E_P\bigg\{\epsilon \wedge \tau \sqrt{\log(d)} \bigg(\max_k \frac{1}{n a_n^{1/2}}\sum_{i=1}^n \omega_{ik}^2\bigg)^{1/2}\bigg\} + \pr_P(\max_{i,j} |\varepsilon_{ij}| > \tau)
\end{align}
for all $\tau \geq 0$. As $\max_{i,j} |\varepsilon_{ij}|=O_{\mathcal{P}}(\tau_{g,n})$, we know that given $\delta$, there exists $C>0$ such that $\sup_{P \in \mathcal{P}} \pr_P(\max_{i,j} |\varepsilon_{ij}| > C\tau_{g,n}) < \delta$ for all $n$. By bounded convergence (Lemma~\ref{lem:bdd_conv}) and \eqref{eq:tau_g_cond}, we then have that
\[
\sup_{P \in \mathcal{P}}  \E_P\bigg\{\epsilon \wedge \tau_{g,n} \sqrt{ \log(d)} \bigg(\max_k \frac{1}{n a_n^{1/2}}\sum_{i=1}^n \omega_{ik}^2\bigg)^{1/2}\bigg\} \to 0,
\]
whence $\max_{j,k}|\nu_{g,jk}| = o_\mathcal{P}(a_n^{1/4})$. Similarly $\max_{j,k} |\nu_{f,jk}| =o_\mathcal{P}(a_n^{1/4})$, which completes the proof of (i).

Turning to (ii), we see that
\[
\max_{j,k} |(1 + \Delta_{jk})^2 - 1| \leq \max_{j,k} |\|\mb R_{jk}\|_2^2/n-1| + \max_{j,k} |\bar{\mb R}_{jk}|.
\]
Lemma~\ref{lem:Sigma_bd_aux} shows that the first term on the RHS is $o_\mathcal{P}(a_n)$. For the second term we have
\begin{align*}
\max_{j,k} |\bar{\mb R}_{jk}| \leq \max_{j,k} |\delta_{jk}| / \sqrt{n} + \max_{j,k} \frac{1}{n}\sum_{i=1}^n \varepsilon_{ij}\xi_{ik} = o_\mathcal{P}(a_n)
\end{align*}
from (i) and Lemma~\ref{lem:Sigma_conc}, noting that (A2) implies in particular that $\log(d)^3=o(n)$.
Thus applying Lemma~\ref{lem:max_mapping}, we have that $\max_{j,k} |\Delta_{jk}| = o_{\mathcal{P}}(a_n)$. 

We now consider (iii).
Let $\tilde{\mbb\Sigma} \in \R^{d_X d_Y \times d}$ be the matrix with rows and columns indexed by pairs $jk=(j,k) \in \{1,\ldots,d_X\} \times \{1,\ldots,d_Y\}$ and entries given by
\[
\tilde{\Sigma}_{jk,lm} = \frac{1}{n} \sum_{i=1}^n \varepsilon_{ij}\xi_{ik}\varepsilon_{il}\xi_{im}.
\]
We know from Lemma~\ref{lem:Sigma_conc} that $\|\mbb\Sigma - \tilde{\mbb\Sigma}\|_\infty = o_\mathcal{P}(a_n)$. It remains to show that $\|\hat{\mbb\Sigma} - \tilde{\mbb\Sigma}\|_\infty = o_\mathcal{P}(a_n)$.

From Lemma~\ref{lem:Sigma_bd_aux} we have that
\[
\max_{j,k,l,m} |\mb R_{jk} ^T \mb R_{lm} /n - \bar{\mb R}_{jk}\bar{\mb R}_{lm} - \tilde{\Sigma}_{jk,lm}| = o_\mathcal{P}(a_n).
\]
It suffices to show that 
\begin{equation} \label{eq:Delta_prob_bd}
\max_{j,k,l,m} |\{(1+\Delta_{jk})(1+\Delta_{lm})\}^{-1} - 1| = o_\mathcal{P}(a_n).
\end{equation}
We already know that $\max_{j,k}|\Delta_{jk}| = o_\mathcal{P}(a_n)$ so applying Lemma~\ref{lem:max_mapping}, we see that $\max_{j,k}|(1+\Delta_{jk})^{-1} -1|= o_\mathcal{P}(a_n)$. It is then straightforward to see that \eqref{eq:Delta_prob_bd} holds. This completes the proof of (iii).
\end{proof}

\begin{lemma} \label{lem:Sigma_bd_aux}
Consider the setup of Theorem~\ref{thm:multivariate} and its proof (Section~\ref{sec:mvt_proof}) as well as that of Lemma~\ref{lem:Sigma_bd}.
We have that
\begin{align*}
\max_{j,k,l,m} |\mb R_{jk} ^T \mb R_{lm} /n - \tilde{\Sigma}_{jk,lm}| &= o_\mathcal{P}(a_n).
\end{align*}
\end{lemma}
\begin{proof}
Fix $i$ and consider $R_{jk,i}R_{lm,i} - \varepsilon_{ij}\xi_{ik}\varepsilon_{il}\xi_{im}$. Writing $\eta_j = f_j(z_i) - \hat{f}_j(z_i)$ and $\omega_k = g_k(z_i) - \hat{g}_k(z_i)$, and suppressing dependence on $i$ (so e.g.\ $\varepsilon_{ij} = \varepsilon_j$) we have
\begin{align*}
R_{jk,i}R_{lm,i} - \varepsilon_{ij}\xi_{ik}\varepsilon_{il}\xi_{im} &= (\eta_j + \varepsilon_j)(\omega_k + \xi_k)(\eta_l + \varepsilon_l)(\omega_m + \xi_m) - \varepsilon_j \xi_k \varepsilon_l \xi_m \\
&= \eta_j \omega_k \eta_l \omega_m \\
&\;\;\; + \eta_j \omega_k \eta_l \xi_m + \eta_j \omega_k\omega_m \varepsilon_l + \eta_j \eta_l \omega_m \xi_k + \omega_k \eta_l \omega_m  \varepsilon_j \\
&\;\;\;+ \eta_j \omega_k \varepsilon_l \xi_m + \eta_j\eta_l \xi_k\xi_m + \eta_j\omega_m\xi_k\varepsilon_l + \omega_k \eta_l \varepsilon_j\xi_m + \omega_k\omega_m\varepsilon_j \varepsilon_l + \eta_l\omega_m \varepsilon_j \xi_k \\
&\;\;\; + \eta_j \xi_k\varepsilon_l\xi_m + \omega_k\varepsilon_j\varepsilon_l\xi_m + \eta_l\varepsilon_j\xi_k\xi_m + \omega_m\varepsilon_j\xi_l\varepsilon_l.
\end{align*}
We see that the sum on the RHS contains terms of four different types of which $\eta_j \omega_k \eta_l \omega_m$, $\eta_j \omega_k \eta_l \xi_m$, $\eta_j \omega_k \varepsilon_l \xi_m$ and $\eta_j \xi_k\varepsilon_l\xi_m$ are representative examples. We will control the sizes of each of these when summed up over $i$.
Turning first to  $\eta_j \omega_k \eta_l \omega_m$, note that $2|\eta_j \omega_k \eta_l \omega_m| \leq \eta_j^2\omega_k^2 + \eta_l^2\omega_m^2$.

The argument of \eqref{eq:A_fA_g} combined with \eqref{eq:A_cond} shows that
\[
\max_{j,k} \frac{1}{n} \sum_{i=1}^n  \eta_{ij}^2\omega_{ik}^2 =o_\mathcal{P}(a_n).
\]

Next we have $2|\eta_j \omega_k \eta_l \xi_m| \leq \eta_j^2\omega_k^2 + \eta_l^2\xi_m^2$.

Arguing as in \eqref{eq:nu_bd}, we have for any $\epsilon > 0$,
\begin{align*}
\pr_P\left(\frac{1}{a_n^2} \max_{l,m} \frac{1}{n}\sum_{i=1}^n \eta_{il}^2\xi_{im}^2 \geq \epsilon \right) &= \pr_P\left(\frac{1}{a_n^2} \max_{l,m} \frac{1}{n}\sum_{i=1}^n \eta_{il}^2\xi_{im}^2 \wedge \epsilon \geq \epsilon \right) \\
&\leq \frac{1}{\epsilon} \E_P \left(\frac{1}{a_n^2} \max_{l,m} \frac{1}{n}\sum_{i=1}^n \eta_{il}^2\xi_{im}^2 \right),
\end{align*}
using Markov's inequality in the final line. Next by \eqref{eq:tau_f_cond}, $\max_{i,m} \xi_{im}^2 =O_\mathcal{P}(\tau_{f,n}^2)$, so
\begin{align*}
\frac{1}{a_n^2} \max_{l,m} \frac{1}{n}\sum_{i=1}^n \eta_{il}^2\xi_{im}^2 &\leq \frac{1}{a_n^2} \max_{l} \frac{1}{n}\sum_{i=1}^n \eta_{il}^2 \max_{i,m} \xi_{im}^2 \\
&= \log(d)^4 \max_l A_{f,l} O_{\mathcal{P}}(\tau_{f,n}^2).
\end{align*}
Again by \eqref{eq:tau_f_cond}, this is $o_{\mathcal{P}}(1)$, so by bounded convergence, we have that
\begin{align} \label{eq:an_sq}
\max_{l,m} \frac{1}{n}\sum_{i=1}^n \eta_{il}^2\xi_{im}^2 =o_\mathcal{P}(a_n^2) = o_\mathcal{P}(a_n).
\end{align}
Considering the third term, we have $2|\eta_j \omega_k \varepsilon_l \xi_m| \leq \eta_j^2 \xi_m^2 + \omega_k^2\varepsilon_l^2$, so this term may be controlled in the same way.

Finally, turning to the fourth term,
\begin{align*}
\max_{j,k,l,m} \frac{1}{n} \bigg|\sum_{i=1}^n \eta_{ij} \xi_{ik}\varepsilon_{il}\xi_{im}\bigg|
&\leq
\max_{l,m} \bigg(\frac{1}{n}\sum_{i=1}^n \varepsilon_{il}^2\xi_{im}^2\bigg)^{1/2} \max{j,k} \bigg( \frac{1}{n} \sum_{i=1}^n \eta_{ij}^2\xi_{ik}^2\bigg)^{1/2} \\
&= O_\mathcal{P}(1) o_\mathcal{P}(a_n) = o_\mathcal{P}(a_n),
\end{align*}
using \eqref{eq:an_sq}.
This completes the proof of the result.
\end{proof}

\begin{lemma} \label{lem:max_mapping}
Let $\mathcal{P}$ be a family of distributions determining the distribution of a triangular array of random variables  $W^{(n)} \in \R^{p_n}$. Suppose that for some sequence $(a_n)_{n \in \mathbb{N}}$, $\max_{j=1,\ldots,p_n} W_j^{(n)} = o_\mathcal{P}(a_n)$ as $n \to \infty$. Then if $D \subset \R$ contains an open neighbourhood of 0 and the function $f:D \to \R$ is continuously differentiable at 0 with $f(0)=c$, we have
\[
\max_{j=1,\ldots,p_n} \{f(W_j^{(n)}) -c\}  = o_\mathcal{P}(a_n).
\]
\end{lemma}
\begin{proof}
Let $\epsilon, \delta >0$.
As $f'$ is continuous at 0, it is bounded on a sufficiently small interval $(-\delta', \delta') \subseteq D$. Let $M = \sup_{x \in (-\delta', \delta')} |f'(x)|$ and set $\eta = \min(\delta', \delta/M)$. Note by the mean-value theorem we have the inequality $|f(x)-c| \leq M|x| \leq \delta$ for all $x  \in (-\eta, \eta)$. Thus
\[
a_n \max_j |f(W_j^{(n)})-c| > \delta \subseteq a_n \max_j |W_j^{(n)}| > \eta.
\]
Now we have that there exists $N$ such that for all $n \geq N$, $\pr_P(a_n \max_j |W_j^{(n)}| > \eta) < \epsilon$ for all $P \in \mathcal{P}$, so from the display above we have that for such $n$, $\pr_P(a_n \max_j |f(W_j^{(n)})-c| > \delta) < \epsilon$.
\end{proof}

\begin{lemma} \label{lem:nu_ineq}
Let $W \in \R^{n \times d}$, $V \in \R^{n \times d}$ be random matrices such that $\E (W | V) = 0$ and the rows of $W$ are independent conditional on $V$. Then for $\epsilon > 0$,
\[
\epsilon\, \pr\bigg( \max_{j} \bigg| \frac{1}{\sqrt{n}} \sum_{i=1}^n W_{ij} V_{ij} \bigg| > \epsilon \bigg) \lesssim \, \E \bigg\{ \epsilon \wedge \lambda\,\sqrt{\log(d) } \, \bigg(\max_{j} \frac{1}{n} \sum_{i=1}^n V_{ij}^2 \bigg)^{1/2} \bigg\} + \epsilon\, \pr(\|W\|_\infty > \lambda)
\]
for any $\lambda \geq 0$.
\end{lemma}
\begin{proof}
We have from Markov's inequality
\begin{align*}
\epsilon\, \pr\bigg( \max_{j} \bigg| \frac{1}{\sqrt{n}} \sum_{i=1}^n W_{ij} V_{ij} \bigg| > \epsilon \bigg) \leq \E \bigg\{\E \bigg(\epsilon \wedge \max_{j} \bigg| \frac{1}{\sqrt{n}} \sum_{i=1}^n W_{ij} V_{ij} \bigg| \; | V \bigg)\bigg\}.
\end{align*}
We will apply a symmetrisation argument to the inner conditional expectation. To this end, introduce $W'$ such that $W'$ and $W$ have the same distribution conditional on $V$ and such that $W' \independent W \given V$. In addition, let $S_1,\ldots,S_n$ be i.i.d.\ Rademacher random variables independent of all other quantities. The RHS of the last display is equal to 
\begin{align*}
& \E \bigg\{\E \bigg(\epsilon \wedge \max_{j} \bigg| \frac{1}{\sqrt{n}} \sum_{i=1}^n \big(W_{ij} - \E(W'_{ij} | W, V)\big)  V_{ij} \bigg| \; | V \bigg)\bigg\} \\
\leq & \E \bigg[ \E \bigg\{\epsilon \wedge \E \bigg( \max_{j} \bigg| \frac{1}{\sqrt{n}} \sum_{i=1}^n \big(W_{ij} - W'_{ij} )\big)  V_{ij} \bigg| \; | V, W \bigg) |V\bigg\}\bigg] \\
\leq &\E \bigg\{\epsilon \wedge \E \bigg( \max_{j} \bigg| \frac{1}{\sqrt{n}} \sum_{i=1}^n (W_{ij} - W'_{ij})  V_{ij} \bigg| \; | V \bigg)\bigg\} \\
=& \E \bigg\{\epsilon \wedge \E \bigg( \max_{j} \bigg| \frac{1}{\sqrt{n}} \sum_{i=1}^n S_i(W_{ij} - W'_{ij})  V_{ij} \bigg| \; | V \bigg)\bigg\} \\
& \leq 2 \E \bigg\{ \epsilon \wedge \E \bigg( \max_{j} \bigg| \frac{1}{\sqrt{n}} \sum_{i=1}^n S_iW_{ij}  V_{ij} \bigg| \; | V \bigg)\bigg\}
\end{align*}
using the triangle inequality in the final line. Now fixing a $\lambda \geq 0$, define $\tilde{W}_{ij} = W_{ij} \ind_{\{\|W\|_\infty \leq \lambda\}}$. 
Half the final expression in the display above is at most
\begin{align*}
 \E \bigg\{ \epsilon \wedge \E \bigg( \max_{j} \bigg| \frac{1}{\sqrt{n}} \sum_{i=1}^n S_i\tilde{W}_{ij}  V_{ij} \bigg| \; | V \bigg)\bigg\} + \epsilon \pr( \|W\|_\infty > \lambda).
\end{align*}
Note that $S_i\tilde{W}_{ij} \in [-\lambda, \lambda]$ and conditional on $V$, $\{S_i\tilde{W}_{ij}\}_{i=1}^n$ are independent. Thus conditional on $V$, $\sum_{i=1}^n S_i\tilde{W}_{ij}  V_{ij} / \sqrt{n}$ is sub-Gaussian with parameter $\lambda \left(\sum_{i=1}^n V_{ij}^2/n\right)^{1/2}$. Using a standard maximal inequality for sub-Gaussian random variables (see Theorem~2.5 in \citet{boucheron2013concentration}), we have
\begin{align*}
\E \bigg( \max_{j} \bigg| \frac{1}{\sqrt{n}} \sum_{i=1}^n S_i\tilde{W}_{ij}  V_{ij} \bigg| \; | V \bigg) \leq \lambda \sqrt{2\log(d)} \; \max_j\left(\sum_{i=1}^n V_{ij}^2/n\right)^{1/2}
\end{align*}
which then gives the result.
\end{proof}

\section{Proof of Theorem~\ref{THM:KERNEL}}
We will prove (ii) first. From Theorem~\ref{thm:univariate} and Remark~\ref{rem:univariate}, it is enough to show that
\begin{equation} \label{eq:MSPE_bd}
\sup_{P \in \mathcal{P}} \E_P\bigg( \frac{1}{n} \sum_{i=1}^n \{f_P(z_i) - \hat{f}(z_i)\}^2 \bigg) = o(\sqrt{n}),
\end{equation}
and an analogous result for $\hat{g}$.
We know from Lemma~\ref{lem:KRR_bd} that
\[
\frac{1}{n} \sum_{i=1}^n \E_P[\{f_P(z_i)-\hat{f}(z_i)\}^2 |\mb Z^{(n)}] \leq \max(\sigma^2, \|f_P\|_{\mathcal{H}}^2) \inf_{\lambda > 0} \bigg\{\frac{1}{n \lambda} \sum_{i=1}^n \min(\hat{\mu}_i/4, \lambda)  + \lambda / 4 \bigg\}.
\]
Lemma~\ref{lem:pop_evals} then gives us
\begin{align*}
\E_P \min_{\lambda > 0} \bigg\{\frac{1}{n\lambda } \sum_{i=1}^n \min(\hat{\mu}_i/4, \lambda)  + \lambda / 4 \bigg\} &\leq \inf_{\lambda > 0} \bigg\{ \frac{1}{n\lambda} \sum_{i=1}^n \E_P \min(\hat{\mu}_i/4, \lambda)  + \lambda / 4\bigg\} \\
&\leq \inf_{\lambda > 0} \bigg\{ \frac{C}{n\lambda} \sum_{j=1}^\infty \min(\mu_{P,j}, \lambda)  + \lambda \bigg\}
\end{align*}
for a constant $C>0$. Note that the first inequality in the last display allows us to effectively move from a fixed design with a design-dependent tuning parameter $\lambda$ to a random design but where $\lambda$ is fixed since the minimum is outside the expectation.
For $P \in \mathcal{P}$, let $\phi_P :[0, \infty) \to [0, \infty)$ be given by
\[
\phi_P(\lambda) = \sum_{j=1}^\infty \min(\mu_{P,j}, \lambda).
\]
Observe that $\phi_P$ is increasing and $\lim_{\lambda \downarrow 0} \sup_{P \in \mathcal{P}} \phi_P(\lambda) = 0$ by \eqref{eq:mu_cond}. Let $\lambda_{P,n} = n^{-1/2}\sqrt{\phi_P(n^{-1/2})}$ so $\sup_{P \in \mathcal{P}}\lambda_{P,n}=o(n^{-1/2})$. Thus for $n$ sufficiently large $\phi_P(\lambda_{P,n}) \leq \phi_P(n^{-1/2})$, whence for such $n$ we have
\begin{align*}
 \sup_{P \in \mathcal{P}} \inf_{\lambda >0} \{\phi_P(\lambda) / (n\lambda) + \lambda\} &\leq \sup_{P \in \mathcal{P}} \frac{\phi_P(\lambda_{P,n})}{ n\lambda_{P,n}} + \lambda_{P,n} \\
 &\leq \sup_{P \in \mathcal{P}} \sqrt{\phi_P(n^{-1/2})} / \sqrt{n}
 = o(n^{-1/2}).
\end{align*}
Putting things together gives \eqref{eq:MSPE_bd}.

To show (i), set $\mathcal{P}=\{P\}$ in the preceding argument and note that $\lim_{\lambda \downarrow 0} \phi_P(\lambda)=0$ by dominated convergence theorem using the summability of the eigenvalues i.e.\ \eqref{eq:mu_cond} always holds.
\subsection{Auxiliary Lemmas}
The following result gives a bound on the prediction error of kernel ridge regression with fixed design. The arguments are similar to those used in the analysis of regular ridge regression, see for example \citet{JMLR:v14:dhillon13a}.
\begin{lemma} \label{lem:KRR_bd}
Let $z_1,\ldots,z_n \in \mathcal{Z}$ (for some input space $\mathcal{Z}$) be deterministic and suppose
\[
x_i = f(z_i) + \varepsilon_i.
\]
Here $\var(\varepsilon_i) \leq \sigma^2$, $\cov(\varepsilon_i, \varepsilon_j)=0$ for $j \neq i$ and $f \in \mathcal{H}$ for some RKHS $(\mathcal{H}, \|\cdot\|_\mathcal{H})$ with reproducing kernel $k:\mathcal{Z} \times \mathcal{Z} \to \R$. Consider performing kernel ridge regression with tuning parameter $\lambda>0$:
\[
\hat{f}_\lambda = \argmin_{h \in \mathcal{H}} \bigg\{ \frac{1}{n}\sum_{i=1}^n \{x_i-h(z_i)\}^2 + \lambda \|f\|_{\mathcal{H}}^2\bigg\}.
\]
Let $K \in \R^{n \times n}$ have $ij$th entry $K_{ij} = k(z_i, z_j)/n$ and denote the eigenvalues of $K$ by $\hat{\mu}_1 \geq \hat{\mu}_2 \geq \cdots \geq \hat{\mu}_n \geq 0$. Then we have
\begin{align} \label{eq:KRR_bd}
\frac{1}{n} \E\bigg\{\sum_{i=1}^n \{f(z_i)-\hat{f}_{\lambda}(z_i)\}^2 \bigg\} &\leq \frac{\sigma^2}{n} \sum_{i=1}^n \frac{\hat{\mu}_i^2}{(\hat{\mu}_i + \lambda)^2} + \|f\|_{\mathcal{H}}^2 \frac{\lambda}{4} \\
&\leq \frac{\sigma^2}{\lambda}\frac{1}{n}\sum_{i=1}^n \min(\hat{\mu}_i/4, \lambda) +\|f\|_{\mathcal{H}}^2 \frac{\lambda}{4}. \notag 
\end{align}
\end{lemma}
\begin{proof}
Let $\mb X = (x_1,\ldots,x_n)^T$.
We know from the representer theorem  \citep{Kimeldorf1970, Schoelkopf2001} that
\[
\Big(\hat{f}_\lambda(z_1),\ldots,\hat{f}_\lambda(z_n)\Big)^T = K(K+\lambda I)^{-1} \mb X.
\]
We now show that
\[
\Big(f(z_1),\ldots,f(z_n)\Big)^T = K \alpha,
\]
for some $\alpha \in \R^n$, and moreover that $\|f\|_{\mathcal{H}}^2 \geq \alpha^T K\alpha/n$.

Let $V = \text{span}\{k( \cdot, z_1), \ldots, k(\cdot, z_n)\} \subseteq \mathcal{H}$ and write $f = u + v$ where $u \in V$ and $v \in V^{\perp}$. Then
\[
f(z_i)= \langle f, k( \cdot, z_i) \rangle = \langle u, k( \cdot, z_i) \rangle,
\]
where $\langle \cdot , \cdot \rangle$ denotes the inner product of $\mathcal{H}$.
Write $u=\sum_{i=1}^n \alpha_i k( \cdot, z_i)$. Then
\[
f(z_i) = \sum_{j=1}^n \alpha_j \langle k( \cdot, z_j), k( \cdot, z_i) \rangle = \sum_{j=1}^n \alpha_j k(z_j, z_i) = K_i^T \alpha,
\]
where $K_i$ is the $i$th column (or row) of $K$.
Thus $K\alpha = \Big(f(z_1),\ldots,f(z_n)\Big)^T$.
By Pythagoras' theorem
\[
\|f\|_{\mathcal{H}}^2 = \|u \|_{\mathcal{H}}^2 + \|v \|_{\mathcal{H}}^2 \geq \|u \|_{\mathcal{H}}^2 = \alpha ^T K \alpha/n.
\]

Now let the eigendecomposition of $K$ be given by $K=UDU^T$ with $D_{ii} = \hat{\mu}_i$ and define $\theta = U^TK\alpha$.
We see that $n$ times the left-hand side of \eqref{eq:KRR_bd} is
\begin{align*}
\E\|K(K+\lambda I)^{-1}(U\theta +\varepsilon) - U\theta\|_2^2 &= \E\|DU^T(UDU^T+\lambda I)^{-1}(U\theta +\varepsilon) - \theta\|_2^2 \\
&= \E\|D(D+\lambda I)^{-1}(\theta +U^T\varepsilon) - \theta\|_2^2 \\
&= \|\{D(D+\lambda I)^{-1}-I\}\theta\|_2^2 + \E\|D(D+\lambda I)^{-1}U^T\varepsilon\|_2^2.
\end{align*}
Let $\Sigma \in \R^{n \times n}$ be the diagonal matrix with $i$th diagonal entry $\var(\varepsilon_i) \leq \sigma^2$.
To compute the second term, we argue as follows.
\begin{align*}
\E\|D(D+\lambda I)^{-1}U^T\varepsilon\|_2^2 &= \E [\{D(D+\lambda I)^{-1}U^T\varepsilon\}^T D(D+\lambda I)^{-1}U^T\varepsilon] \\
&= \E [\tr\{D(D+\lambda I)^{-1}U^T\varepsilon \varepsilon^T U D(D+\lambda I)^{-1} \}] \\
&=  \tr\{D(D+\lambda I)^{-1}U^T \Sigma U D(D+\lambda I)^{-1} \} \\
&= \tr \{U D^2(D+\lambda I)^{-2} U^T \Sigma\} \\
&\leq \sigma^2\tr\{D^2(D+\lambda I)^{-2}\} \\
&= \sigma^2 \sum_{i=1}^n \frac{\hat{\mu}_i^2}{(\hat{\mu}_i + \lambda)^2}.
\end{align*}
For the first term, we have
\begin{align*}
\|\{D(D+\lambda I)^{-1}-I\}\theta\|_2^2 = \sum_{i=1}^n \frac{\lambda^2\theta_i^2}{(\hat{\mu}_i + \lambda)^2}.
\end{align*}
Now as $\theta=DU^T\alpha$ note that $\theta_i=0$ when $d_i=0$. Let $D^+$ be the diagonal matrix with $i$th diagonal entry equal to $D_{ii}^{-1}$ if $D_{ii} > 0$ and 0 otherwise.
Then
\[
\sum_{i:\hat{\mu}_i>0} \frac{\theta_i^2}{\hat{\mu}_i} =\|\sqrt{D^+}\theta\|_2^2 = \alpha^T K UD^+U^T K \alpha = \alpha^T UDD^+DU^T  \alpha =  \alpha^T K \alpha \leq n \|f\|_{\mathcal{H}}^2.
\]
Next
\begin{align*}
\sum_{i=1}^n \frac{\theta_i^2}{\hat{\mu}_i}\frac{\hat{\mu}_i\lambda^2}{(\hat{\mu}_i + \lambda)^2} &\leq n\|f\|_{\mathcal{H}}^2 \max_{i=1,\ldots,n} \frac{\hat{\mu}_i\lambda^2}{(\hat{\mu}_i + \lambda)^2} \leq \lambda n\|f\|_{\mathcal{H}}^2/4,
\end{align*}
using the inequality $(a + b)^2 \geq 4ab$ in the final line. Finally note that
\[
\frac{\hat{\mu}_i^2}{(\hat{\mu}_i + \lambda)^2} \leq \min\{1, \hat{\mu}_i^2/(4\hat{\mu}_i\lambda)\} = \min(\lambda, \hat{\mu}_i/4)/\lambda. 
\]
Putting things together gives the result.
\end{proof}

\begin{lemma} \label{lem:pop_evals}
	Consider the setup of Theorem~\ref{THM:KERNEL}. For all $r > 0$,
	\[
	\E_P\bigg( \frac{1}{n} \sum_{i=1}^n \min(\hat{\mu}_i, r) \bigg) \leq \frac{1}{n} \sum_{i=1}^\infty \min(\mu_{P,i}, r).
	\]
\end{lemma}
The proof below is not included in the published version of this paper, where instead some results in a book are cited in order to establish this result. However, it was subsequently discovered that the results in the book were incorrect. Below is a complete proof of the result.
\begin{proof}
	In the following, we drop the subscript $P$ for notational simplicity and write $\mathcal{Z}$ 
	for the space in which $Z$ takes its values.
	It suffices to show that given any $\epsilon > 0$, we have
\[
\E\bigg(\sum_{i=1}^n \min(\hat{\mu}_i , r) \bigg) \leq \epsilon + \sum_{i=1}^\infty \min(\mu_i, r).
\]
Now let $d$ be such that
\[
\sum_{i=d+1}^\infty \mu_i < \epsilon/n.
\]
Let 
$ \phi : \mathcal{Z} \to \R^d$
be given by $\phi(z) = \big(\sqrt{\mu_j} e_j(z) \big)_{j=1}^d$, and set
\[
\Phi := \begin{pmatrix} \phi(z_1)^T \\ \vdots \\ \phi(z_n)^T \end{pmatrix} \in \R^{n \times d}. 
\]
By the Cauchy--Schwarz inequality, for all $z$ and $z'$,
\[
\sum_{j=1}^\infty \mu_j|e_j(z) e_j(z')| \leq 
\sqrt{k(z,z)k(z',z')} < \infty.
\]
Thus by Fubuni's theorem, for all $v \in \R^n$ we have
\[
v^T(n K - \Phi\Phi^T)v = \sum_{l=d+1}^\infty \mu_l \left(\sum_{i=1}^n v_i e_l(z_i)\right) \left(\sum_{j=1}^n  v_j  e_l(z_j)\right) \geq 0,
\]
so $K - \Phi\Phi^T/n$ is positive semi-definite. 

Next let $\mathbb{S}^d_+$ be the cone of symmetric positive semi-definite $d \times d$ matrices, and for $A \in \mathbb{S}^d_+$ and $i=1,\ldots,d$, let $\lambda_i(A)$ denote the $i$th largest eigenvalue of $A$. Let $f : \mathbb{S}^d_+ \to \R$ be given by
\[
f(A) = \sum_{i=1}^d \min(\lambda_i(A), r).
\]
By Weyl's inequality, noting that the nonzero eigenvalues of $\Phi^T\Phi$ and $\Phi \Phi^T$ coincide, we have, for all $i$,
\[
\hat{\mu}_i \leq \lambda_i(\Phi^T\Phi/n) + \lambda_1(K - \Phi\Phi^T/n)
\]
and so
\[
\min(\hat{\mu}_i , r) \leq \min(\lambda_i(\Phi^T\Phi /n) , r) + \tr(n K - \Phi\Phi^T)/n.
\]
Thus
\begin{equation} \label{eq:eq1}
	\E\bigg(\sum_{i=1}^n \min(\hat{\mu}_i , r) \bigg) \leq \E f( \Phi^T\Phi/n) + \E \tr(n K - \Phi\Phi^T).
\end{equation}
Now by Fubini's theorem,
\[
\E \tr(n K - \Phi\Phi^T) = \sum_{i=1}^n \sum_{l=d+1}^\infty \mu_l \E e_l(z_i)^2 = n \sum_{l=d+1}^\infty \mu_l  < \epsilon.
\]
We now claim that $f$ is concave, from which the result will follow. Indeed, then by Jensen's inequality,
$\E f( \Phi^T\Phi/n) \leq f(\E \Phi^T\Phi/n)$ and
\[
\frac{1}{n}\big(\E \Phi^T\Phi\big)_{ij} = \sqrt{\mu_i\mu_j} \E e_j(Z) e_i(Z) = \mu_i \ind_{\{i=j\}}.
\]
Thus
\[
f(\E \Phi^T\Phi/n) = \sum_{i=1}^d \min(\mu_i,r),
\]
and so returning to \eqref{eq:eq1} we would have
\[
\E\bigg(\sum_{i=1}^n \min(\hat{\mu}_i , r) \bigg) \leq \epsilon + \sum_{i=1}^\infty \min(\mu_i,r).
\]

We now show that $f$ is concave. Take $t \in (0, 1)$ and $A,B \in\mathbb{S}^d_+$. We will show that
\begin{equation} \label{eq:concave}
	\sum_{i=1}^d (\lambda_i(tA + (1-t)B)-r)_+ \leq \sum_{i=1}^d \{t(\lambda_i(A)-r)_+ + (1-t)(\lambda_i(B)-r)_+ \},
\end{equation}
where $( \cdot)_+ :=\max(\cdot, 0)$ denotes the positive part. This will prove concavity of $f$ as
\begin{align*}
	\sum_{i=1}^d \lambda_i(tA +(1-t)B) &= \tr(tA + (1-t)B)  \\
	& = t \tr (A) + (1-t)\tr(B) = \sum_{i=1}^d \{t\lambda_i(A) + (1-t)\lambda_i(B) \}, 
\end{align*}
so subtracting \eqref{eq:concave} yields $f(tA + (1-t)B) \geq tf(A) +(1-t)f(B)$ as desired.

Certainly
\eqref{eq:concave} holds when $r \geq \lambda_1(tA + (1-t)B) $.
Now by Lidskii's inequality, for each $j=1,\ldots,d$,
\begin{equation} \label{eq:lidskii}
	\sum_{i=1}^j\lambda_i(tA + (1-t)B) \leq \sum_{i=1}^j \{t\lambda_i(A) + (1-t)\lambda_i(B) \}.
\end{equation}
For convenience, let us set $\lambda_{d+1}(tA + (1-t)B)=0$. Then for any $j=1,\ldots,d$, if
$\lambda_{j+1}(tA + (1-t)B) \leq r \leq \lambda_j(tA + (1-t)B) $, we have
\begin{align*}
	\sum_{i=1}^d(\lambda_i(tA + (1-t)B) - r)_+ &= \sum_{i=1}^j(\lambda_i(tA + (1-t)B) - r) \\
	&\leq \sum_{i=1}^j \{t(\lambda_i(A)-r) + (1-t)(\lambda_i(B)-r) \} \\
	&\leq \sum_{i=1}^d \{t(\lambda_i(A)-r)_+ + (1-t)(\lambda_i(B)-r)_+ \},
\end{align*}
using \eqref{eq:lidskii} for the first inequality. We thus have that \eqref{eq:concave} holds whatever the value of $r$, and so $f$ is concave, which completes the proof.
\end{proof}

\end{cbunit}

\end{document}